\documentclass[10pt,a4paper,twoside,reqno]{amsart}

\usepackage{amssymb}
\usepackage{mathtools}
\usepackage[shortlabels]{enumitem}
\usepackage[utf8]{inputenc}
\usepackage{color}
\usepackage{microtype,mathabx}
\usepackage{tikz}
\usetikzlibrary{positioning,decorations.pathreplacing,decorations.markings,arrows.meta,calc,fit,quotes,cd,math,arrows,backgrounds,shapes.geometric}
\usepackage[a4paper,top=3cm,bottom=3cm,inner=2.5cm,outer=2.5cm]{geometry}
\usepackage{graphicx}
\definecolor{darkgreen}{rgb}{0,0.45,0}
\usepackage{bm}
\usepackage[colorlinks,citecolor=darkgreen,urlcolor=gray,final,hyperindex,linktoc=page,hyperfootnotes=true]{hyperref}
\usepackage[arrow, matrix, tips, curve, graph, rotate]{xy}
\SelectTips{cm}{10}
\newcommand{\cd}[2][]{\vcenter{\hbox{\xymatrix#1{#2}}}}
\usepackage[capitalize]{cleveref}
\crefname{equation}{}{}
\crefname{lem}{Lemma}{Lemmas}
\crefname{thm}{Theorem}{Theorems}
\crefname{defi}{Definition}{Definitions}
\crefname{conj}{Conjecture}{Conjectures}
\crefname{ex}{Example}{Examples}
\crefname{sec}{Section}{Sections}
\crefname{prop}{Proposition}{Propositions}

\def\slashedrightarrow{\relbar\joinrel\mapstochar\joinrel\rightarrow}

\newcommand{\horightarrow}{\slashedrightarrow}

\theoremstyle{plain}
\newtheorem{thm}{Theorem}[section]
\newtheorem{cor}[thm]{Corollary}
\newtheorem{lem}[thm]{Lemma}
\newtheorem{prop}[thm]{Proposition}

\theoremstyle{remark}
\newtheorem{rmk}[thm]{Remark}
\newtheorem{ex}[thm]{Example}

\theoremstyle{definition}
\newtheorem{defi}[thm]{Definition}

\definecolor{mypurple}{rgb}{0.5, 0.0, 0.5}

\newcommand{\co}{\colon}
\newcommand{\ca}{\mathcal}
\newcommand{\dc}{\mathbb}
\newcommand{\nc}{\mathsf}
\newcommand{\cat}{\nc}
\newcommand{\thg}{{\mathord{\text{--}}}}

\newcommand{\wc}{\widecheck}
\newcommand{\wh}{\widehat}

\newcommand{\one}{\mathbf{1}}
\newcommand{\mi}{\textrm{-}}
\newcommand{\ot}{\otimes}
\newcommand{\Set}{\nc{Set}}
\newcommand{\Cat}{\nc{Cat}}
\newcommand{\VCat}{\nc{Cat}_\ca{V}}
\newcommand{\VSym}{\dc{S}\nc{ym}_\ca{V}}
\newcommand{\Sym}{\dc{S}\nc{ym}}
\newcommand{\VCatSym}{\dc{C}\nc{atSym}_\ca{V}}
\newcommand{\CatSym}{\dc{C}\nc{atSym}}

\newcommand{\VProf}{\dc{P}\nc{rof}_\ca{V}}
\newcommand{\VMat}{\dc{M}\nc{at}_\ca{V}}

\newcommand{\Prof}{\dc{P}\nc{rof}}
\newcommand{\Mat}{\dc{M}\nc{at}}

\newcommand{\op}{\mathrm{op}}
\newcommand{\id}{\mathrm{id}}

\tikzset{tick/.style={postaction={decorate,decoration={markings,mark=at
position 0.5 with {\draw[-] (0,.4ex) -- (0,-.4ex);}}}}}
\tikzset{bigtick/.style={postaction={decorate,decoration={markings,mark=at
position 0.5 with {\draw[-] (0,.6ex) -- (0,-.6ex);}}}}}
\newcommand{\tickar}{\begin{tikzcd}[baseline=-0.5ex,cramped,sep=small,ampersand
replacement=\&]{}\ar[r,tick]\&{}\end{tikzcd}}
\newcommand{\ticktwoar}{\begin{tikzcd}[baseline=-0.5ex,cramped,sep=small,
ampersand replacement=\&]{}\ar[r,Rightarrow,bigtick]\&{}\end{tikzcd}}
\newcommand{\tickarlong}{\begin{tikzcd}[baseline=-0.5ex,cramped,ampersand
replacement=\&]{}\ar[r,tick]\&{}\end{tikzcd}}
\newcommand{\xtickar}[1]{\stackrel{#1}{\tickarlong}}
\newcommand{\sbul}{\scriptstyle\bullet}
\tikzset{bul/.style={postaction={decoration={markings,mark=at position 0.5 with
{\node{$\sbul$};}},decorate}}}

\tikzset{Rightarrow/.style={double equal sign distance,>={Implies},->},
triple/.style={-,preaction={draw,Rightarrow}}}

\newcommand{\tor}{\ensuremath{\relbar\joinrel\mapstochar\joinrel\rightarrow}}

\newcommand{\comp}[1]{\wh{#1}}
\newcommand{\coj}[1]{\wc{#1}}
\newcommand{\kot}{\boxtimes}

\newcommand{\vid}{1}
\newcommand{\hid}{\mathrm{id}}
\newcommand{\Two}{\scriptstyle\Downarrow}
\newcommand{\Ddownarrow}{\rotatebox[origin=c]{-90}{$\Rrightarrow$}}
\newcommand{\Kl}{{\dc{K}\nc{l}}}

\newcommand{\twocell}[9]{\begin{tikzcd}[ampersand replacement=\&]
#1\ar[r,tick,"{#2}"]\ar[d,"{#8}"']\ar[dr,phantom,"\Two{#9}"] \& #3\ar[d,"{#4}"]
\\
#7\ar[r,tick,"{#6}"'] \& #5
\end{tikzcd}}

\newcommand{\twocong}[2][0.5]{\ar@{}[#2] \save ?(#1)*{\cong}\restore}
\newcommand{\twoeq}[2][0.5]{\ar@{}[#2] \save ?(#1)*{=}\restore}
\newcommand{\rtwocell}[3][0.5]{\ar@{}[#2] \ar@{=>}?(#1)+/l 0.2cm/;?(#1)+/r
0.2cm/^{#3}}
\newcommand{\ltwocell}[3][0.5]{\ar@{}[#2] \ar@{=>}?(#1)+/r 0.2cm/;?(#1)+/l
0.2cm/^{#3}}
\newcommand{\ltwocello}[3][0.5]{\ar@{}[#2] \ar@{=>}?(#1)+/r 0.2cm/;?(#1)+/l
0.2cm/_{#3}}
\newcommand{\dtwocell}[3][0.5]{\ar@{}[#2]|{\Two {#3}}}
\newcommand{\dltwocell}[3][0.5]{\ar@{}[#2] \ar@{=>}?(#1)+/ur  0.2cm/;?(#1)+/dl
0.2cm/^{#3}}
\newcommand{\drtwocell}[3][0.5]{\ar@{}[#2] \ar@{=>}?(#1)+/ul  0.2cm/;?(#1)+/dr
0.2cm/^{#3}}
\newcommand{\dthreecell}[3][0.5]{\ar@{}[#2] \ar@3{->}?(#1)+/u  0.2cm/;?(#1)+/d
0.2cm/^{#3}}
\newcommand{\utwocell}[3][0.5]{\ar@{}[#2] \ar@{=>}?(#1)+/d 0.2cm/;?(#1)+/u
0.2cm/_{#3}}
\newcommand{\dtwocelltarg}[3][0.5]{\ar@{}#2 \ar@{=>}?(#1)+/u  0.2cm/;?(#1)+/d
0.2cm/^{#3}}
\newcommand{\utwocelltarg}[3][0.5]{\ar@{}#2 \ar@{=>}?(#1)+/d  0.2cm/;?(#1)+/u
0.2cm/_{#3}}

\newcommand{\defeq}{\mathrel{\mathop:}=}

\newcommand{\freesmc}{S}

\newcommand{\sym}{\mathfrak{S}}

\begin{document}
\title[Monoidal Kleisli bicategories]{
Monoidal Kleisli bicategories and \\ the  arithmetic product of coloured symmetric
sequences}

\author[N. Gambino]{Nicola Gambino}
\address{Department of Mathematics, University of Manchester}
\email{nicola.gambino@manchester.ac.uk}

\author[R. Garner]{Richard Garner}
\address{School of Mathematical and Physical Sciences, Macquarie University}
\email{richard.garner@mq.edu.au}

\author[C. Vasilakopoulou]{Christina Vasilakopoulou}
\address{School of Applied Mathematical and Physical Sciences, National Technical University of Athens}
\email{cvasilak@math.ntua.gr}

\keywords{Kleisli bicategory, double category, monoidal structure, symmetric sequence, species of structures}.

\subjclass[2020]{18N10, 18N15, 18M80, 18C20, 18M05}

\begin{abstract} We extend the arithmetic product of species of structures and symmetric
sequences studied by Maia and M\'endez and by Dwyer and Hess
 to coloured symmetric sequences and show that it determines
a normal oplax monoidal structure on the bicategory of coloured symmetric sequences.
In order to do this, we establish general results on extending monoidal structures to Kleisli
bicategories. Our approach uses monoidal double categories, which help us to attack the difficult problem of verifying the coherence conditions for a monoidal bicategory in an efficient way.
\end{abstract}

\date{\today}

\maketitle

\setcounter{tocdepth}{1}
\tableofcontents

\section{Introduction}

\subsection*{Context, aim and motivation}
Joyal's theory of species of structures~\cite{JoyalA:thecsf} provides an illuminating and powerful approach to enumerative combinatorics, as amply illustrated in~\cite{BergeronF:comstls}, and finds applications also
in algebra~\cite{Species}.
 By definition, a species of structures $F$ is simply a functor from the category~$\mathfrak{B}$ of finite sets and bijections to the category of sets and functions, mapping a  finite set $U$ of `labels' to a set~$F[U]$ of `$F$-structures' ({\em e.g.}~binary rooted trees) labelled by elements of $U$.
Importantly for applications, species of structures support a calculus of operations (which includes substitution, sum, product and differentiation) that has a combinatorial interpretation and provides a `categorification' of the calculus of exponential power series widely used in combinatorics~\cite{WilfH;generating}. This point of view is supported by the introduction of
the so-called analytic functor associated to a species of structures~\cite{FoncteursAnalytiques}, which is defined by the formula
\[
F(X) = \sum_{n \in \mathbb{N}} \frac{F[n] \times X^n}{\sym_n} \mathrlap{,}
\]
where, for $n \in \mathbb{N}$, $[n] = \{ 1, \ldots, n \}$ and the fraction denotes the quotient of $F[n] \times X^n$ by
the evident action of the $n$-th symmetric group $\sym_n$.
The passage from species of structures to analytic functors goes via
symmetric sequences, which are defined as functors from $\sym$, the skeleton of $\mathfrak{B}$ whose objects are finite cardinals, into the category of sets.
Under the equivalence between species of structures and symmetric sequences, the substitution operation of species of structures corresponds to the substitution monoidal structure on symmetric sequences defined
in~\cite{KellyGM:opejpm}, which is of interest since monoids with respect to it are precisely symmetric operads~\cite{BoardmanJ:homias,Mayoperad}.

In~\cite{MaiaM:aripcs}, Maia and M\'endez introduced a new operation on species of structures, baptized \emph{arithmetic product}, and provided a combinatorial interpretation for it. Given species $F$ and $G$, their arithmetic product $F \kot G$ is defined by letting
\[
(F_1 \kot F_2)[U] = \sum_{(\pi_1, \pi_2) \in \mathcal{R}[U]} F_1[\pi_1] \times F_2[\pi_2] \mathrlap{,}
\]
where $\mathcal{R}[U]$ denotes a certain set of  partitions of $U$, called rectangles.
Independently of the work of Maia and M\'endez,  Dwyer and Hess rediscovered
this operation\footnote{Dwyer and Hess called it matrix multiplication. Here we prefer to say arithmetic product in order to
avoid potential confusion with the composition operation of the bicategory of matrices.} in the context of symmetric sequences~\cite{DwyerW:BoardmanVtpo}
in order to extend the Boardman--Vogt tensor product of symmetric operads~\cite{BoardmanJ:homias} to operadic bimodules. For symmetric sequences, the
arithmetic product is defined from the product of natural numbers on $\sym$ (which is functorial, even if it is not the cartesian product)  by Day
convolution~\cite{DayBJ:clocf}, via the coend formula
\begin{equation}
\label{equ:matrix-mult-dwyer-hess}
(F_1 \kot F_2)[m] = \int^{m_1, m_2} \sym[m,m_1 \cdot m_2] \times F_1[m_1] \times F_2[m_2]   \mathrlap{.}
\end{equation}
The connection between the arithmetic product for species in~\cite{MaiaM:aripcs} and symmetric sequences in~\cite{DwyerW:BoardmanVtpo} seems to
have been first noted by Bremner and Dotsenko in~\cite{BremnerM:boavtp}.

In~\cite{DwyerW:BoardmanVtpo}, the authors also observed that the arithmetic product of symmetric sequences interacts with the substitution monoidal
structure in an interesting way, in that there is a natural transformation with components
\begin{equation}
\label{equ:intro-interchange}
(G_1 \circ F_1) \kot (G_2 \circ F_2) \to (G_1 \kot G_2) \circ (F_1 \kot F_2)
\end{equation}
which are not necessarily invertible.
Dwyer and Hess conjectured that this transformation underlies what is usually called a \emph{duoidal} or \emph{2-monoidal} structure, in which two
monoidal structures interact by means of an interchange law~\cite{Species}. The conjecture was settled positively by the second-named author and
L\'opez\ Franco, who also showed that this duoidal structure is {\em normal}, in the sense that the units of the two monoidal structures essentially
coincide. This was done as part of their general study of commutative operations~\cite{GarnerLopezFranco}, which involves introducing a general notion
of \emph{commuting tensor product} of $\circ$-monoids in a  normal duoidal category $(\mathcal{V}, \mathord{\kot},\circ)$. When this notion is
instantiated at the normal duoidal category of symmetric sequences, it re-finds the  Boardman--Vogt tensor product  $P \otimes_{\mathrm{BV}} Q$ of
symmetric operads $P$ and $Q$. In particular, to express that the operations of $P$ and $Q$ \emph{commute} with each other in $P \otimes_{\mathrm{BV}}
Q$ one uses a diagram of the form
\begin{equation}
\label{equ:bv}
\begin{tikzcd}
P \kot Q \ar[r, shift left=.75ex] \ar[r, shift right=.75ex] & (P \otimes_{\mathrm{BV}} Q) \circ (P \otimes_{\mathrm{BV}} Q) \ar[r] & P \otimes_{\mathrm{BV}} Q
\end{tikzcd}
\end{equation}
which involves both the arithmetic product and the substitution monoidal structures.
Furthermore, \cite{Duoidalmeasuring} uses this duoidal structure to establish an enrichment of symmetric operads in symmetric cooperads.

The aim of this paper is to generalise the definition of the arithmetic product, and the key results concerning it, from symmetric sequences to
coloured symmetric sequences: this corresponds to the passage from symmetric operads to coloured symmetric operads, {\em i.e.}~from the
single-object to the many-object case. We will show that such generalisation not only is possible, but actually
determines a new kind of low-dimensional categorical structure.

The motivation for this work is manifold. First, it is part of a wider research programme aimed  at understanding the structure of the bicategory of coloured symmetric sequences and related bicategories, with applications to logic and theoretical computer science, \emph{cf.}~\cite{cattani2005profunctors,hyland2010some}. In particular, it provides the basis to extend the Garner--L\'opez\ Franco theory of commutativity, to re-find the Boardman--Vogt tensor product of coloured symmetric operads (\emph{cf.} \eqref{equ:bv}), and to develop a corresponding tensor product of bimodules between them, generalising the results of Dwyer and Hess, a project that we leave for future work.  The results presented here are also useful to extend the study of enrichment in \cite{VCocats} to relate coloured (co)operads and their (co)modules. Finally, we hope that
our results may eventually be of interest in combinatorics, since the arithmetic product of coloured symmetric sequences defined here
induces a corresponding operation on coloured species of structures~\cite{FioreM:carcbg}, which extends the arithmetic product of Maia and M\'endez to variants of Joyal's species of structures that are particular instances of coloured species of structures.

\subsection*{Main results}
Whereas the original arithmetic product is extra structure on the substitution monoidal category of symmetric sequences, our generalised arithmetic
product for coloured symmetric
sequences will be extra structure on the bicategory~$\mathsf{Sym}$ of coloured symmetric sequences
of~\cite{FioreM:carcbg}.
Recall that, for sets $X$ and $Y$, an $(X,Y)$-coloured symmetric sequence is a functor $M \colon \freesmc Y^{\op} \times X \to \Set$,
where $\freesmc Y$ is the free symmetric strict monoidal category on $Y$.
 Such a functor $M$ assigns a set $M(\vec y, x)$ to each $\vec y = (y_1, \ldots, y_n) \in \freesmc Y$ and $x \in X$, the elements of which can be
thought of as operations $f \co y_1, \ldots, y_n \to x$, with inputs of sorts $y_1, \ldots, y_n$ and output of sort $x$, typically pictured as
corollas. Taking $X= Y = 1$ recovers the notion of ordinary symmetric sequence since $S 1 = \sym$; and as shown
in~\cite{FioreM:matmccs,FioreM:carcbg}, the calculus of symmetric sequences can be extended to their coloured counterparts. In particular, the
substitution monoidal structure can be generalised to a composition operation, which is the composition of the bicategory $\mathsf{Sym}$ whose objects
are sets, and whose maps from $X$ to $Y$ are the symmetric $(X,Y)$-coloured  sequences~\cite{FioreM:carcbg,FioreM:relpkb}. The monads in this
bicategory are then symmetric  coloured operads~\cite{BaezJ:higdan}.

Given an
$(X_1,Y_1)$-coloured symmetric sequence
$M$ and a $(X_2, Y_2)$-coloured symmetric sequence $M_2$, we will define their arithmetic product as the $(X_1 \times X_2, Y_1 \times Y_2)$-coloured symmetric sequence $M_1 \kot M_2$ given by
\begin{equation}
\label{equ:coloured-matrix-multiplication}
(M_1 \kot M_2)(\vec{y}, (x_1,x_2)) =
\int^{\vec y^1, \vec y^2 \in SY} \freesmc(Y_1 \times Y_2)[\vec y, \vec y_1 \kot \vec y_2] \times M_1(\vec y_1, x_1) \times M_2(\vec y_2, x_2) 
  \mathrlap{,}
\end{equation}
where
\begin{equation}
\label{equ:monoidal-for-freesmc}
\kot \co \freesmc Y_1 \times \freesmc Y_2 \to \freesmc(Y_1 \times Y_2)
\end{equation}
is an operation determined by lexicographic ordering of pairs.
As expected, when~$X_1 = Y_1 = X_2 = Y_2 = 1$, we obtain the arithmetic product in~\eqref{equ:matrix-mult-dwyer-hess} of Dwyer and Hess.

Our main result, \cref{thm:main-app-2}, asserts that the said arithmetic product determines a \emph{normal oplax monoidal structure} on the bicategory
$\mathsf{Sym}$ of coloured symmetric sequences. The notion of a normal oplax monoidal bicategory appears to be novel and is introduced here
in~\cref{def:oplaxbicat} as the natural `many-object' generalisation of the normal duoidal structure in~\eqref{equ:intro-interchange}.
The key challenge to be overcome to obtain our main result is the verification of the axioms for a normal oplax monoidal bicategory, which are of the same daunting complexity as those for a monoidal bicategory~\cite{CoherenceTricats,Gurskitricats}. As such, attempting a direct verification seems hopelessly
complicated and unlikely to result in any insight. Instead, we solve the problem developing ideas of 2-dimensional monad theory~\cite{2-dimmonadtheory}, obtaining some general results that are of independent interest.

Our approach exploits crucially the notion of a (pseudo) double category, which adeptly handles the bookwork around dealing with structures involving two kinds
of morphisms. More specifically, in a double category one has objects, two kinds of 1-cells (called horizontal and vertical), and squares (which help
to relate horizontal and vertical 1-cells).
Of key importance for our development are the double category of profunctors $\Prof$, which has categories as objects,
profunctors~\cite{BenabouJ:dis,BenabouJ:distw} (also known as distributors or bimodules) as horizontal 1-cells and
functors as vertical 1-cells; and the double category of matrices $\Mat$, which is the full double subcategory of $\Prof$ spanned by sets (viewed as
discrete categories). 

Double categories are important for us because they provide an efficient way of constructing three-dimensional structures such as monoidal
bicategories and, as we shall see, their oplax variants. The basic insight, as explained
in~\cite{GarnerR:lowdsf,ConstrSymMonBicats,ConstrSymMonBicatsFun}, is that, in order to obtain a monoidal structure on a bicategory $\mathcal{E}$, it
is sufficient to represent $\mathcal{E}$ as the horizontal bicategory of a double category and then construct a monoidal structure on
this double category. Since the coherence data and axioms for a monoidal double category are of the same character as those of a monoidal category, rather than a monoidal bicategory, this significantly reduces the volume and complexity of the checks required to establish the structure.

The relevance of this to our situation is that the bicategory of coloured symmetric sequences $\mathsf{Sym}$ can be represented as the horizontal bicategory
of a double category $\Sym$ in which vertical 1-cells are functions between sets. In fact, building on~\cite{CruttwellShulman,FioreM:relpkb}, this
$\Sym$ can be seen as a full sub-double category (on the discrete objects) of the double
category $\CatSym$ of \emph{categorical symmetric sequences}; which can, in turn, be constructed as
the Kleisli double category of a double monad on the double category $\Prof$ of profunctors.
The double monad in question maps a category $X$ to its symmetric strict monoidal completion~$\freesmc X$, extending the corresponding $2$-monad on the $2$-category of categories.

Given the above, the desired normal oplax monoidal structure on the bicategory of coloured symmetric sequences can be obtained as follows. Firstly
(\cref{thm:oplaxmonoidalbicat}), we extend the results of~\cite{GarnerR:lowdsf,ConstrSymMonBicats,ConstrSymMonBicatsFun} to establish that such a structure can be obtained from a normal oplax monoidal structure on $\Sym$, or more generally, $\CatSym$.
Secondly, to obtain this, we prove and apply a result (\cref{thm:oplaxmonoidalKlT}) which isolates sufficient conditions on a double monad $T$ on
a double category $\mathbb{C}$ under which a monoidal structure on $\mathbb{C}$ will extend to an oplax monoidal structure on the Kleisli double category $\Kl(T)$. Pleasingly, this condition on $T$ turns out to be a natural one, namely a suitably adapted form of the \emph{pseudo-commutativity} of~\cite{HylandPower}. This is satisfied by the 2-monad $\freesmc$ used in our application, and indeed, the operation~\eqref{equ:monoidal-for-freesmc} featured in the definition of the arithmetic product is part of this pseudo-commutative structure on $\freesmc \co \Cat \to \Cat$.

Thus, using this result, the monoidal structure on $\Prof$ given by the cartesian product of sets extends to give the arithmetic product oplax
monoidal
structure on the double category $\CatSym$ of categorical symmetric sequences, which in turn induces the desired
oplax monoidal structure on the double category $\Sym$ of coloured symmetric sequences. This general method
leads exactly to the
formula in~\eqref{equ:coloured-matrix-multiplication}, which is a natural generalisation of that in~\eqref{equ:matrix-mult-dwyer-hess}.

While our approach offers a clear pathway to prove our main results, and others besides, we still have to overcome significant technical challenges, dealing with coherence conditions at the double categorical level, keeping track of strictness and weakness of the structures involved. Roughly speaking, vertical structure tends to be stricter than horizontal one, but the two are closely related under the assumption that the double categories under consideration are fibrant, in the sense of~\cite{Framedbicats}. This allows us to induce a lot of the horizontal, weaker, structure that we need for applications from vertical, stricter, one, that is already known, thereby keeping some control of the complexity of our calculations.

While it undoubtedly requires more groundwork to set up the abstract approach that we take, the end result is a modular framework which is easily 
applicable to other, related situations. For example, although we shall not do so here, it is entirely straightforward to adapt our results from the 
setting appropriate for studying symmetric coloured operads to the setting appropriate for many-sorted algebraic theories: it is simply a matter of 
replacing the double monad $S$ for symmetric strict monoidal categories with a corresponding double monad $F$ for categories with strictly 
associative finite products, and verifying that everything still carries through. In this setting, we obtain an oplax monoidal structure on the 
appropriate Kleisli bicategory---which is essentially the bicategory of sifted-cocontinuous functors between presheaf categories---which extends the 
duoidal structure on the functor category $[F(1)^\mathrm{op}, \cat{Set}]$ used in~\cite{GarnerLopezFranco} to study the commuting tensor product of 
single-sorted algebraic theories.

For expository convenience, we outlined our results above for $\Set$-valued coloured symmetric sequences, but in fact our development will be carried out in a more general enriched context. For any symmetric monoidal closed cocomplete $\ca{V}$ there is an analogue of the $2$-monad $\freesmc$ leading to a bicategory of $\ca{V}$-enriched symmetric sequences. However, to obtain an oplax monoidal structure on this bicategory, we will restrict to the case where the tensor product of $\ca{V}$ is in fact \emph{cartesian} product. The reason for this restriction, which was also made in~\cite{DwyerW:BoardmanVtpo,GarnerLopezFranco}, is
that the $2$-monad $\freesmc$ is only pseudo-commutative when $\ca{V}$ is cartesian,  since the structure maps in~\eqref{equ:monoidal-for-freesmc} for this pseudo-commutativity involve a `duplication' of objects that is not available in
the general symmetric monoidal closed setting.

\subsection*{Outline of the paper}
\cref{sec:dbl-cat,sec:dbl-fun} recall the notions of double category, double functor, horizontal and vertical transformation, and
modification. We pay particular attention to companions, leading to the notion of a special vertical transformation, and establish a few useful lemmas
about them. \cref{sec:mon-dbl-cat,sec:mon-dbl-fun} focus on monoidal double categories, monoidal double functors, monoidal horizontal and vertical
transformations and monoidal modifications. In particular, we show that, for monoidal double categories $\dc{C}$ and $\dc{D}$,
monoidal double functors between them are the objects of a functor double category (\cref{thm:functor-dblcat-monoidal}).
\cref{sec:monoids} considers monoids in monoidal double categories. We use this notion in \cref{sec:dbl-monad,sec:mon-dbl-monad} to define double monads and monoidal double monads and obtain results on them in a homogeneous manner. To do this, we show that the double category of monoidal double endofunctors
on a monoidal double category admits a monoidal structure, given by composition (\cref{thm:ps-hom-c-c-monoidal}).
In \cref{sec:kleisli}, we consider Kleisli double categories and establish sufficient conditions for a double monad on a monoidal double
category to determine a monoidal structure on the Kleisli double category. We apply these results to coloured symmetric sequences in \cref{sec:application}, leading up to our main results on the existence of oplax monoidal structures on the relevant double category
(\cref{thm:double-catsym-oplax-monoidal}) and bicategory (\cref{thm:main-app-2}).

\subsection*{Acknowledgements} We are grateful to Mike Shulman for helpful conversations
and to Thomas Ehrhard for pointing us to~\cite{Premonoidal}, which led us to formulate the
notion of centrality in~\cref{def:central}. We also thank the referee for the careful reading of the paper and the helpful suggestions.

Gambino acknowledges that this
material is based upon work supported by the US Air Force Office for
Scientific Research under award number FA9550-21-1-0007 and by EPSRC via
grant EP/V002325/2. Garner acknowledges the support of Australian
Research Council grants FT160100393 and DP190102432. Vasilakopoulou acknowledges the support of the General Secretariat of Research and
Innovation (GSRI) and the Hellenic Foundation for Research and Innovation (HFRI).

\section{Double categories}
\label{sec:dbl-cat}

By \emph{double category}, we will mean a (horizontally) weak double category, also known as a pseudo double category; relevant material can be found in \cite{Limitsindoublecats, Adjointfordoublecats, Garner2006Double, ConstrSymMonBicatsFun}. 

\begin{defi}[Double category]
A \emph{double category} $\dc{C}$ consists of the following data:
\begin{itemize}
\item a category $\dc{C}_0$, whose objects are \emph{$0$-cells}, and whose arrows are \emph{vertical 1-cells} $f \co X \to X'$;
\item a category $\dc{C}_1$, whose objects are \emph{horizontal 1-cells} $M\colon X\horightarrow Y$, and whose arrows are \emph{2-morphisms}
\begin{equation}\label{eq:2map}
\begin{tikzcd}X\ar[d,"f"']\ar[r,tick,"M"]\ar[dr,phantom,"\Two\phi"] &
Y\ar[d,"g"] \\  X'\ar[r,tick,"M'"'] & Y'\mathrlap{ ;}\end{tikzcd}
\end{equation}
\item two functors $\mathbf{s},\mathbf{t}\colon\dc{C}_1\rightarrow \dc{C}_0$ called \emph{source} and \emph{target} respectively;
\item \emph{composition} and \emph{identity} functors $\circ\colon\dc{C}_1\times_{\dc{C}_0}\dc{C}_1\to\dc{C}_1$ and
$\hid\colon\dc{C}_0\to\dc{C}_1$;
\item natural families of globular isomorphisms in $\mathbb{C}_1$:
\begin{equation}\label{eq:aellr}
a_{M, N, P} \co (M\circ N)\circ P \Rightarrow M\circ(N\circ  P) \, , \quad
\ell_M \co \hid_{Y}\circ M \Rightarrow M \quad\text{and} \quad
r_M \co M\circ\hid_{X}  \Rightarrow M \mathrlap{.}
\end{equation}
\end{itemize}
These data are required to satisfy
coherence axioms analogous to those for a bicategory; see~\cite[\S 7.1]{Limitsindoublecats}.
\end{defi}

In the last item of the preceding definition, we use the notion of a \emph{globular} 2-morphism in a double category~$\dc{C}$; this is a
2-morphism $\phi$
for which $\mathbf{s}(\phi)$ and $\mathbf{t}(\phi)$ are identities:
\[
\begin{tikzcd}X\ar[d,equal]\ar[r,tick,"M"]\ar[dr,phantom,"\Two\phi"] &
X \ar[d,equal] \\  X'\ar[r,tick,"M'"'] & Y\mathrlap{ .} \end{tikzcd}
\]

\begin{rmk}
  \label{rk:directions}
  In our definition of double category, vertical $1$-cells compose categorically, i.e., strictly associatively, while horizontal $1$-cells compose bicategorically, i.e., with associativity only up to coherent isomorphism. Compared to the original definition in~\cite{Limitsindoublecats, Adjointfordoublecats}, we have chosen to interchange the vertical and horizontal directions so as to match up with later work such as~\cite{Garner2006Double, ConstrSymMonBicatsFun}. The reader should bear this reversal of sense in mind when comparing the definitions and constructions that follow with those of~\cite{Limitsindoublecats, Adjointfordoublecats}.
\end{rmk}
Given a double category $\dc{C}$, we write $\ca{H}(\dc{C})$ for the its \emph{horizontal
bicategory}, comprising the 0-cells, horizontal 1-cells and globular 2-morphisms; and we write $\nc{V}(\dc{C})$  for
its \emph{vertical 2-category}, whose objects and morphisms are the  0-cells and vertical 1-cells of $\dc{C}$, and where a $2$-cell $f \Rightarrow g$ is a 2-morphism in $\dc{C}$ of
the form
\[
\begin{tikzcd}X\ar[d,"g"']\ar[dr,phantom,"\Two\phi"]\ar[r,tick,"\hid_X"] &
X\ar[d,"f"] \\
X'\ar[r,tick,"\hid_{X'}"'] & X'\mathrlap{ .}\end{tikzcd}
\]
Notice that vertical composition of 2-cells in $\nc{V}(\dc{C})$ is given by pasting in the horizontal direction in $\dc{C}$, and vice versa.

\begin{ex}[Bicategories and monoidal categories as double
  categories] \label{ex:mon} Any bicategory can be seen as a double
  category with only identity vertical arrows. In particular, any
  monoidal category $(\ca{V}, \circ, J)$ can be regarded as a double
  category with a single object and only the identity vertical arrow.
\end{ex}

\begin{ex}[The double category of matrices] \label{ex:mat}
Fix a monoidal category~$(\ca{V}, \circ, J)$ with small coproducts, such that the tensor product
preserves coproducts in each variable; in particular, this holds whenever the monoidal structure is closed. Given sets
$X$ and $Y$, an $(X,Y)$-\emph{matrix}  $M\colon X\horightarrow Y$ is a family of objects $\bigl(M(y,x) \in \ca{V} : x \in X, y\in Y\bigr)$.
The double category of \emph{$\ca{V}$-matrices} $\VMat$ has objects and vertical $1$-cells being sets and functions, respectively;
horizontal 1-cells $M \colon X\horightarrow Y$
being $(X,Y)$-matrices; and 2-morphisms $\phi \colon M \Rightarrow N$ as in~\eqref{eq:2map} being families of $\ca{V}$-arrows $\bigl(\phi_{yx}\colon
M(y,x) \to N(gy,fx) : x \in X, y \in Y\bigr)$. Horizontal composition $N \circ M\colon X\horightarrow Y\horightarrow Z$ is given
by
\begin{displaymath}
 (N\circ M)(z,x) =\sum_{y\in Y} N(z,y) \circ M(y,x)\mathrlap{ ,}
\end{displaymath}
while the horizontal identities $\hid_X\colon X\horightarrow X$ are defined by $\hid_X(x,x')=J$ if
$x=x'$ and $\hid_X(x,x') = 0$ if~$x\neq x'$. The horizontal bicategory of this double category is the usual bicategory of enriched matrices; see~\cite{Varthrenr}.
\end{ex}

The choice of a different notation for \cref{ex:prof} below is needed in view of \cref{ex:mat-mon}.

\begin{ex}[The double category of profunctors]  \label{ex:prof} Fix a braided monoidal cocomplete category~$(\ca{V}, \otimes, I)$ in which
the tensor product preserves colimits in each variable. Below, we use freely the standard notions of~$\ca{V}$-category, $\ca{V}$-functor and
$\ca{V}$-natural transformation, for which we invite readers to refer to~\cite{Kelly}.
Recall that, for small $\ca{V}$-categories~$X$ and~$Y$, a $(Y,X)$-\emph{profunctor}  $M\colon X\horightarrow Y$ is a $\ca{V}$-functor $M\colon  Y^\op\otimes X \to \ca{V}$. We may now define the double category $\VProf$ of \emph{$\ca{V}$-profunctors} as follows. The objects and vertical $1$-cells of $\VProf$ are
small $\ca{V}$-categories and $\ca{V}$-functors; the horizontal
$1$-cells are $\ca{V}$-profunctors, and squares as in~\eqref{eq:2map} (writing now~$F$ and~$G$ instead of $f$ and $g$)
are $\ca{V}$-natural transformations~$\phi \co M \Rightarrow M' \, (G \otimes F)$, where
\[
Y^\op\otimes
X \xrightarrow{G\otimes F} {Y'}^\op\otimes X' \xrightarrow{M'}\ca{V} \mathrlap{.}
\]
Horizontal composition of $M\colon X \horightarrow Y$ and $N\colon Y \horightarrow Z$ is
given by the coend
\begin{equation}
\label{equ:comp-of-prof}
 (N\circ M)(z,x)=\int^{y\in Y} N(z,y) \otimes M(y,x) \mathrlap{.}
\end{equation}
The horizontal identity $\hid_X \co X \horightarrow X$ is defined by $\hid_X(x',x) = X[x',x]$\footnote{We use brackets for hom-objects
of categories, and parentheses for applications of (pro)functors and other maps.}. Note that $\VMat$ can be regarded as
a sub-double category of $\VProf$ by considering sets as discrete $\ca{V}$-categories.
The horizontal bicategory of $\VProf$ is the familiar bicategory of profunctors, while the vertical 2-category is the 2-category
of small $\ca{V}$-categories.
\end{ex}

These examples illustrate the fact, pointed out in Remark~\ref{rk:directions}, that in our double categorical structures, the vertical structure is strict, and relatively straightforward, while the horizontal structure is weaker and requires careful consideration of coherence issues. The management of these coherence issues is simplified when the horizontal and vertical structures can be universally related in the following way.

\begin{defi} Let $\dc{C}$ be a double category. A \emph{companion} for a vertical
  $1$-cell $f \co X \to X'$ in $\dc{C}$ is given by a horizontal
  1-cell $\comp{f}\colon X\horightarrow X'$ along with 2-morphisms
\begin{equation}
\label{eq:comp2cells}
\begin{tikzcd}
 X\ar[r,tick,"\comp{f}"]\ar[d,"f"']\ar[dr,phantom,"\Two p_1"] & {X'}\ar[d,equal] \\
 {X'}\ar[r,tick,"\hid_{X'}"'] & {X'}
 \end{tikzcd}
 \qquad \text{and} \qquad
 \begin{tikzcd}
 X\ar[r,tick,"\hid_X"]\ar[d,equal]\ar[dr,phantom,"\Two p_2"] & X\ar[d,"f"] \\
 X\ar[r,tick,"\comp{f}"'] & {X'}
 \end{tikzcd}
\end{equation}
such that
\begin{equation}\label{eq:piaxioms}
\begin{tikzcd}
X\ar[r,tick,"\hid_X"]\ar[d,equal]\ar[dr,phantom,"\Two p_2"] & X\ar[d,"f"] \\
X\ar[d,"f"']\ar[r,tick,"\comp{f}"]\ar[dr,phantom,"\Two p_1"] & X'\ar[d,equal] \\
X'\ar[r,tick,"\hid_{X'}"'] & X'
\end{tikzcd}=\hid_f\qquad\text{and} \qquad
\begin{tikzcd}
X\ar[r,tick,"\hid_X"]\ar[d,equal]\ar[dr,phantom,"\Two p_2"] & X\ar[d,"f"]\ar[r,tick,"\comp{f}"]\ar[dr,phantom,"\Two p_1"] & X'\ar[d,equal] \\
X\ar[r,tick,"\comp{f}"']\ar[d,equal]\ar[drr,phantom,"\Two \ell"] & X'\ar[r,tick,"\hid_{X'}"'] & X' \ar[d,equal]\\
X\ar[rr,tick,"\comp{f}"'] & & X'
\end{tikzcd} =
\begin{tikzcd}
X\ar[r,tick,"\hid_X"]\ar[d,equal]\ar[drr,phantom,"\Two r"] & X\ar[r,tick,"\comp{f}"] & X'\ar[d,equal] \\
X\ar[rr,tick,"\comp{f}"'] & & X'\mathrlap{ .}
\end{tikzcd}
\end{equation}
\end{defi}

Although companions are defined algebraically, they also
have a universal characterisation.
\begin{lem}
  \label{lem:universalcompcoj}
  Let $f \colon X \rightarrow X'$ be vertical $1$-cell of a double category $\dc{C}$.
  Giving a companion $(\comp{f}, p_1, p_2)$ for $f$ is equivalent to giving either of the following:
  \begin{itemize}
  \item A horizontal $1$-cell $\comp{f}$ and $2$-morphism $p_1$ as in~\eqref{eq:comp2cells} such that, for every horizontal $1$-cell ${M \colon X'
\horightarrow Y'}$, the composite $2$-morphism to the left below is cartesian with respect to  $(\mathbf{s}, \mathbf{t}) \colon \mathbb{C}_1
\rightarrow \mathbb{C}_0 \times \mathbb{C}_0$; or
  \item A horizontal $1$-cell $\comp{f}$ and $2$-morphism $p_2$ as in~\eqref{eq:comp2cells} such that, for every horizontal $1$-cell ${M \colon W
\horightarrow X}$, the composite $2$-morphism to the right below is opcartesian with respect to $(\mathbf{s}, \mathbf{t}) \colon \mathbb{C}_1
\rightarrow \mathbb{C}_0 \times \mathbb{C}_0$:
  \end{itemize}
  \begin{equation*}
    \begin{tikzcd}
      X\ar[r,tick,"\comp{f}"]\ar[d,"f"']\ar[dr,phantom,"\Two p_1"] & X'\ar[r,tick,"M"]\ar[d,equal] & Y'\ar[d,equal] \\
      X'\ar[r,tick,"\hid_{X'}"]\ar[d,equal]\ar[drr,phantom,"\Two r"] & X' \ar[r,tick,"M"] & Y' \ar[d,equal]\\
      X'\ar[rr,tick,"M"'] & & Y'
    \end{tikzcd} \qquad\qquad
    \begin{tikzcd}
      W\ar[rr,tick,"M"]\ar[d,equal]\ar[drr,phantom,"\Two \ell^{\mi1}"] & & X\ar[d,equal]\\
      W\ar[r,tick,"M"]\ar[d,equal] & X\ar[r,tick,"\hid_X"]\ar[d,equal]\ar[dr,phantom,"\Two p_2"] & X\ar[d,"f"] \\
      W\ar[r,tick,"M"'] & X \ar[r,tick,"\comp{f}"'] & X'\mathrlap{ .}
    \end{tikzcd}
  \end{equation*}
\end{lem}

By virtue of this result, companions of a vertical 1-cell are
unique up to a unique globular $2$-isomorphism, so that by the usual abuse of notation we may refer simply to
\emph{the} companion. In what follows we will often
require the existence of certain companions, but in many examples
of interest we have all companions and also all \emph{conjoints}---the
dual notion to companion, which associates to a vertical $1$-cell $f
\colon X \rightarrow X'$ a horizontal $1$-cell $\coj{f} \colon X'
\horightarrow X$ along with unit and counit $2$-morphisms. In this case, we
may speak of a
\emph{framed bicategory} in the sense of~\cite{Framedbicats} or \emph{fibrant double category} in other references. By the
above lemma and an appropriate dual lemma for conjoints, a double category is a framed bicategory if and only if~$(\mathbf{s},\mathbf{t})\colon\dc{C}_1\to\dc{C}_0\times\dc{C}_0$ is a Grothendieck
fibration, or equivalently, a Grothendieck opfibration.

\begin{ex}\label{ex:Proffibrant}
The double category $\VProf$ of $\ca{V}$-profunctors is a fibrant double category. The
companion and conjoint of a $\ca{V}$-functor
$F\colon X \to X'$ are the $\ca{V}$-profunctors $\comp{F}\colon X\horightarrow X'$
and
$\coj{F}\colon X' \horightarrow X$ given by
\begin{equation}
\label{equ:comp-and-coj-for-prof}
 \comp{F}(x',x)=X'[x',Fx]\quad\textrm{and}\quad\coj{F}(x,x')=X'[Fx,x'] \mathrlap{.}
 \end{equation}
 It follows \emph{a fortiori} that the double category $\VMat$ of
 matrices is also fibrant, where for a function~$f \colon X\to X'$ its companion and conjoint $\comp{f}\colon X\horightarrow X'$ and $\coj{f}\colon X'\horightarrow X$ are the $\ca{V}$-matrices
given by:
\begin{displaymath}
	\comp{f}(x',x)=\coj{f}(x,x')=\begin{cases}
		I & \mathrm{if  }\;f(x)=x' \mathrlap{,} \\
		0 & \mathrm{ otherwise.}
	\end{cases}
\end{displaymath}
Note that, in these examples, we have that $\comp F \dashv \coj F$ and $\comp f \dashv \coj f$ in the horizontal bicategory. In fact, it is \emph{always} true that the companion and conjoint of a vertical $1$-cell are adjoint in this way.
\end{ex}

The universality of companions in Lemma~\ref{lem:universalcompcoj} immediately implies the following omnibus proposition.

\begin{prop}
  \label{prop:companionomnibus}
 Let $\dc{C}$ be a double category.
 \begin{enumerate}[(i)]
  \item \label{omni-i} The vertical identity $1$-cell $\vid_X \colon X \rightarrow X$ has the horizontal identity $\hid_X \colon X \horightarrow X$ as a companion.
  \item  \label{omni-ii}  If the vertical $1$-cells $f \colon X \rightarrow X'$ and $g \colon X' \rightarrow X''$ have the companions $\comp f$ and $\comp g$, then $g \circ f$ has the companion $\comp g \circ \comp f$.
  \item  \label{omni-iii}  If the vertical $1$-cells $f \colon X \rightarrow X'$ and $g \colon Y \rightarrow Y'$ admit companions, then pasting with the companion $2$-morphisms for $f$ and $g$ gives a bijection between $2$-morphisms $\phi$ as in~\cref{eq:2map} and globular $2$-morphisms
\begin{equation}\label{eq:transpose}
\begin{tikzcd}
  X\ar[r,tick,"M"]\ar[d,equal]\ar[drr,phantom,"\Two\comp{\phi}"] &
Y\ar[r,tick,"\comp{g}"] & {Y'}\ar[d,equal] \\
  X\ar[r,tick,"\comp{f}"'] & {X'}\ar[r,tick,"{M'}"'] & {Y'}\mathrlap{ .}
 \end{tikzcd}
\end{equation}
  If $f$ and $g$ are invertible in $\dc{C}_0$, then under this correspondence $\phi$ is invertible in $\dc{C}_1$ if and only if $\comp \phi$ is.
\item \label{omni-iv}  If $\mathsf{V}(\dc{C})'$ denotes the locally full sub-$2$-category of $\mathsf{V}(\dc{C})$ with the same objects, and as
morphisms just those the vertical $1$-cells which admit companions, then taking companions underlies an identity-on-objects homomorphism of
bicategories $\mathsf{V}(\dc{C})' \rightarrow \ca{H}(\dc{C})$.
  \item \label{omni-v}  If $f \colon X \rightarrow X'$ is a vertical $1$-isomorphism in $\dc{C}$ and both $f$ and $f^{\mi1}$ admit companions, then $\smash{\comp f}$ is an
equivalence in $\ca{H}(\dc{C})$ with pseudoinverse $\smash{\comp{f^{\mi1}}}$;
\item \label{omni-vi}  If  $\phi \colon f \Rightarrow g$ is an invertible $2$-cell in $\mathsf{V}(\mathbb{C})$, and $f$ and $g$ admit companions, then $\comp \phi$ is an invertible globular $2$-morphism $\comp f \Rightarrow \comp g$ in $\dc{C}$.
  \end{enumerate}
\end{prop}

\section{Maps of double categories}
\label{sec:dbl-fun}

In this section, we recall the various kinds of maps existing between double categories, starting with the notions of double functor and oplax double functor. In most of the paper we will work with double functors, which
preserve composition and identities up to isomorphism. However, we will also need oplax double functors, which preserve composition and identities only up to a non-invertible 2-cell, in one important situation, namely when we define
 oplax monoidal structure in~\cref{def:mondoublecat}. 

\begin{defi}[Oplax double functor, double functor]\label{def:doublefunctor}
Let $\dc{C}$ and $\dc{D}$ be double categories.  An \emph{oplax double functor} $F\colon\dc{C}\to\dc{D}$ 
consists of the following data:
\begin{itemize}
\item two ordinary functors $F_0\colon\dc{C}_0\to\dc{D}_0$,
$F_1\colon\dc{C}_1\to\dc{D}_1$, denoted by the same letter $F$ below, such that $\mathbf{s}F_1 = F_0 \mathbf{s}$ and $\mathbf{t}F_1 = F_0 \mathbf{t}$, as displayed in:
\begin{displaymath}
 \begin{tikzcd}
 X\ar[r,tick,"M"]\ar[d,"f"']\ar[dr,phantom,"\scriptstyle\Downarrow\phi"] &
Y\ar[d,"g"] \\
 X'\ar[r,tick,"M'"']\ & Y'
 \end{tikzcd}\quad\mapsto\quad
 \begin{tikzcd}
 FX\ar[r,tick,"FM"]\ar[d,"Ff"']\ar[dr,phantom,"\scriptstyle\Downarrow F\phi"]
& FY\ar[d,"Fg"] \\
 FX'\ar[r,tick,"FM'"']\ & FY'\mathrlap{ ;}
 \end{tikzcd}
\end{displaymath}
\item two natural transformations with components
\begin{equation}\label{eq:oplaxfunctorcoherence}
 \begin{tikzcd}
FX\ar[rr,tick,"F(N\circ
M)"]\ar[d,equal]\ar[drr,phantom,"\scriptstyle\Downarrow\xi_{M,N}"] &&
FZ\ar[d,equal] \\
FX\ar[r,tick,"FM"'] & FY\ar[r,tick,"FN"'] & FZ
 \end{tikzcd}\qquad\qquad
 \begin{tikzcd}
FX\ar[rr,tick,"F(\hid_X)"]\ar[drr,phantom,"\scriptstyle\Downarrow\xi_X"]\ar[d,
equal] && FX\ar[d,equal] \\
FX\ar[rr,tick,"\hid_{FX}"'] && FX\mathrlap{ .}
 \end{tikzcd}
\end{equation}
\end{itemize}
These data are required to satisfy coherence conditions similar to those for an oplax functor between bicategories,
one regarding associativity and two for unitality; see~\cite[\S 7.2]{Limitsindoublecats}.

A {\em (pseudo) double functor}  $F\colon\dc{C}\to\dc{D}$ is an oplax double
functor for which the 2-cells $\xi_{M,N}$ and $\xi_X$ are invertible.
\end{defi}

\begin{lem} \label{lem:dbl-to-bicat-functor}
Let $F \co \dc{C} \to \dc{D}$ be an oplax (pseudo) double functor. Then $F$ induces an oplax (pseudo) functor of bicategories $\ca{H}(F) \co \ca{H}(\dc{C})
\to \ca{H}(\dc{D})$.
This assignment extends to an ordinary functor from the category of double categories and oplax (pseudo) double functors to the
category of bicategories and oplax (pseudo) functors.
\end{lem}

\begin{proof}
This follows from the definition, and is an oplax analogue of \cite[Theorem~4.1]{ConstrSymMonBicatsFun}.
Recall that functors of double categories and bicategories compose strictly associatively.
\end{proof}

\begin{lem}\label{lem:Ff}
  Let $F \colon \dc{C} \rightarrow \dc{D}$ be a double functor and $f \colon X \rightarrow X'$ a
  vertical $1$-cell  of $\dc{C}$. Assume that $f$ admits a companion $\comp f$ with structure cells $p_1, p_2$. Then the vertical $1$-cell $Ff \colon
FX \rightarrow FX'$ of $\dc{D}$ admits the companion $F \comp{f}$ via the structure cells
  \begin{equation*}
 \begin{tikzcd}
FX\ar[r,tick,"F\comp{f}"]\ar[d,"Ff"']\ar[dr,phantom,"\Two Fp_1"] &
F{X'}\ar[d,equal]
\\
FX'\ar[r,tick,"F(\hid_{X'})"]\ar[dr,phantom,"\scriptstyle\Downarrow\xi_{X'}"]\ar[d,
equal] & FX'\ar[d,equal] \\
FX'\ar[r,tick,"\hid_{FX'}"'] & FX'\mathrlap{ .}
\end{tikzcd} \qquad \text{and} \qquad
 \begin{tikzcd}
FX\ar[r,tick,"\hid_{FX}"]\ar[dr,phantom,"\scriptstyle\Downarrow\xi_{X}^{\mi 1}"]\ar[d,
equal] & FX\ar[d,equal] \\
FX\ar[d,equal]\ar[r,tick,"F(\hid_{X})"]\ar[dr,phantom,"\Two Fp_2"] & FX\ar[d,"Ff"]\\
FX'\ar[r,tick,"\hid_{FX'}"'] & FX'\mathrlap{ .}
\end{tikzcd}
\end{equation*}
\end{lem}

In future, we will tend to suppress the unit coherence cells $\xi$ appearing above, and simply write $Fp_1$ and $Fp_2$ for these pasting composites.

Since a double category has two kinds of $1$-cell, vertical and horizontal, there are
two kinds of natural transformations between double functors,
vertical and horizontal, depending on the directions of their components.
For our applications, it is the \emph{horizontal} transformations, recalled in \cref{def:hor-transf}, which will be most important, since these induce structure on the horizontal bicategory. However, the \emph{vertical} transformations, recalled in \cref{def:vert-transf}, are simpler to construct and work with, and so fundamental to our development will be the possibility of turning a vertical natural transformation into a horizontal
one in the presence of well-behaved companions. The precise conditions needed are isolated in the
the notion of a special vertical transformation (\cref{def:special-vertical-transformation}), and are justified in \cref{lem:2}, where we show that the special vertical transformations are exactly the vertical $1$-cells of the functor double category (\cref{thm:functor-dblcat}) which admit companions.

The following definition can be found, for example, in \cite[\S2.4]{Poly}; or in~\cite[\S7.4]{Limitsindoublecats} under the name ‘strong vertical transformation’ (recalling the reversal of sense of Remark~\ref{rk:directions}).

\begin{defi}[Horizontal transformation] \label{def:hor-transf}
\label{def:horizontaltransf}
Let $F,G\colon\dc{C}\to\dc{D}$ be  oplax double functors.
A \emph{horizontal transformation} $\beta \colon F \ticktwoar G$ consists of
\begin{itemize}
\item horizontal $1$-cell components $\beta_X\colon FX\horightarrow GX$ in $\dc{D}$ for each object
$X \in \dc{C}$;
\item $2$-morphism components
\begin{displaymath}
\twocell{FX}{\beta_X}{GX}{Gf}{GX'}{\beta_{X'}}{FX'}{Ff}{\beta_f}
\end{displaymath}
in $\dc{D}$ for each vertical $1$-cell $f\colon X\to X'$ in $\dc{C}$;
\item invertible globular coherence $2$-morphisms
\begin{displaymath}
 \begin{tikzcd}
FX\ar[r,tick,"FM"]\ar[drr,phantom,"\Two\beta_M"]\ar[d,equal] &
FY\ar[r,tick,"\beta_Y"] & GY\ar[d,equal] \\
FX\ar[r,tick,"\beta_X"'] & GX\ar[r,tick,"GM"'] & GY
 \end{tikzcd}
\end{displaymath}
in $\dc{D}$ for each horizontal $1$-cell $M\colon X\horightarrow Y$ in $\dc{C}$.
\end{itemize}
These data are required to satisfy, firstly, the two axioms
\begin{equation}\label{eq:betafunctor}
 \begin{tikzcd}
FX\ar[r,tick,"\beta_X"]\ar[d,"Ff"']\ar[dr,phantom,"\Two\beta_f"] &
GX\ar[d,"Gf"] \\
FX'\ar[r,tick,"\beta_{X'}"]\ar[d,"Fg"']\ar[dr,phantom,"\Two\beta_g"] &
GX'\ar[d,"Gg"] \\
FX''\ar[r,tick,"\beta_{X''}"'] & GX''
 \end{tikzcd}=
 \begin{tikzcd}
FX\ar[r,tick,"\beta_X"]\ar[dd,"{F(gf)}"']\ar[ddr,phantom,"\Two\beta_{gf}"] &
GX\ar[dd,"{G(gf)}"] \\
\hole \\
FX''\ar[r,tick,"\beta_{X''}"'] & GX''
 \end{tikzcd} \ \ \text{and} \ \
 \begin{tikzcd}
FX\ar[r,tick,"\beta_X"]\ar[d,equal,"F(\vid_X)"']\ar[dr,phantom,"\Two\beta_{\vid_X}"] &
GX\ar[d,equal,"G(\vid_X)"]  \\
FX\ar[r,tick,"\beta_{X}"'] & FX
 \end{tikzcd}=
 \begin{tikzcd}
FX\ar[r,tick,"\beta_X"]\ar[dr,phantom,"\Two\vid_{\beta_X}"]\ar[d,equal] &
GX\ar[d,equal] \\
FX\ar[r,tick,"\beta_X"'] & GX
 \end{tikzcd}
\end{equation}
expressing that $\beta_{(\thg)}\colon\dc{C}_0\to\dc{D}_1$ is a
functor; then the axiom
\begin{equation}\label{eq:betanaturality}
\begin{tikzcd}
FX\ar[r,tick,"FM"]\ar[d,equal]\ar[drr,phantom,"\Two\beta_M"] &
FY\ar[r,tick,"\beta_Y"] & GY\ar[d,equal] \\
FX\ar[r,tick,"\beta_X"']\ar[d,"Ff"']\ar[dr,phantom,"\Two\beta_f"] &
GX\ar[r,tick,"GM"']\ar[dr,phantom,"\Two
G\phi"]\ar[d,"Gf"] & GY\ar[d,"Gg"] \\
FX'\ar[r,tick,"\beta_{X'}"'] & GX'\ar[r,tick,"GM'"'] & GY'
\end{tikzcd}=
\begin{tikzcd}
 FX\ar[d,"Ff"']\ar[r,tick,"FM"]\ar[dr,phantom,"\Two F\phi"] &
FY\ar[d,"Fg"]\ar[dr,phantom,"\Two\beta_g"]\ar[r,tick,"\beta_Y"] & GY\ar[d,"Gg"]
\\
FX'\ar[r,tick,"FM'"']\ar[d,equal]\ar[drr,phantom,"\Two\beta_{M'}"] &
FY'\ar[r,tick,"\beta_{Y'}"'] & GY'\ar[d,equal] \\
FX'\ar[r,tick,"\beta_{X'}"'] & GX'\ar[r,tick,"GM'"'] & GY'
\end{tikzcd}
\end{equation}
expressing that the $2$-morphisms $\beta_M$ are components of a
natural transformation; and finally, the axioms
\begin{displaymath}
\begin{tikzcd}
FX\ar[rr,tick,"F(N\circ M)"]\ar[d,equal]\ar[drr,phantom,"\Two\xi_{M,N}"] &&
FZ\ar[d,equal]\ar[r,tick,"\beta_Z"] & GZ\ar[d,equal] \\
FX\ar[r,tick,"FM"']\ar[d,equal] &
FY\ar[r,tick,"FN"']\ar[drr,phantom,"\Two\beta_N"]\ar[d,equal] &
FZ\ar[r,tick,"\beta_Z"'] & GZ\ar[d,equal] \\
FX\ar[r,tick,"FM"']\ar[drr,phantom,"\Two\beta_M"]\ar[d,equal] &
FY\ar[r,tick,"\beta_Y"'] & GY\ar[r,tick,"GM"']\ar[d,equal] & GZ\ar[d,equal] \\
FX\ar[r,tick,"\beta_X"'] & GX\ar[r,tick,"GM"'] & GY\ar[r,tick,"GN"'] & GZ
\end{tikzcd}=
\begin{tikzcd}
FX\ar[rr,tick,"F(N\circ M)"]\ar[d,equal]\ar[drrr,phantom,"\Two\beta_{N\circ
M}"] && FZ\ar[r,tick,"\beta_Z"] & GZ\ar[d,equal] \\
FX\ar[r,tick,"\beta_X"']\ar[d,equal] & GX\ar[rr,tick,"G(N\circ
M)"']\ar[d,equal]\ar[drr,phantom,"\Two\xi_{M,N}"] && GZ\ar[d,equal]\\
FX\ar[r,tick,"\beta_X"'] & GX\ar[r,tick,"GM"'] & GY\ar[r,tick,"GN"'] & GZ
\end{tikzcd}
\end{displaymath}
\begin{displaymath}
 \begin{tikzcd}
FX\ar[r,tick,"F(\hid_X)"]\ar[d,equal]\ar[drr,phantom,"\Two\beta_{\hid}"] &
FX\ar[r,tick,"\beta_X"] & GX\ar[d,equal] \\
FX\ar[r,tick,"\beta_X"']\ar[d,equal] &
GX\ar[r,tick,"G(\hid_X)"']\ar[d,equal]\ar[dr,phantom,"\Two\xi_X"] &
GX\ar[d,equal] \\
FX\ar[r,tick,"\beta_X"']\ar[d,equal] & GX\ar[r,tick,"\hid_{GX}"'] &
GX\ar[d,equal] \\
FX\ar[rr,tick,"\beta_X"'] && GX
 \end{tikzcd}=
 \begin{tikzcd}
FX\ar[r,tick,"F(\hid_X)"]\ar[d,equal]\ar[dr,phantom,"\Two\xi_X"] &
FX\ar[r,tick,"\beta_X"]\ar[d,equal] & GX\ar[d,equal] \\
FX\ar[r,tick,"\hid_{FX}"']\ar[d,equal] & FX\ar[r,tick,"\beta_X"'] &
GX\ar[d,equal] \\
FX\ar[rr,tick,"\beta_X"'] && GX
 \end{tikzcd}
\end{displaymath}
expressing compatibility of $\beta$ with the double structure of $F$ and
$G$.
\end{defi}

The following definition can be found {\em e.g.} in \cite[\S2.3]{Poly} or \cite[Definition~2.8]{ConstrSymMonBicatsFun} under the name `tight
transformation'.

\begin{defi}[Vertical transformation] \label{def:vert-transf}
Let $F,F'\colon\dc{C}\to\dc{D}$ be oplax
double functors. A \emph{vertical transformation} $\sigma\colon F\Rightarrow F'$
consists of the following data:
\begin{itemize}
\item vertical $1$-cell components $\sigma_X\colon FX\to F'X$ in $\dc{D}$ for each object $X \in \dc{C}$;
\item $2$-morphism components
\begin{equation}\label{eq:transfcomponent}
 \begin{tikzcd}
FX\ar[r,tick,"FM"]\ar[d,"\sigma_X"']\ar[dr,phantom,
"\scriptstyle\Downarrow\sigma_M"] & FY\ar[d,"\sigma_Y"] \\
  F'X\ar[r,tick,"F'M"'] & F'Y
 \end{tikzcd}
\end{equation}
in $\dc{D}$ for each horizontal $1$-cell $M \colon X \tickar Y$ in $\dc{C}$.
\end{itemize}
These data are required to satisfy, firstly, the axiom
\begin{equation}\label{eq:verticaltransfax0}
 \begin{tikzcd}
FX\ar[r,tick,"FM"]\ar[d,"\sigma_X"']\ar[dr,phantom,
"\Two\sigma_M"] & FY\ar[d,"\sigma_Y"] \\
F'X\ar[r,tick,"F'M"]\ar[d,"F'f"']\ar[dr,phantom,"\Two F'\phi"] &
F'Y\ar[d,"F'g"] \\
F'X'\ar[r,tick,"F'M'"'] & F'Y'
 \end{tikzcd}=
 \begin{tikzcd}
FX\ar[r,tick,"FM"]\ar[d,"Ff"']\ar[dr,phantom,"\Two F\phi"] & FY\ar[d,"Fg"] \\
FX'\ar[r,tick,"FM'"]\ar[d,"\sigma_{X'}"']\ar[dr,phantom,"\Two\sigma_{M'}"] &
FY'\ar[d,"\sigma_{Y'}"] \\
F'X'\ar[r,tick,"F'M'"'] & F'Y'
 \end{tikzcd}
\end{equation}
expressing that we have two ordinary natural transformations $F_0\Rightarrow F'_0$ and
$F_1\Rightarrow F'_1$; then the axiom
\begin{equation}\label{eq:verticaltransfax}
 \begin{tikzcd}
FX\ar[rr,tick,"F(N\circ
M)"]\ar[d,equal]\ar[drr,phantom,"\scriptstyle\Downarrow\xi_{M,N}"] &&
FZ\ar[d,equal] \\
FX\ar[r,tick,"FM"]\ar[d,"\sigma_X"']\ar[dr,phantom,
"\scriptstyle\Downarrow\sigma_M"] &
FY\ar[r,tick,"FN"]\ar[d,"\sigma_Y"description]\ar[dr,phantom,"\Two\sigma_N"] &
FZ\ar[d,"\sigma_Z"] \\
F'X\ar[r,tick,"F'M"'] & F'Y\ar[r,tick,"F'N"'] & F'Z
 \end{tikzcd}=
 \begin{tikzcd}
FX\ar[rr,tick,"F(N\circ
M)"]\ar[d,"\sigma_X"']\ar[drr,phantom,"\scriptstyle\Downarrow\sigma_{N\circ
M}"]
&& FZ\ar[d,"\sigma_Z"] \\
F'X\ar[rr,tick,"F'(N\circ M)"']\ar[drr,phantom,"\Two\xi_{M,N}"]\ar[d,equal]
&&
F'Z\ar[d,equal] \\
F'X\ar[r,tick,"F'M"'] & F'Y\ar[r,tick,"F'N"'] & F'Z
 \end{tikzcd}
 \end{equation}
expressing compatibility with horizontal composition; and finally, the axiom
\begin{equation}
\label{eq:verticaltransfax1}
 \begin{tikzcd}
 FX\ar[r,tick,"F(\hid_X)"]\ar[d,equal]\ar[dr,phantom,"\Two\xi_X"] &
FX\ar[d,equal] \\
 FX\ar[r,tick,"\hid_{FX}"]\ar[d,"\sigma_X"']\ar[dr,phantom,"\Two\hid_{\sigma_X}"]
& FX\ar[d,"\sigma_X"] \\
 F'X\ar[r,tick,"\hid_{F'X}"'] & F'X
 \end{tikzcd}=
 \begin{tikzcd}
FX\ar[r,tick,"F(\hid_X)"]\ar[d,"\sigma_X"']\ar[dr,phantom,"\Two\sigma_{
\hid_X } " ] & FX\ar[d,"\sigma_X"] \\
 F'X\ar[r,tick,"F'(\hid_X)"']\ar[d,equal]\ar[dr,phantom,"\Two\xi_X"] &
F'X\ar[d,equal] \\
 F'X\ar[r,tick,"\hid_{F'X}"'] & F'X
 \end{tikzcd}
\end{equation}
expressing compatibility with horizontal identities.
\end{defi}

It is easy to see that a horizontal transformation $\beta \co F \ticktwoar F'$ between double functors induces a pseudonatural transformation $\ca{H}(\beta)\colon\ca{H}(F)\Rightarrow\ca{H}(F')$ between the associated homomorphisms of bicategories. On the other hand, from a vertical transformation $\sigma \colon F \Rightarrow F'$, there is no direct way of inducing anything $\ca{H}(F)\Rightarrow\ca{H}(F')$. However, there is an indirect way of doing so, if we can first turn the vertical transformation
$\sigma \colon F \Rightarrow F'$ into a horizontal one $\comp \sigma \colon F \ticktwoar F'$. The following definition isolates the properties required of $\sigma$ for this to be possible.

\begin{defi}[Special vertical transformation] \label{def:special-vertical-transformation}
Let $\sigma\colon F\Rightarrow F'$ be a vertical transformation between oplax double functors $\dc{C} \rightarrow \dc{D}$. We say that $\sigma$ is
\emph{special} if for every $X \in \dc{C}$, the vertical $1$-cell component $\sigma_X \colon FX \rightarrow
F'X$
has a companion $\comp{\sigma}_X \colon FX \horightarrow F'X$ in $\dc{D}$, and the
companion transposes
\begin{equation}\label{eq:6}
\begin{tikzcd}
  FX\ar[r,tick,"FM"]\ar[d,equal]\ar[drr,phantom,"\Two\comp{\sigma}_M"] &
FY\ar[r,tick,"\comp{\sigma}_Y"] & F'Y\ar[d,equal] \\
  FX\ar[r,tick,"\comp{\sigma}_X"'] & F'X\ar[r,tick,"F'M"'] & F'Y
 \end{tikzcd}
\end{equation}
of the $2$-morphism components~\eqref{eq:transfcomponent} of $\sigma$ are invertible.
\end{defi}

A special vertical transformation was called a transformation with \emph{loosely strong companions} in \cite[Definition~4.10]{ConstrSymMonBicatsFun},
characterised precisely by the following proposition; when considered in the setting of fibrant double categories, it was called a \emph{horizontally strong} 
transformation in \cite[Definition~A.4]{CruttwellShulman}.

\begin{prop}
  \label{prop:inducedhorizontalfromspecial}
  Let $\sigma \colon F \Rightarrow F'$ be a special vertical transformation between oplax double functors. The companion $1$-cells $\comp{\sigma}_X \colon FX \tickar F'X$ are the horizontal $1$-cell components of a horizontal transformation $\comp \sigma \colon F \ticktwoar F'$, whose $2$-morphism components $\comp{\sigma}_f$ are the companion transposes of the equalities of vertical $1$-cells $\sigma_{X'} \circ Ff = F'f \circ \sigma_X$ as in:
\begin{equation*}
  \twocell{FX}{\comp{\sigma}_X}{F'X}{F'f}{F'{X'}}{\comp{\sigma}_{X'}}{FX'}{Ff}{\comp{\sigma}_f}
\end{equation*}
and whose globular coherence $2$-morphisms are given by~\cref{eq:6}. In particular, $\sigma$ induces a pseudonatural transformation $\ca{H}(\hat
\sigma) \colon \ca{H}(F) \Rightarrow \ca{H}(F')$ between the induced oplax functors of horizontal bicategories.
\end{prop}
\begin{proof}
  The horizontal transformation axioms are a straightforward diagram chase using the universal property of companions.
\end{proof}

We could now proceed to verify by hand further desirable properties of the construction $\sigma \mapsto \comp \sigma$ (for example, its
functoriality); however, this turns out to be unneccessary, as we can in fact characterise $\comp \sigma$ as a companion for $\sigma$ in a suitable
\emph{functor double category}, and then apply results such as \cref{prop:companionomnibus}. We first define these functor double categories.

\begin{defi}[Modification]\label{def:modification}
Let $\beta \co F \ticktwoar G$ and $\beta' \co F' \ticktwoar G'$ be horizontal transformations  and let
$\sigma \co F \Rightarrow F'$ and $\tau \co G \Rightarrow G'$ be vertical transformations between oplax double functors $\dc{C} \rightarrow \dc{D}$.
 A \emph{modification}
\begin{displaymath}
\begin{tikzcd}
F\ar[r,Rightarrow,bigtick,"\beta"]\ar[d,Rightarrow,"{\sigma}"']\ar[dr,phantom,"\scriptstyle\Ddownarrow{\gamma}"] & G \ar[d,Rightarrow,"\tau"] \\
F'\ar[r,Rightarrow,bigtick,"{\beta'}"'] & G'
\end{tikzcd}
\end{displaymath}
consists of  2-morphisms
\begin{displaymath}
  \twocell{FX}{\beta_X}{GX}{\tau_X}{G'X}{\beta'_X}{F'X}{\sigma_X}{\gamma_X}
\end{displaymath}
in $\dc{D}$ for every object $X \in \dc{C}$, subject to the naturality axiom:
\begin{displaymath}
 \begin{tikzcd}
FX\ar[r,tick,"\beta_X"]\ar[d,"\sigma_X"']\ar[dr,phantom,"\Two\gamma_X"] &
GX\ar[d,"\tau_X"] \\
F'X\ar[r,tick,"\beta'_X"']\ar[d,"F'f"']\ar[dr,phantom,"\Two\beta'_f"] &
G'X\ar[d,"G'f"] \\
F'X'\ar[r,tick,"\beta'_{X'}"'] & G'X'
 \end{tikzcd}=
 \begin{tikzcd}
FX\ar[r,tick,"\beta_X"]\ar[d,"Ff"']\ar[dr,phantom,"\Two\beta_f"] &
GX\ar[d,"Gf"] \\
FX'\ar[r,tick,"\beta_{X'}"']\ar[d,"\sigma_{X'}"']\ar[dr,phantom,"\Two\gamma_{X'}
" ] &
GX'\ar[d,"\tau_{X'}"] \\
F'X'\ar[r,tick,"\beta'_{X'}"'] & G'X'
 \end{tikzcd}\mathrlap{;}
\end{displaymath}
and the following axiom expressing compatibility with $\beta$, $\beta'$, $\sigma$, $\tau$:
\begin{displaymath}
\begin{tikzcd}
FX\ar[r,tick,"FM"]\ar[d,equal]\ar[drr,phantom,"\Two\beta_M"] &
FY\ar[r,tick,"\beta_Y"] & GY\ar[d,equal] \\
FX\ar[r,tick,"\beta_X"']\ar[dr,phantom,"\Two\gamma_X"]\ar[d,"\sigma_X"'] &
GX\ar[d,"\tau_X"']\ar[r,tick,"GM"']\ar[dr,phantom,"\Two\tau_M"] &
GY\ar[d,"\tau_Y"] \\
F'X\ar[r,tick,"\beta'_X"'] & G'X\ar[r,tick,"G'M"'] & G'Y
\end{tikzcd}=
\begin{tikzcd}
FX\ar[d,"\sigma_X"']\ar[r,tick,"FM"]\ar[dr,phantom,"\Two\sigma_M"] &
FY\ar[d,"\sigma_Y"']\ar[dr,phantom,"\Two\gamma_Y"]\ar[r,tick,"\beta_Y"] &
GY\ar[d,"\tau_Y"] \\
F'X\ar[r,tick,"F'M"']\ar[d,equal]\ar[drr,phantom,"\Two\beta'_M"] &
F'Y\ar[r,tick,"\beta'_Y"'] & G'Y\ar[d,equal] \\
F'X\ar[r,tick,"\beta'_X"'] & G'X\ar[r,tick,"G'M"'] & G'Y \mathrlap{.}
\end{tikzcd}
\end{displaymath}
\end{defi}

\begin{prop}[Functor double categories] \label{thm:functor-dblcat}
Let $\dc{C}$ and $\dc{D}$ be double categories.
There is a double category $\nc{DblCat}[\dc{C},\dc{D}]$ (resp., $\nc{DblCat}_{\mathrm{oplax}}[\dc{C},\dc{D}]$) composed of double functors from $\dc{C}$ to $\dc{D}$ (resp., oplax double functors), vertical transformations, horizontal
transformations and modifications.
\end{prop}

\begin{proof}
  Each of the forms of vertical and horizontal composition is given by composing the relevant component data in the same direction; verifying the
axioms is routine. The only point requiring any further note is that, for composable horizontal transformations $\beta \colon F \ticktwoar G$ and
$\delta \colon G \ticktwoar H$, the globular coherence $2$-isomorphisms of the composite $\delta \circ \beta \colon F \ticktwoar H$ are given by
  \begin{equation*}
     \begin{tikzcd}
FX\ar[r,tick,"FM"]\ar[drr,phantom,"\Two\beta_M"]\ar[d,equal] &
FY\ar[r,tick,"\beta_Y"] & GY\ar[d,equal] \ar[r,tick,"\delta_Y"] & HY \ar[d,equal]\\
FX\ar[d,equal]\ar[r,tick,"\beta_X"] & GX\ar[drr,phantom,"\Two\delta_M"]\ar[d,equal]\ar[r,tick,"GM"] & GY \ar[r,tick,"\delta_Y"] & HY
\ar[d,equal]\\
FX\ar[r,tick,"\beta_X"'] & GX\ar[r,tick,"\delta_X"'] & HX \ar[r,tick,"HM"'] & HY\mathrlap{ .}
 \end{tikzcd} \qedhere
\end{equation*}
\end{proof}

\begin{prop} \label{lem:2}
  Let $\sigma \colon F \Rightarrow F'$  be a vertical transformation between double functors (resp., oplax double functors).
  Then $\sigma$ has a
companion  as a vertical 1-cell in the double category $\nc{DblCat}[\dc{C}, \dc{D}]$ (resp., $\nc{DblCat}_{\mathrm{oplax}}[\dc{C}, \dc{D}]$) if and only if it is
special in the sense of \cref{def:special-vertical-transformation}.
\end{prop}

\begin{proof} If $\sigma$ is special, then we have an associated horizontal transformation $\comp{\sigma} \co  F \ticktwoar F'$ via \cref{prop:inducedhorizontalfromspecial}. Moreover, we can define modifications
\begin{equation}
  \label{eq:7}
\begin{tikzcd}
 F\ar[r,Rightarrow,bigtick,"\comp{\sigma}"]\ar[d,Rightarrow,"\sigma"']\ar[dr,phantom,"\Two p_1"] & {F'}\ar[d,equal] \\
 {F'}\ar[r,Rightarrow,bigtick,"\hid_{F'}"'] & {F'}
 \end{tikzcd}
 \qquad \text{and} \qquad
 \begin{tikzcd}
 F\ar[r,Rightarrow,bigtick,"\hid_F"]\ar[d,equal]\ar[dr,phantom,"\Two p_2"] & F\ar[d,Rightarrow,"\sigma"] \\
 F\ar[r,Rightarrow,bigtick,"\comp{\sigma}"'] & {F'}
 \end{tikzcd}
\end{equation}
whose component $2$-morphisms are those witnessing that each $\comp{\sigma}_X$ is a companion of $\sigma_X$;
the modification axioms of \cref{def:modification} are now easily verified, and it is clear that these modifications satisfy the companion axioms since they do so componentwise.

Suppose conversely that $\sigma$ has a companion in the functor double category, namely a horizontal transformation $\comp\sigma$
with the modifications witnessing this given as in~\eqref{eq:7}. The components of these modifications witness that each horizontal $1$-cell ${\comp
\sigma}_X$ is a companion for the vertical $1$-cell $\sigma_X$. Furthermore, the second modification axiom for $p_1$ ensures that the invertible
coherence $2$-morphism $\comp{\sigma}_M$ is the companion transpose of the $2$-morphism component $\sigma_M$; in particular, this says that $\sigma$
is special as required.
\end{proof}

As a sample application of the utility of this result, let us use it to give an efficient proof of:

\begin{prop}\label{lem:pseudonatadjequiv}
  Let $\sigma\colon F\Rightarrow F' \colon \dc{C} \rightarrow \dc{D}$
  be an invertible vertical transformation between oplax double
  functors. If the $1$-cell components of $\sigma$ and
  $\sigma^{\mi 1}$ have companions, then they induce a horizontal
  equivalence $\comp{\sigma} \colon F \ticktwoar F'$ and so a pseudonatural
  equivalence $\comp{\sigma}\colon\ca{H}(F)\Rightarrow\ca{H}(F')$
  between oplax functors of bicategories.
\end{prop}
\begin{proof}
  Since $\sigma$ is invertible and its components have companions, it is special; likewise, $\sigma^{\mi1}$ is special. So by \cref{lem:2}, both $\sigma$ and $\sigma^{\mi1}$ admit companions in $\nc{DblCat}_{\mathrm{oplax}}[\dc{C}, \dc{D}]$. It follows by~\cref{prop:companionomnibus}\ref{omni-v} that $\comp \sigma$ is an equivalence in $\ca{H}(\nc{DblCat}_{\mathrm{oplax}}[\dc{C}, \dc{D}])$ as desired.
\end{proof}

We conclude this section with a miscellaneous technical lemma concerning components of a vertical transformation, which will be used in~\cref{sec:monoids,sec:kleisli,sec:application}.

\begin{lem}\label{lem:companioncomponents}
Let $\sigma\colon F\Rightarrow F'$ be a vertical transformation
and $f\colon X\to X'$ be a vertical 1-cell in $\dc{C}$. If $f$  has a companion $\comp{f}$, then the component
$\sigma_{\comp{f}}$ is the transpose of the naturality
vertical identity as in
\begin{displaymath}
\twocell{FX}{F\comp{f}}{F{X'}}{\sigma_{X'}}{F'{X'}}{F'\comp{f}}{F'X}{\sigma_X}{
\sigma_{\comp {f}}}
\;=\;
\begin{tikzcd}
FX\ar[r,tick,"\hid"]\ar[d,"\sigma_X"'] & FX\ar[r,tick,"\hid"]\ar[d,"\sigma_X"] &
FX\ar[r,tick,"F\comp{f}"]\ar[d,"Ff"']\ar[dr,phantom,"\Two Fp_1"] &
F{X'}\ar[d,equal]
\\
F'X\ar[r,tick,"\hid"]\ar[d,equal]\ar[dr,phantom,"\Two F'p_2"] & F'X\ar[d,"F'f"]
&
F{X'}\ar[d,"\sigma_{X'}"']\ar[r,tick,"\hid"] & F{X'}\ar[d,"\sigma_{X'}"] \\
F'X\ar[r,tick,"F'\comp{f}"'] & F'{X'}\ar[r,tick,"\hid"'] &
F'{X'}\ar[r,tick,"\hid"'] &
F'{X'} \mathrlap{.}
\end{tikzcd}
\end{displaymath}
\end{lem}
\begin{proof}
It suffices to show these two $2$-morphisms have the same companion transposes, which follows by the calculation (in which we again suppress unit coherence $2$-morphisms for $F$ and $F'$):
\begin{displaymath}
 \begin{tikzcd}
FX\ar[r,tick,"\hid"]\ar[dr,phantom,"\Two Fp_2"]\ar[d,equal] &
FX\ar[d,"Ff"] \\
FX\ar[r,tick,"F\comp{f}"]\ar[d,"\sigma_X"']\ar[dr,phantom,"\Two\sigma_{\comp{f}}
"]
& F{X'}\ar[d,"\sigma_{X'}"] \\
F'X\ar[r,tick,"F'\comp{f}"']\ar[d,"F'f"']\ar[dr,phantom,"\Two F'p_1"] &
F'{X'}\ar[d,equal] \\
F'{X'}\ar[r,tick,"\hid"'] & F'{X'}
 \end{tikzcd}=
 \begin{tikzcd}
FX\ar[r,tick,"\hid"]\ar[dr,phantom,"\Two Fp_2"]\ar[d,equal] &
FX\ar[d,"Ff"] \\
FX\ar[r,tick,"F\comp{f}"]\ar[d,"Ff"']\ar[dr,phantom,"\Two Fp_1"]
& F{X'}\ar[d,equal] \\
F{X'}\ar[r,tick,"F(\hid)"']\ar[d,"\sigma_{X'}"']\ar[dr,phantom,"\Two
\sigma_{\hid}"] &
F{X'}\ar[d,"\sigma_{X'}"] \\
F'{X'}\ar[r,tick,"\hid"'] & F'{X'}
 \end{tikzcd}=
 \begin{tikzcd}
FX\ar[r,tick,"\hid"]\ar[d,"Ff"'] & FX\ar[d,"Ff"] \\
F{X'}\ar[d,"\sigma_{X'}"'] & F{X'}\ar[d,"\sigma_{X'}"]\\
F'{X'}\ar[r,tick,"\hid"'] & F'{X'}
 \end{tikzcd}
\end{displaymath}
using naturality of $\sigma$; the companion axiom \cref{eq:piaxioms}; and axiom \cref{eq:verticaltransfax1} for a vertical transformation.
\end{proof}

\section{Monoidal double categories}
\label{sec:mon-dbl-cat}

The aim of this section is to introduce the notions of \emph{monoidal double category} and \emph{oplax monoidal double category}, and to
prove some useful facts about them. Both notions describe double categories endowed with a monoidal product: the key difference is that in the former
case, this  tensor product is a double functor, while in the latter case, it is merely an oplax double functor as in~\cref{def:doublefunctor}. In
particular, it should be emphasised that `oplax' only modifies the functoriality of
the tensor product, rather than the nature of the associativity and unit constraints for this tensor, which for us will always be invertible.

In what follows, we will be concerned with the the question of extending a \emph{monoidal} structure on a double category to an \emph{oplax monoidal}
structure on an associated Kleisli double category. Since we need both notions, 
we here define them simultaneously.

\begin{defi}[Oplax monoidal double category, monoidal double category] \label{def:mondoublecat}\label{def:oplaxdoublecat}
Let $\dc{C}$ be a double category. An \emph{oplax monoidal structure} on $\dc{C}$ consists of the following data:
\begin{itemize}
\item an oplax double functor $\ot\colon\dc{C}\times\dc{C}\to\dc{C}$;
\item an oplax double functor $I\colon\one\to\dc{C}$;
\item invertible vertical transformations $\alpha \co \mathord{\ot}\circ(1\times\mathord{\ot}) \Rightarrow \mathord{\ot}\circ(\mathord{\ot}\times1)$,
$\lambda \co \mathord{\ot}\circ(I\times 1) \Rightarrow 1$ and $\rho \co \mathord{\ot}\circ(1 \times I)\Rightarrow 1$
\end{itemize}
satisfying the usual Mac~Lane coherence axioms for $\alpha$, $\lambda$ and $\rho$.
Said another way, this structure amounts to the following:
\begin{itemize}
 \item monoidal structures $(\ot_0,I_0)$ and $(\ot_1,I_1)$ on the categories $\dc{C}_0$ and $\dc{C}_1$;
 \item \emph{strict} monoidality of $\mathbf{s}, \mathbf{t} \colon \dc{C}_1 \rightrightarrows \dc{C}_0$. For example, the associativity
constraint for $\dc{C}_1$ has components
 \begin{equation}\label{eq:associativitycomponents}
 \begin{tikzcd}[column sep=.7in]
(X_1\ot X_2)\ot X_2 \ar[r,tick,"(M_1 \ot M_2)\ot
M_3"]\ar[d,"\alpha_{X_1,X_2,X_3}"']\ar[dr,phantom,"\Two\alpha_{M_1,M_2,M_3}"] &
(Y_1 \ot Y_2)\ot Y_3 \ar[d,"\alpha_{Y_1,Y_2,Y_3}"] \\
X_1 \ot(X_2 \ot X_3)\ar[r,tick,"M_1 \ot(M_2 \ot M_3)"'] & Y_1 \ot(Y_2 \ot Y_3) \mathrlap{;}
 \end{tikzcd}
 \end{equation}
 \item globular 2-morphisms
 \begin{equation}\label{eq:structure2cells1}
\begin{tikzcd}[column sep=.5in]
 X_1\ot X_2\ar[rr,tick,"(N_1\circ M_1)\ot(N_2\circ M_2)"]\ar[d,equal]\ar[drr,phantom,"\scriptstyle\Downarrow\tau"] && Z_1\ot Z_2\ar[d,equal] \\
 X_1\ot X_2\ar[r,tick,"M_1\ot M_2"'] & Y_1\ot Y_2\ar[r,tick,"N_1\ot N_2"'] & Z_1\ot Z_2
\end{tikzcd}\quad\text{and} \quad
\begin{tikzcd}
 X_1\ot
X_2\ar[rr,tick,"\hid_{X_1}\ot\hid_{X_2}"]\ar[d,equal]\ar[drr,phantom,
"\scriptstyle\Downarrow\eta"] && X_1\ot X_2\ar[d,equal] \\
 X_1 \ot X_2 \ar[rr,tick,"\hid_{X_1\ot X_2}"'] && X_1\ot X_2\mathrlap{,}
\end{tikzcd}
\end{equation}
subject to axioms that make $\ot$ into an oplax double functor;
\item globular 2-morphisms
\begin{equation}\label{eq:structure2cells2}
\begin{tikzcd}
I_0\ar[d,equal]\ar[rr,tick,"I_1"]\ar[drr,phantom,
"\scriptstyle\Downarrow\delta"] && I_0\ar[d,equal] \\
 I_0\ar[r,tick,"I_1"'] & I_0\ar[r,tick,"I_1"'] & I_0
 \end{tikzcd}\quad\text{and} \quad
 \begin{tikzcd}
I_0\ar[d,equal]\ar[rr,tick,"I_1"]\ar[drr,phantom,
"\scriptstyle\Downarrow\iota"]  && I_0\ar[d,equal] \\
I_0\ar[rr,tick,"\hid_{I_0}"'] && I_0 \mathrlap{,}
 \end{tikzcd}
\end{equation}
subject to axioms that make $I$ into an oplax double functor;
\item two axioms ensuring that the associativity constraint $\alpha$ is a vertical
transformation between oplax double functors;
\item four axioms ensuring that the unit constraints $\lambda$ and $\rho$ are vertical
transformations between oplax double functors.
\end{itemize}
The above axioms are written explicitly in \cref{sec:appendix}.
We have a \emph{monoidal double category} when the tensor and unit are specified by double functors, rather than oplax double
functors; said another way, when each of the $2$-morphisms in~\eqref{eq:structure2cells1} and~\eqref{eq:structure2cells2} is invertible.
\end{defi}

What we call here an oplax monoidal double category is what is called simply a monoidal double category in
\cite[\S5.5]{Adjointfordoublecats}; it is equally well a pseudomonoid in the cartesian monoidal 2-category of double categories, oplax double
functors and vertical transformations.

\begin{rmk} We defined a monoidal double category to be an oplax monoidal double category
satisfying some additional properties; but these additional properties in fact allow us to simplify the structure further, as explained
in~\cite[Page~8]{ConstrSymMonBicatsFun}. Indeed, in a monoidal double category, the monoidal unit
$I_1$ of~$\dc{C}_1$ is always canonically isomorphic to $\hid_{I_0}$ via $\iota$; and it does no harm to assume that, in fact, $I_1$ is $\hid_{I_0}$
and $\iota$ is the identity $2$-morphism---which in turn forces $\delta = \ell_{\hid_{I_0}} = r_{\hid_{I_0}}$ for the
globular isomorphisms~\eqref{eq:aellr} for horizontal identities in $\dc{C}$.
As such, if in specifying a monoidal double category we follow these conventions, then we need only provide the invertible
structure 2-morphisms $\tau$ and $\eta$ satisfying the appropriate coherence axioms.
By contrast, in an oplax monoidal double category, none of the data are redundant: indeed, $\delta$ and $\iota$ as in~\cref{eq:structure2cells2} now
specify a \emph{comonad} structure on $I_1$ in $\mathcal{H}(\dc{C})$, see \cref{eq:deltaiotaaxioms}.

Moreover, notice that just as a double category is an internal pseudocategory in
the 2-category of small categories, functors and natural transformations,
an oplax monoidal  double category is an internal  pseudocategory in the $2$-category of monoidal categories,
lax monoidal functors and monoidal transformations for which the
source and target functors are strict monoidal.
\end{rmk}

As mentioned in the introduction, the notion of oplax monoidal structure will be exploited in future work in order to provide a general notion of \emph{commuting tensor product}, generalising the theory of~\cite{GarnerLopezFranco}, which will in particular recover the Boardman--Vogt tensor product of symmetric coloured operads and its extension to operadic bimodules in~\cite{DwyerW:BoardmanVtpo}. For these applications, it will be important that the oplax monoidal structure is \emph{normal} in the sense of the following definition.

\begin{defi}[Normal oplax monoidal double category]
\label{def:normality}
An oplax monoidal double category $\mathbb{C}$ is said to be \emph{normal} if:
\begin{enumerate}[(i)]
\item \label{normal-i} $I \colon \mathbf{1} \rightarrow \dc{C}$ is a (pseudo) double functor;
\item \label{normal-ii} for all objects $X_1,X_2 \in \dc{C}$, the following restricted oplax double functors are (pseudo) double functors:
  \begin{equation}
    \label{eq:onevariablefunctors}
    X_1 \ot (\thg) \colon \dc{C} \cong \mathbf{1} \times \dc{C} \xrightarrow{X_1 \times \dc{C}} \dc{C} \times \dc{C} \xrightarrow{\ot} \dc{C} \ \
    \text{and} \ \
    (\thg) \ot X_2 \colon \dc{C} \cong \dc{C} \times \mathbf{1} \xrightarrow{\dc{C} \times X_2} \dc{C} \times \dc{C} \xrightarrow{\ot} \dc{C}\mathrlap{ .}
  \end{equation}
\end{enumerate}
Said another way, both $2$-morphisms $\delta$ and $\iota$ in~\eqref{eq:structure2cells2} are invertible; while in~\eqref{eq:structure2cells1}, $\eta$
is invertible, and each~$\tau$ for which $M_1 = N_1 = \hid_{X_1}$ or $M_2 = N_2 = \hid_{X_2}$ is also invertible.
\end{defi}

We give now some examples of monoidal double categories and oplax monoidal double categories. The example of profunctors in \cref{ex:prof-mon} will be
fundamental for our application in \cref{sec:application}.

\begin{ex}[Duoidal categories]
\label{ex:oneobjectoplax}
Recall from \cref{ex:mon} that a monoidal category $(\ca{V}, \circ, J)$ is the same thing as a double category with
a single object and only identity vertical arrow. To equip this double category
with an oplax monoidal structure in the sense of \cref{def:mondoublecat}
is the same thing as equipping~$\ca{V}$ with additional structure making it into a duoidal category~\cite{Species}.
Explicitly, this amounts to providing a second monoidal structure $(\otimes, I)$ on $\ca{V}$, along with maps
\begin{gather*}
\xi\colon (Y_1 \circ X_1)\otimes(Y_2 \circ X_2) \to(Y_1 \otimes Y_2)\circ(X_1 \otimes X_2)\text{,} \\
\mu\colon J\ot J\to J\text{,} \qquad
\gamma\colon I\to I\circ I\text{,} \qquad \nu\colon I\to J \mathrlap{,}
\end{gather*}
satisfying appropriate axioms. Here, the interchange law $\xi$ corresponds to the square $\tau$ in~\cref{eq:structure2cells1}.

This oplax monoidal structure is a genuine monoidal structure whenever all of $\xi$, $\mu$, $\gamma$ and $\nu$ are invertible. In this case, by the
Eckmann--Hilton argument, the identity functor underlies a monoidal isomorphism $(\ca{V}, \ot, I) \rightarrow (\ca{V}, \circ, J)$, and the two
isomorphic monoidal structures are each \emph{braided}. Loosely, then, we may say that in this situation, $\ca{V}$ `is' a braided monoidal
category. In particular, if we merely start with a braided monoidal category $(\ca{V}, \ot, I)$, then it becomes duoidal on taking $\mathord{\circ} =
\mathord{\ot}$, $J=I$, $\nu = \id$, $\mu = r_I$, $\gamma = r_I^{\mi1}$, and $\xi$ the canonical constraint built from associativity and braiding
maps.

Returning to the general situation, the oplax monoidal structure on $\ca{V}$ \emph{qua} double category is normal if, and only if, the duoidal
structure on $\ca{V}$ is \emph{normal} meaning that $\nu$, $\gamma$ and $\mu$ are all invertible. Indeed, to say that the unit double functor of the oplax monoidal structure is pseudo is precisely to say that $\gamma$ and $\nu$ are invertible, which by the duoidal axioms implies the invertibility of $\mu$ also. So the duoidal structure on $\ca{V}$ is normal precisely when its oplax monoidal structure \emph{qua} double category satisfies \cref{def:normality}\ref{normal-i}. What is less obvious is that, in this one-object case, \cref{def:normality}\ref{normal-ii} is an automatic consequence of \cref{def:normality}\ref{normal-i}; but \ref{normal-ii} in this case amounts to the invertibility of $\xi$ when $X_1 = Y_1 = J$ or $X_2 = Y_2 = J$, and this follow from the oplax monoidality of the unit constraints for $\ot$.
\end{ex}

\begin{rmk}
  Looking at the previous example, the reader may wonder why we impose \cref{def:normality}\ref{normal-ii} at all, given that \cref{def:normality}\ref{normal-i} by itself gives a faithful ``many-object'' generalisation of the notion of normal duoidal category. The reason for imposing the extra condition is that it ensures companions in our double category are stable under tensoring by objects, which will be crucial when we come to show that normal oplax monoidal double categories give rise to normal oplax monoidal bicategories---see the proofs of \cref{lem:new} and \cref{thm:oplaxmonoidalbicat} below. A second justification for the condition comes from work-in-progress, which generalises the theory of commuting tensor products developed
  in~\cite{GarnerLopezFranco} to a ``many-object'' setting. In this generalisation, the normal duoidal categories used as an enrichment base in \emph{op.~cit.~} will be replaced with normal oplax monoidal double categories in the above sense---and, again, \cref{def:normality}\ref{normal-ii} will be necessary in order to make any progress with the theory.
\end{rmk}

\begin{ex}[Oplax monoidal structure on $\VMat$] \label{ex:mat-mon}
Recall from~\cref{ex:mat} that, for a monoidal category $(\ca{V}, \circ, J)$
in which the tensor product preserves coproducts in each variable,
we have a double category~$\VMat$ of $\ca{V}$-matrices. If $\ca{V}$ is further
equipped with a second monoidal structure $(\otimes, I)$ which also preserves coproducts in each variable, and data
as above making it into a duoidal category, then $\VMat$ acquires an
oplax monoidal structure.
The tensor product is given by the cartesian product of
sets and functions on the vertical level, and for $\ca{V}$-matrices $M_1 \colon
X_1 \horightarrow Y_1$ and $M_2 \colon X_2 \horightarrow Y_2$ the tensor $M_1 \otimes M_2 \colon X_1 \times
X_2 \horightarrow Y_1 \times Y_2$ defined by letting
\begin{displaymath}
 (M_1 \otimes M_2) ( (y_1,y_2),(x_1,x_2) )  =M_1(y_1, x_1) \otimes M_2(y_2,x_2)\mathrlap{ .}
\end{displaymath}
The monoidal unit is $I\colon1\horightarrow1$ with unique component
$I_{*,*}=I$. The structure cells~\cref{eq:structure2cells1} and~\cref{eq:structure2cells2} are formed
using the duoidal structure maps, with the most complex case being that the $2$-morphism $\tau \co (N_1 \circ M_1)\otimes( N_2 \circ M_2) \to (N_1
\ot N_2)\circ (M_1 \ot M_2)$ has components:
\begin{equation*}
  \begin{tikzcd}
    \big(\sum_{y_1} N_1(z_1,y_1) \circ M_1(y_1,x_1) \big)  \ot
    \bigl(\sum_{y_2} N_2(z_2,y_2) \circ M_2(y_2,x_2)\bigr) \ar[d,"\cong"]\\
    \sum_{y_1,y_2} \big( N_1(z_1,y_1) \circ M_1(y_1,x_1) \big)  \ot
    \bigl(N_2(z_2,y_2) \circ M_2(y_2,x_2)\bigr) \ar[d,"\sum_{y_1, y_2} \xi"]\\
    \sum_{y_1,y_2}\bigl(N_1(z_1,y_1) \otimes  N_2 (z_2,y_2) \bigr) \circ\bigl( M_1(y_1, x_1) \otimes M_2(y_2,x_2) \bigr)\mathrlap{ .}
  \end{tikzcd}
\end{equation*}
It is not hard to see that this oplax monoidal structure is normal precisely when $\ca{V}$ is normal as a duoidal category in the sense of the
preceding example, and that it is genuinely monoidal just when the duoidal structure of $\ca{V}$ comes from a braided monoidal structure.
\end{ex}

\begin{ex}[Monoidal structure on $\VProf$] \label{ex:prof-mon}
Let $\ca{V}$ be a braided monoidal category in
which the tensor product preserves colimits in each variable and recall the
double category $\VProf$ of \cref{ex:prof}.
This double category  admits a monoidal structure extending that of $\VMat$ in the braided case.
On objects and vertical $1$-cells, this is simply the monoidal structure of the 2-category $\VCat$.
On horizontal $1$-cells, given $\ca{V}$-profunctors $M_1 \colon X_1 \horightarrow Y_1$ and $M_2 \colon X_2 \horightarrow Y_2$, the
$\ca{V}$-profunctor $M_1 \ot M_2 \colon X_1 \ot Y_2 \horightarrow Y_1 \ot Y_2$ is given by
\begin{equation}
\label{equ:tensor-of-prof}
 (M_1 \ot M_2)\bigl( (y_1,y_2) (x_1,x_2)\bigr)=M_1(y_1,x_1)\ot M_2(y_2,x_2)\mathrlap{ ,}
\end{equation}
with a corresponding definition on $2$-morphisms.
The key structure isomorphism $\tau$ in~\cref{eq:structure2cells1} is formed using the
braiding of the tensor product on $\ca{V}$ and the fact that it preserves colimits in each variable.

One can carry out the construction above also when $\ca{V}$ is merely a duoidal category, thereby extending \cref{ex:mat-mon}, but
we shall not need this level of generality for our application in \cref{sec:application}.
\end{ex}

One of the main results of~\cite{ConstrSymMonBicatsFun}, building on~\cite{ConstrSymMonBicats, GarnerR:lowdsf}, is that under suitable
assumptions, the horizontal bicategory of a monoidal double category
is a monoidal bicategory in the sense of~\cite{CoherenceTricats}; see Theorem 1.1 of \emph{op.~cit.} In view of our application in
\cref{sec:application}, we would like a generalisation of this result which endows the horizontal bicategory of an oplax monoidal double category
with an oplax monoidal structure.

The first obstacle to be faced is the definition of oplax monoidal structure on a bicategory $\ca{K}$. As before, `oplax' refers to the strictness
of the tensor product functor, and so we might attempt the following naive adaptation of the usual notion of monoidal bicategory.
Firstly, we would require oplax functors of bicategories
$\ot\colon\ca{K}\times\ca{K}\to\ca{K}$
and $I\colon\mathbf{1}\to\ca{K}$; then
pseudonatural equivalences
\begin{equation}\label{eq:monbicatdata1}
  \begin{tikzcd}
    \ca{K}^3 \ar[r,"\mathord{\ot} \times 1"]
    \ar[d,"1 \times \mathord{\ot}"'] \ar[dr,phantom,"\Two \alpha"]&
    \ca{K}^2
    \ar[d,"\ot"] \\
    \ca{K}^2
    \ar[r,"\ot"'] & \ca{K}
  \end{tikzcd} \qquad
  \begin{tikzcd}[column sep=small]
    & \ca{K}^2 \ar[dr,"\ot"] \ar[d,phantom,"\Two \lambda"] \\
    \ca{K} \ar[ur,"I \times 1"]
    \ar[rr,"1"'] & {} & \ca{K}
  \end{tikzcd} \quad \text{and} \quad
  \begin{tikzcd}[column sep=small]
    & \ca{K}^2 \ar[dr,"\ot"] \ar[d,phantom,"\Two \rho"] \\
    \ca{K} \ar[ur,"1 \times I"]
    \ar[rr,"1"'] & {} & \ca{K}
  \end{tikzcd}
\end{equation}
giving the associativity and unit constraints; then invertible modifications
\begin{equation}
  \label{eq:monbicatdata2}
  \begin{tikzcd}[column sep=small]
    & \ca{K}^3 \ar[rr,"\ot \times 1"] \ar[dr,phantom,"\Two \alpha \times 1"] &  &
    \ca{K}^2 \ar[dr,"\ot"] \ar[dd,phantom,"\Two \alpha"] \\
    \ca{K}^4 \ar[dr,"1 \times 1 \times \ot"']
    \ar[rr,"1 \times \ot \times 1"]
    \ar[ur,"\ot \times 1 \times 1"] &  &
    \ca{K}^3 \ar[ur,"\ot \times 1" description] \ar[dr,"1 \times \ot" description] \ar[dl,phantom,"\Two 1 \times \alpha"] & &
    \ca{K} \\
    & \ca{K}^3 \ar[rr,"1 \times \ot"'] & {} &
    \ca{K}^2 \ar[ur,"\ot"']
  \end{tikzcd}
  \begin{tikzcd}[column sep=1em]
    \ar[r,triple,"\pi"] & {}
  \end{tikzcd}
  \begin{tikzcd}[column sep=small]
    & \ca{K}^3 \ar[rr,"\ot \times 1"] \ar[dr,"1 \times \ot" description] \ar[dd,phantom,"="] & &
    \ca{K}^2 \ar[dr,"\ot"] \\
    \ca{K}^4 \ar[dr,"1 \times 1 \times \ot"']
    \ar[ur,"\ot \times 1 \times 1"] &  &
    \ca{K}^2 \ar[rr,"\ot"] \ar[dr,phantom,"\Two \alpha"] \ar[ur,phantom,"\Two \alpha"] & &
    \ca{K} \\
    & \ca{K}^3 \ar[rr,"1 \times \ot"']
    \ar[ur,"\ot \times 1" description] & &
    \ca{K}^2 \ar[ur,"\ot"']
  \end{tikzcd}
\end{equation}
\begin{equation}
  \label{eq:monbicatdata3}
\begin{tikzcd}[column sep=small]
    & \ca{K}^3 \ar[dr,"\ot \times 1"]
    \ar[d,phantom,"\Two \rho \times 1"] \\
    \ca{K}^2 \ar[ur,"1 \times I \times 1"]
    \ar[rr,"1"] \ar[d,"\ot"'] \ar[drr,phantom,"="] & {} &
    \ca{K}^2
    \ar[d,"\ot"] \\
    \ca{K} \ar[rr,"1"'] & & \ca{K}
  \end{tikzcd}
  \begin{tikzcd}[column sep=1em]
    \ar[r,triple,"\mu"] & {}
  \end{tikzcd}
  \begin{tikzcd}[column sep=small]
    & \ca{K}^3 \ar[dr,"\ot \times 1"]
    \ar[d,"1 \times \ot" description] \\
    \ca{K}^2 \ar[ur,"1 \times I\times 1"]
    \ar[r,phantom,"\Two 1 \times \lambda"] \ar[d,"1"'] &
    \ca{K}^2 \ar[dr,"\ot" description]
    \ar[r,phantom,"\Two \alpha"] \ar[d,phantom,"="] &
    \ca{K}^2
    \ar[d,"\ot"] \\
    \ca{K} \ar[ur,"1" description] \ar[rr,"\ot"'] & {} & \ca{K}
  \end{tikzcd}
\end{equation}
\begin{equation}
  \label{eq:monbicatdata4}
\begin{tikzcd}[column sep=small]
    & \ca{K}^3 \ar[dr,"\ot \times 1"]
    \ar[d,phantom,"\Two \lambda \times 1"] \\
    \ca{K}^2 \ar[ur,"I \times 1 \times 1"]
    \ar[rr,"1"] \ar[d,"\ot"'] \ar[drr,phantom,"="] & {} &
    \ca{K}^2
    \ar[d,"\ot"] \\
    \ca{K} \ar[rr,"1"'] & & \ca{K}
  \end{tikzcd}
  \begin{tikzcd}[column sep=1em]
    \ar[r,triple,"L"] & {}
  \end{tikzcd}
  \begin{tikzcd}[column sep=small]
    & \ca{K}^3 \ar[dr,"\ot \times 1"]
    \ar[d,"1 \times \ot" description] \\
    \ca{K}^2 \ar[ur,"I \times 1 \times 1"]
    \ar[r,phantom,"="] \ar[d,"\ot"'] &
    \ca{K}^2 \ar[dr,"\ot" description]
    \ar[r,phantom,"\Two \alpha"] \ar[d,phantom,"\Two \lambda"] &
    \ca{K}^2
    \ar[d,"\ot"] \\
    \ca{K} \ar[ur,"I \times 1" description] \ar[rr,"1"'] & {} & \ca{K}
  \end{tikzcd} \ \text{and} \
  \begin{tikzcd}[column sep=small]
    & \ca{K}^3 \ar[dr,"\ot \times 1"]
    \ar[d,"\ot \times 1" description] \\
    \ca{K}^2 \ar[ur,"1 \times 1 \times I"]
    \ar[d,"\ot"']  \ar[r,phantom,"="] &
    \ca{K}^2 \ar[dr,"\ot"] \ar[r,phantom,"="]  \ar[d,phantom,"\Two \rho"]  &
    \ca{K}^2
    \ar[d,"\ot"] \\
    \ca{K} \ar[rr,"1"'] \ar[ur,"1 \times I" description] & {} & \ca{K}
  \end{tikzcd}
  \begin{tikzcd}[column sep=1em]
    \ar[r,triple,"R"] & {}
  \end{tikzcd}
  \begin{tikzcd}[column sep=small]
    & \ca{K}^3 \ar[dr,"\ot \times 1"]
    \ar[d,"1 \times \ot" description] \\
    \ca{K}^2 \ar[ur,"1 \times 1 \times \lambda"]
    \ar[r,phantom,"\Two 1 \times \rho"] \ar[d,"1"'] &
    \ca{K}^2 \ar[dr,"\ot" description]
    \ar[r,phantom,"\Two \alpha"] \ar[d,phantom,"="] &
    \ca{K}^2
    \ar[d,"\ot"] \\
    \ca{K} \ar[ur,"1" description] \ar[rr,"\ot"'] & {} & \ca{K}
  \end{tikzcd}
\end{equation}
witnessing the pseudo-coherence of the constraint cells; and, finally, the appropriate coherence axioms for these pseudo-coherences, as found, for
example, in~\cite[\S A.1]{Mccruddenthesis}.

The problem with this definition can be seen in~\eqref{eq:monbicatdata2} above. In the domain of $\pi$ we have, among other things, the pseudonatural
transformation $\alpha \times 1$ whiskered by the oplax functor $\ot \colon \ca{K}^2 \rightarrow \ca{K}$. However, such a composition
does
\emph{not} yield another pseudonatural transformation, nor even a lax or oplax transformation. 
So $\pi$ is not well-defined; and similar issues
arise for $\mu$, $L$ and $R$ in~\eqref{eq:monbicatdata3} and~\eqref{eq:monbicatdata4}.

We will resolve this issue by imposing a further constraint on the components of the pseudonatural equivalences $\alpha, \lambda, \rho$ of 
\cref{eq:monbicatdata1} which we will term \emph{centrality}, loosely inspired by the nomenclature of~\cite{Premonoidal}. Centrality of the components 
of a pseudonatural transformation $\gamma$ will ensure that composites of the form $\gamma \otimes 1 \defeq \otimes \circ (\gamma \times 1)$ and $1 
\otimes \gamma \defeq \otimes \circ (1 \times \gamma)$ are well-defined pseudonatural transformations; and, furthermore, that the same pseudonaturality 
holds for any iterated tensorings such as $((1 \otimes \gamma) \otimes 1) \otimes 1$. Applied to the case of $\alpha$, $\lambda$ and $\rho$, this will 
ensure that the transformations appearing in~\crefrange{eq:monbicatdata2}{eq:monbicatdata4}, as well as all of those appearing in the corresponding 
coherence axioms, make sense.

\begin{defi}\label{def:central}
Let $\ca{K}$ be a bicategory endowed with an oplax functor $(\otimes, \tau, \eta) \colon\ca{K}\times\ca{K}\to\ca{K}$, an oplax functor $(I, \delta, \iota) \colon 1 \rightarrow 
\ca{K}$ and pseudonatural equivalences as in~\eqref{eq:monbicatdata1}. 
A 1-cell $f\colon X\to Y$ of $\ca{K}$ is said to be 
\emph{central} when for all maps $g \colon X' \rightarrow X$, $h \colon Y \rightarrow Y'$, $k \colon U \rightarrow W$ and $\ell \colon V \rightarrow 
Z$, the following composite oplax structure cells are invertible:
\begin{equation}\label{eq:central}
\begin{tikzcd}[row sep=.2in,column sep=.1in] & {} \\
(U \otimes X) \otimes V\ar[rr,"((k \ot h) \circ (\hid \ot f))\ot \ell" description]\ar[rr,bend left,"(k \ot (h\circ f))\ot \ell"]\ar[dr,bend 
right=5,"(\hid \ot f)\ot\hid"'] & \ar[d,phantom,"\Two \tau"]\ar[u,phantom,"\Two \tau \otimes \vid_\ell"]& (W \otimes Y') \otimes Z \\
& (U \otimes Y) \otimes V \ar[ur,bend right=5,"(k\ot h) \ot \ell"'] &
\end{tikzcd}\ \ 
\begin{tikzcd}[row sep=.2in,column sep=.1in] & {}\\
(U \otimes X') \otimes V \ar[rr,"((\hid \ot f) \circ (k \ot g))\ot \ell" description]\ar[rr,bend left,"(k \ot (f\circ g))\ot \ell"]\ar[dr,bend 
right=5,"(k\ot g) \ot \ell"'] &\ar[d,phantom,"\Two \tau"]\ar[u,phantom,"\Two \tau \otimes \vid_\ell"]& (W \otimes Y) \otimes Z \rlap{ .} \\
& (W \otimes X) \otimes Z \ar[ur,bend right=5,"(\hid \ot f)\ot\hid"'] &
\end{tikzcd}
\end{equation}
\end{defi}
Note that in~\eqref{eq:central}, we consider three-fold tensor products bracketed to the \emph{left}. We could equally have chosen to bracket to the \emph{right}, but this would make no difference, since composing~\eqref{eq:central} with the components of $\alpha$ and their pseudoinverses would yield invertibility of the corresponding cells for the other bracketing.

Now, by taking $k = \id_I$ or $\ell = \mathrm{id}_I$ in~\eqref{eq:central}, and composing with the 
components of $\lambda$ or $\rho$ and their pseudoinverses, we obtain the invertibility of oplax constraints of the 
following forms:
\begin{equation*}
\begin{tikzcd}[row sep=.1in,column sep=-0.05in] & {} \ar[dd,phantom,"\Two \tau"]\\
X \otimes V\ar[rr,bend left,"(h\circ f)\ot \ell"]\ar[dr,bend right=5,"f\ot\hid"'] & & Y' \otimes Z \\
& Y \otimes V \ar[ur,bend right=5,"h \ot \ell"'] &
\end{tikzcd}\ 
\begin{tikzcd}[row sep=.1in,column sep=-.05in] & {} \ar[dd,phantom,"\Two \tau"]\\
  X' \otimes V\ar[rr,bend left,"(f\circ g)\ot \ell"]\ar[dr,bend right=5,"g \ot \ell"'] && Y \otimes Z\\
& X \otimes Z \ar[ur,bend right=5,"f\ot\hid"'] &
\end{tikzcd}\ 
\begin{tikzcd}[row sep=.1in,column sep=-0.05in] & {} \ar[dd,phantom,"\Two \tau"] \\
  U \otimes X\ar[rr,bend left,"k \ot (h\circ f)"]\ar[dr,bend right=5,"\hid \ot f"'] & & W \otimes Y' \\
  & U \otimes Y \ar[ur,bend right=5,"k\ot h"'] &
\end{tikzcd}\ 
\begin{tikzcd}[row sep=.1in,column sep=-0.05in] & {} \ar[dd,phantom,"\Two \tau"]\\
  U \otimes X'\ar[rr,bend left,"k \ot (f\circ g)"]\ar[dr,bend right=5,"k\ot g"'] && W \otimes Y\rlap{ .} \\
  & W \otimes X \ar[ur,bend right=5,"\hid \ot f"'] &
\end{tikzcd}
\end{equation*}
Because of this, if $\gamma \colon F \Rightarrow G \colon \ca{L} \rightarrow \ca{K}$ is a pseudonatural transformation with central components, then both $\gamma \otimes 1$ and $1 \otimes \gamma$ will also be pseudonatural; for example, in the case of $\gamma \otimes 1$,  the pseudonaturality with respect to $f \co X \to Y$ and $g \co C \to D$
 is witnessed by the invertible $2$-cell
\begin{displaymath}
\begin{tikzcd}[sep=0.8in]
FX\ot C\ar[r,"Ff\ot g"] \ar[d,"\gamma_X\ot\hid_C"']\ar[dr,bend left,"(\gamma_Y\circ Ff)\ot g"description]\ar[dr,bend 
left=50,phantom,"\Two \tau^{-1}"]\ar[dr,bend right,"(Gf\circ\gamma_X)\ot g"description]\ar[dr,phantom,"\scriptstyle\Two \gamma_f \otimes \hid_{g}"]
\ar[dr,bend right=50,phantom,"\Two \tau"']& FY\ot 
D\ar[d,"\gamma_Y\ot\hid_D"] \\
GX\ot C\ar[r,"Gf\ot g"'] & GY\ot D\rlap{ .}
\end{tikzcd}
\end{displaymath}

In a similar way, the general form of~\eqref{eq:central} implies that $(1 \otimes \gamma) \otimes 1$ is also pseudonatural; note that this does 
\emph{not} seem to follow from the pseudonaturality of $1 \otimes \gamma$ and $\gamma \otimes 1$. However, once we have pseudonaturality of $(1 
\otimes \gamma) \otimes 1$, we obtain \emph{a fortiori} that of, say, $(1 \otimes \gamma) \otimes (1 \otimes 1)$ and so by composing with the 
equivalence components of $\alpha$, the pseudonaturality of $((1 \otimes \gamma) \otimes 1) \otimes 1$. By following this pattern, we see that  
\emph{any} tensoring of $\gamma$ with identity pseudonatural transformations will again be pseudonatural.

In particular, if we require the pseudonatural transformations $\alpha$, $\lambda$ and $\rho$ to themselves have central components, then we see that 
every $2$-cell pasting which appears in the axioms~\crefrange{eq:monbicatdata2}{eq:monbicatdata4} will be a well-defined pseudonatural transformation, 
and likewise for the pastings appearing in the coherence axioms.  Thus, we are justified in giving:

\begin{defi}
  \label{def:oplaxbicat}
  Let $\ca{K}$ be a bicategory. An \emph{oplax monoidal structure} on $\ca{K}$ consists of:
  \begin{itemize}
  \item an oplax functor of bicategories $\ot \colon\ca{K}\times\ca{K}\to\ca{K}$;
\item an oplax homomorphism $I\colon\one\to\ca{K}$;
\item pseudonatural equivalences $\alpha, \lambda, \rho$ as in~\eqref{eq:monbicatdata1}, whose components are central;
\item invertible modifications $\pi, \mu, L, R$ as in~\eqref{eq:monbicatdata2}--\eqref{eq:monbicatdata4};
\end{itemize}
satisfying the coherence axioms for a monoidal bicategory as found, for example, in~\cite[\S A.1]{Mccruddenthesis}.
Like in \cref{def:normality}, we say that the oplax monoidal structure on $\ca{K}$ is \emph{normal} if $I \colon 1 \rightarrow \ca{K}$ is a homomorphism of bicategories, and $\otimes$ is pseudo in each variable, {\em i.e.}~for each $X,Y \in \ca{K}$ the oplax functors $X \ot (\thg) \colon \ca{K} \rightarrow \ca{K}$ and 
$(\thg) \ot Y \colon \ca{K} \rightarrow \ca{K}$ are homomorphisms of bicategories.
\end{defi}

We now explain how a normal oplax monoidal double category $\dc{C}$ gives rise to a normal oplax monoidal bicategory. First of all, the oplax double functor $\ot$ and the pseudo double functor $I$ induce
functors on the horizontal bicategory
\begin{equation}
  \label{eq:inducedoplaxmorphism}
  \ot\colon\ca{H}(\dc{C})\times\ca{H}(\dc{C})\to\ca{H}(\dc{C}), \quad I\colon\mathbf{1}\to\ca{H}(\dc{C})
\end{equation}
which we denote with the same symbol. Here, as per \cref{lem:dbl-to-bicat-functor}, $\ot$ is an oplax functor of bicategories (which is pseudo in each variable) and $I$ is a homomorphism of
bicategories.
If now we assume that the components of the invertible vertical transformations $\alpha$, $\lambda$ and $\rho$ associated to the monoidal structure on $\dc{C}$ have companions, then they will induce 
pseudonatural equivalences
\begin{equation}
  \label{eq:inducedcoherence}
  \comp{\alpha} \co \mathord{\ot}\circ(1\times\mathord{\ot}) \Rightarrow \mathord{\ot}\circ(\mathord{\ot}\times1) \qquad 
\comp{\lambda} \co \mathord{\ot}\circ(I\times 1) \Rightarrow 1 \qquad \comp{r} \co 1 \Rightarrow \mathord{\ot}\circ(1 \times I)
\end{equation}
between oplax functors of bicategories, according to \cref{lem:pseudonatadjequiv} 
and since $\ca{H}$ is functorial. To proceed further, we need the components of $\comp{\alpha},\comp{\lambda}$ and $\comp{\rho}$ to be central. This will be a consequence of the following lemma. 
Note that \emph{normality} of the oplax monoidal structure on $\mathbb{C}$ is important 
for the proof. It is not clear to us if the corresponding result without it would hold; however, since normality will be present in our applications, we have not pursued this point any further.


\begin{lem}\label{lem:new}
Let $\dc{C}$ be a normal oplax monoidal double category in which the vertical $1$-cells giving associativity, left and right unit constraints have companions. If $f\colon X\to Y$ is any vertical isomorphism in~$\dc{C}$ that has a companion, then $\wh{f}\colon 
X\tickar Y$ is central in the horizontal bicategory $\ca{H}(\dc{C})$ with respect to the structure of~\cref{eq:inducedoplaxmorphism} and~\cref{eq:inducedcoherence}.
\end{lem}

\begin{proof} Fix $f \co X \to Y$ with companion $\wh{f} \co X \tickar Y$ as in the statement.
To check the conditions in \cref{def:central}, we must show, for any horizontal $1$-cells
 $g \colon X' \tickar X$, $h \colon Y \tickar Y'$, $k \colon U \tickar W$ and $\ell \colon V \tickar Z$, 
that the two composite 2-cells in~\cref{eq:central} are invertible.
For the one on the left-hand side, we must show invertibility of the globular $2$-morphism in
\begin{equation}\label{eq:tautau1}
\begin{tikzcd}[column sep=.6in]
(U\ot X)\ot V\ar[d,equal]\ar[rr,tick,"(k\ot(h\circ \comp{f}))\ot \ell"]\ar[drr,phantom,"\Two\tau(\tau\ot\vid)"] && (W\ot Y')\ot Z\ar[d,equal] \\
(U\ot X)\ot V\ar[r,tick,"(\hid\ot\comp{f})\ot\hid"'] & (U\ot Y)\ot V\ar[r,tick,"(k\ot h)\ot \ell"'] & (W\ot Y')\ot Z
\mathrlap{.}
\end{tikzcd}
\end{equation}
Because $\otimes$ is pseudo in each variable and $f$ has companion $\comp{f}$, it follows that $(\vid \ot f) \ot \vid$ has companion $(\hid \ot \comp{f}) \ot \hid$. Thus, $\tau(\tau \otimes 1)$ in~\eqref{eq:tautau1}
is invertible if and only if its companion transpose
\begin{equation*}
\begin{tikzcd}[column sep=.7in]
(U\ot X)\ot V\ar[d,"(\vid\ot f)\ot\vid"']\ar[r,tick,"(k\ot(h\circ \comp{f}))\ot \ell"]\ar[dr,phantom,"\Two \comp{\tau(\tau \otimes \vid)}"] & 
(W\ot Y')\ot Z\ar[d,equal] \\
(U\ot Y)\ot V\ar[r,tick,"(k\ot h)\ot \ell"'] & (W\ot Y')\ot Z 
\end{tikzcd}
\end{equation*}
is invertible. We claim that this companion transpose is actually
given by the following tensor product in~$\dc{C}$:
\begin{equation}\label{eq:thistranspose}
\Biggl(\begin{tikzcd}[column sep=.7in]
U\ar[d,equal]\ar[r,tick,"k"]\ar[dr,phantom,"\Two \vid_k"] & 
W\ar[d,equal] \\
U\ar[r,tick,"k"'] & W
\end{tikzcd} \quad \otimes \quad
\begin{tikzcd}[column sep=.7in]
  X\ar[d,"f"']\ar[r,tick,"\comp{f}"]\ar[dr,phantom,"\Two p_1"] &
  Y\ar[d,equal]\ar[r,tick,"h"]\ar[dr,phantom,"\Two \vid_h"] &
Y'\ar[d,equal] \\
Y\ar[r,tick,"\hid"'] & Y \ar[r,tick,"h"'] & Y'
\end{tikzcd}\Biggr)
\quad \ot \quad
\begin{tikzcd}[column sep=.7in]
V\ar[d,equal]\ar[r,tick,"\ell"]\ar[dr,phantom,"\Two \vid_\ell"] & 
Z\ar[d,equal] \\
V\ar[r,tick,"\ell"'] & Z \rlap{.}
\end{tikzcd}
\end{equation}
Note that this is clearly invertible, since $f$ is an isomorphism and so $p_1$ is
invertible.
To show that \eqref{eq:thistranspose} is a transpose companion of \eqref{eq:tautau1}
it suffices to use the explicit definition of transposition of a 2-morphism.
Indeed, pasting $(\vid \otimes p_2) \otimes \vid$ to the left of \cref{eq:thistranspose} and using the (right-hand side) axiom \cref{eq:piaxioms} and naturality of the components of $\tau$, we obtain the 2-cell \cref{eq:tautau1}. 

It is possible to verify that the composite 2-cell on the right-hand side of~\cref{eq:central} is 
invertible by a similar argument, but pasting with $(\vid \otimes p_1^{-1}) \otimes \vid$ instead of $(\vid \otimes p_2) \otimes \vid$. 
\end{proof}

\begin{thm}\label{thm:oplaxmonoidalbicat}
If $\dc{C}$ is a normal oplax monoidal double category in which the vertical $1$-cells giving associativity, left and right unit constraints have companions, then the
horizontal bicategory $\ca{H}(\dc{C})$ inherits a normal oplax monoidal structure with underlying data~\cref{eq:inducedoplaxmorphism} and~\cref{eq:inducedcoherence}.
\end{thm}

\begin{proof}
Since the pseudonatural transformations $\comp{\alpha}$, $\comp{\lambda}$ and $\comp{\rho}$ of~\eqref{eq:inducedcoherence} have as their components the horizontal companions of vertical isomorphisms, we can apply \cref{lem:new} to see that these components are all central in the sense of 
\cref{def:central}.
We now need to provide the four invertible modifications of \cref{def:oplaxbicat} for $\ca{H}(\dc{C})$ to have the structure of a normal oplax
monoidal bicategory. The components of \eqref{eq:monbicatdata2}--\eqref{eq:monbicatdata4} are of the
form
\begin{displaymath}
\begin{tikzpicture}[baseline=3ex]
\node[regular polygon, regular polygon sides=5,minimum width=42mm]
(PG) {$\Two{\pi}$}
(PG.corner 1) node (PG1) {$(X_1 \ot X_2)\ot(X_3 \ot X_4)$}
(PG.corner 2) node (PG2) {$((X_1\ot X_2)\ot X_3)\ot X_4$}
(PG.corner 3) node (PG3) {$\mathllap{(X_1 \ot (X_2 \ot {}}X_3))\ot X_4$}
(PG.corner 4) node (PG4) {$X_1\ot((X_2 \mathrlap{{}\ot X_3)\ot X_4)}$}
(PG.corner 5) node (PG5) {$X_1\ot(X_2\ot(X_3\ot X_4))$};
\draw
(PG1) edge[->] node[above right] {$\scriptstyle \comp{\alpha}$} (PG5)
(PG2) edge[->] node[above left] {$\scriptstyle \comp{\alpha}$} (PG1)
(PG2) edge[->] node[left] {$\scriptstyle\comp{\alpha}\otimes \hid$} (PG3)
(PG3) edge[->] node[below] {$\scriptstyle\comp{\alpha}$} (PG4)
(PG4) edge[->] node[right] {$\scriptstyle \hid\ot\comp{\alpha}$} (PG5);
\end{tikzpicture}
\end{displaymath}
\begin{displaymath}
\begin{tikzcd}[row sep=.2in]
(X_1\ot I)\ot X_2\ar[rr,"\comp{\alpha}"]\ar[dr,"\comp{\rho}\ot\hid"']\ar[rr, bend right=10, phantom, "\Two\mu"] && X_1\ot(I\ot X_2)\ar[dl,"\hid\ot
\comp{\lambda}"] \\
& X_1\ot X_2 & \\
(X_1\ot X_2)\ot I\ar[rr,"\comp{\alpha}"]\ar[dr,"\comp{\rho}"']\ar[rr, bend right=10, phantom, "\Two R"] && X_1\ot(X_2\ot \lambda)\ar[dl,"\hid\ot
\comp{\rho}"] \\
& X_1\ot X_2 & \\
(I\ot X_1)\ot X_2\ar[rr,"\comp{\alpha}"]\ar[dr,"\comp{\lambda}\ot\hid"']\ar[rr, bend right=10, phantom, "\Two L"] && I\ot(X_1\ot
X_2)\ar[dl,"\comp{\lambda}"] \\
& X_1\ot X_2 &
\end{tikzcd}
\end{displaymath}
Notice that the two sides in each case are companions of the corresponding sides of the pentagon axiom, the triangle axiom and two known
equations for the ordinary monoidal category $\dc{C}_0$, due to \cref{prop:companionomnibus,lem:Ff}. For example, since $\ot$ is a pseudo double functor in
each variable, each $(\thg)\ot X$ and $Y\ot(\thg)$ preserves companions thus $\comp{\alpha}\ot\hid_X$
is canonically a companion of $\alpha\ot\vid_X$. As a result, we take $\pi,\mu, L, R$ to be the unique isomorphisms between companions of the same
vertical 1-cells. It can then be verified that these invertible cells form a modification between pseudonatural
transformations of oplax double functors by \cite[Lemma~4.8]{ConstrSymMonBicats}\footnote{Although the cited result refers
to vertical transformations between \emph{pseudo} double functors, the proof is identical in the oplax setting.}.

Finally, the three equations that relate those $\pi,\mu, L, R$ can be checked in exactly the same way as in the proof of
\cite[Theorem~5.1]{ConstrSymMonBicats}. In more detail, the domain and codomain of the pasted 2-cells involved in the equations are companions of
the
same isomorphism in $\dc{C}_0$, namely the unique $(((X_1\ot X_2)\ot X_3)\ot X_4)\ot X_5 \cong X_1\ot(X_2\ot(X_3\ot(X_4\ot X_5)))$ as well as the associator $(X_1\ot X_2)\ot X_3 \cong X_1\ot(X_2 \ot X_3)$. Using a collection of technical lemmas \cite[Lemma~3.11,\ 3.14,\ 3.15,\ 3.19,\ 4.10]{ConstrSymMonBicats} concerning the
composition as well as the tensoring of the canonical isomorphisms between companions (the latter adjusted in the normal oplax monoidal case in a
straightforward way), we deduce that there can only be a unique invertible 2-cell inside each one of the diagrams which is compatible with the companion data, hence the
equations must hold.
\end{proof}

\section{Maps of monoidal double categories}
\label{sec:mon-dbl-fun}

For our development in \cref{sec:dbl-monad,sec:mon-dbl-monad}, we will need results concerning both double monads and pseudomonoidal double monads. It turns out that many of these results can be proved uniformly across the two cases, by exhibiting both kind of structure as monoids in suitable endofunctor double categories. This is much as ordinary monads and monoidal monads can be seen as monoids in appropriate endofunctor categories.
In order to do this for the case of pseudomonoidal double monads, we need to construct a suitable double category of (lax) monoidal double functors
and monoidal transformations. While the notions of lax monoidal double functor and monoidal horizontal transformation (recalled in
\cref{def:monoidaldoublefunctor} and~\cref{def:monhortransf} below) are as expected, it turns out that in our motivating examples, the vertical
transformations which we need are not monoidal in the obvious way but only \emph{pseudomonoidal}. While this
may seem an innocuous change, it adds an additional layer of subtlety to our development,
very much in analogy with what happens in the purely 2-categorical setting~\cite{HylandPower}.

We begin with the notion of a  lax monoidal double functor. If we view monoidal double categories as pseudomonoids in a $2$-category of double categories, double functors and vertical transformations, then the lax monoidal functors are simply the lax morphisms of pseudomonoids. This definition can also be found in \cite[Definition~2.14]{ConstrSymMonBicatsFun}, though note that there, the (invertible) structure maps $\tau$ and $\eta$ of a monoidal 
double category (\cref{def:oplaxdoublecat}) are oriented in the opposite direction. 

\begin{defi}[Lax monoidal double functor]
\label{def:monoidaldoublefunctor}
 Let $\dc{C}$ and $\dc{D}$ be monoidal double categories. A \emph{lax monoidal
double functor} $F\colon\dc{C}\to\dc{D}$ is a (pseudo) double functor equipped with:
\begin{itemize}
\item  a vertical transformation $F^2 \co \mathord{\ot} \circ (F \times F) \Rightarrow F \circ \mathord{\ot}$,
 whose vertical $1$-cell components we denote by $F^2_{X_1,X_2}\colon FX_1\ot FX_2\to F(X_1\ot X_2)$, and whose $2$-morphism components we denote by
\begin{equation}\label{eq:F2}
\begin{tikzcd}[column sep=.6in]
FX_1\ot FX_2 \ar[r,tick,"FM\ot
FN"]\ar[d,"{ F^2_{X_1,X_2}}"']\ar[dr,phantom,"\Two F^2_{M,N}"] &
FY_1 \ot FY_2\ar[d," F^2_{Y_1,Y_2}"] \\
F(X_1\ot X_2)\ar[r,tick,"F(M\ot N)"'] & F(Y_1\ot Y_2)\mathrlap{;}
\end{tikzcd}
\end{equation}
\item a vertical transformation $F^0 \co I_\dc{D} \Rightarrow F \circ I_\dc{C}$,
 whose vertical $1$-cell component we denote by $F^0\colon I\to FI$ and whose $2$-morphism component we
 denote by
 \begin{equation}
  \label{eq:laxmon2cells-2}
 \twocell{I}{\hid_I}{I}{ F^0}{FI\mathrlap{;}}{F(\hid_{I})}{FI}{ F^0}{}
\end{equation}
\end{itemize}
subject to axioms expressing that the vertical $1$-cells $F^2_{X_1, X_2} \co FX_1 \ot FX_2 \rightarrow F(X_1 \ot X_2)$ and $F^0 \colon I \rightarrow
FI$ endow $F_0 \colon \dc{C}_0 \rightarrow \dc{D}_0$ with the structure of a lax monoidal functor, and that the $2$-morphisms of~\eqref{eq:F2}
and~\eqref{eq:laxmon2cells-2} do the same for $F_1 \colon \dc{C}_0 \rightarrow \dc{D}_0$.
\end{defi}

The reader will notice that we have not named the $2$-morphism in~\eqref{eq:laxmon2cells-2}. This is because its definition is forced: for indeed, since $F^0$ is a vertical transformation between double functors, the axiom~\cref{eq:verticaltransfax1} causes~\eqref{eq:laxmon2cells-2} to be equal to $\hid_{F^0}$ followed by the unit structure isomorphism of $F$.

We now turn to monoidal transformations between lax monoidal double functors. We begin with the horizontal case, which is as expected, though we could not find it in the literature.
\begin{defi}[Monoidal horizontal transformation]\label{def:monhortransf}
Let $F, G \colon \mathbb{C} \rightarrow \dc{D}$ be
lax monoidal double functors.
A \emph{monoidal horizontal transformation}
$\beta \colon F \ticktwoar G$ is a horizontal transformation endowed with cells
  \begin{equation}\label{eq:1}
    \cd[@C+2em]{
      F{X_1} \otimes F{X_2}
      \ar|@{|}[r]^-{\beta_{X_1} \otimes \beta_{X_2}}
      \ar[d]_-{ F^2_{X_1,{X_2}}}
      \dtwocell{dr}{\beta^2_{X_1,{X_2}}} &
      G{X_1} \otimes G{X_2}
      \ar[d]^-{ G^2_{X,{X_2}}} \\
      F({X_1} \otimes {X_2})
      \ar|@{|}[r]_-{\beta_{{X_1} \otimes {X_2}}} &
      G({X_1} \otimes {X_2})
    } \qquad \text{and} \qquad \cd[@C+2em]{
      I
      \ar|@{|}[r]^-{\hid_{I}}
      \ar[d]_-{ F^0}
      \dtwocell{dr}{\beta^0} &
      I
      \ar[d]^-{ G^0} \\
      FI
      \ar|@{|}[r]_-{\beta_{I}} &
      GI
    }
  \end{equation}
which, firstly, make $\beta_{(\thg)} \colon \mathbb{C}_0 \rightarrow \dc{D}_1$ into a
lax monoidal functor; in other words, such that the naturality condition
\begin{displaymath}
\begin{tikzcd}
FX_1 \ot
F{X_2}\ar[r,tick,"\beta_X\ot\beta_{X_2}"]\ar[d,"F^2_{X_1,{X_2}}"']\ar[dr,phantom,
"\Two\beta^2_{X_1, {X_2}}"] & GX\ot G{X_2}\ar[d,"G^2_{X_1,{X_2}}"] \\
F(X_1 \ot {X_2})\ar[r,tick,"\beta_{X_1 \ot {X_2}}"]\ar[d,"F(f\ot
g)"']\ar[dr,phantom,"\Two\beta_{f\ot g}"'] &
G(X_1\ot {X_2})\ar[d,"G(f\ot g)"] \\
F(X_1'\ot {X_2}')\ar[r,tick,"\beta_{X_1 '\ot {X_2}'}"'] & G(X_1'\ot {X_2}')
\end{tikzcd}=
\begin{tikzcd}
FX_1\ot F{X_2}\ar[r,tick,"\beta_{X_1} \ot\beta_{X_2}"]\ar[d,"Ff\ot Fg"']
\ar[dr,phantom,"\Two\beta_f\ot\beta_g"] &
GX_1 \ot G{X_2}\ar[d,"Gf\ot Gg"] \\
FX'_1 \ot F{X_2}'\ar[r,tick,"\beta_{X_1'}\ot\beta_{{X_2}'}"]\ar[d,"F^2_{X_1',{X_2}'}"']
\ar[dr,phantom,"\Two\beta^2_{X_1', {X_2}'}"] &
GX_1'\ot G{X_2}'\ar[d,"G^2_{X_1',{X_2}'}"] \\
F(X_1 '\ot {X_2}')\ar[r,tick,"\beta_{X_1'\ot {X_2}'}"'] & G(X_1'\ot {X_2}')
\end{tikzcd}
\end{displaymath}
is satisfied, along with the usual associativity and unitality conditions, identifying the two evident $2$-morphisms $(\beta_{X_1} \ot \beta_{X_2}) \ot \beta_{X_3} \rightrightarrows \beta_{X_1 \ot (X_2 \ot X_3)}$, the two $2$-morphisms $\beta_{X_1} \otimes \hid_{I} \rightrightarrows \beta_{X_1 \ot I}$ and the two $2$-morphisms $\hid_{I} \ot \beta_{X_2} \rightrightarrows \beta_{I \ot X_1}$.
We moreover require the equality of the pastings:
 \[
 \cd{
      FX_1 \otimes FX_2
      \ar|@{|}[r]^-{FM \otimes FN}
      \ar[d]_-{ F^2_{X,Z}}
      \dtwocell{dr}{ F^2_{M,N}} &
      FY_1 \otimes FY_2
      \ar|@{|}[r]^-{\beta_Y \otimes \beta_W}
      \ar[d]|-{ F^2_{Y_1,Y_2}}
      \dtwocell{dr}{\beta^2_{Y_1 Y_2}}&
      GY_2 \otimes GY_2 \ar[d]^-{ G^2_{Y_1,Y_2}} \\
      F(X_1 \otimes X_2)
      \ar|@{|}[r]^-{F(M \otimes N)}
      \ar@{=}[d]_-{}
      \dtwocell{drr}{\beta_{M \otimes N}}&
      F(Y_1 \otimes Y_2)
      \ar|@{|}[r]^-{\beta_{Y_1 \otimes Y_2}} &
      G(Y_1 \otimes Y_2)
      \ar@{=}[d]_-{} \\
      F(X_1 \otimes X_2)
      \ar|@{|}[r]_-{\beta_{X_1 \otimes X_2}} &
      G(X_1 \otimes X_2)
      \ar|@{|}[r]_-{G(M \otimes N)} &
      G(Y_1 \otimes Y_2)}\;=\;
      \cd{
      FX_1 \otimes FX_2
      \ar|@{|}[r]^-{FM \otimes FN}
      \ar@{=}[d]_-{}
      \dtwocell{drr}{\beta_M \otimes \beta_N} &
      FY_1 \otimes F Y_2
      \ar|@{|}[r]^-{\beta_{Y_1} \otimes \beta_{Y_2}} &
      GY_1 \otimes GY_2 \ar@{=}[d]_-{} \\
      FX_1 \otimes F X_2
      \ar|@{|}[r]_-{\beta_{X_1} \otimes \beta_{X_2}}
      \ar[d]_-{ F^2_{X_1,X_2}}
      \dtwocell{dr}{\beta^2_{X_1, X_2}} &
      GX_1 \otimes GX_2
      \ar|@{|}[r]_-{GM \otimes GN}
      \ar[d]|-{ G^2_{X_1,X_2}}
      \dtwocell{dr}{ G^2_{M,N}}
      &
      GY_1 \otimes GY_2
      \ar[d]^-{ G^2_{Y_1,Y_2}} \\
      F(X_1 \otimes X_2)
      \ar|@{|}[r]_-{\beta_{X_1 \otimes X_2}} &
      G(X_1 \otimes X_2)
      \ar|@{|}[r]_-{G(M \otimes N)} &
      G(Y_1 \otimes Y_2)}
      \]
  expressing that the natural transformation giving the globular cell
  components of $\beta$ is a monoidal natural transformation. (Note that the `nullary' axiom corresponding to this `binary' axiom holds automatically and need not be stated explicitly.)
\end{defi}

We now consider monoidality of vertical transformations. Given the view of monoidal double categories and lax monoidal double functors as pseudomonoids and lax pseudomonoid maps, the obvious thing to consider would be the corresponding transformations of pseudomonoids, and this would yield the notion of monoidal vertical transformation considered in~\cite[Definition~2.15]{ConstrSymMonBicatsFun}. However, we will need something slightly more general for our applications ({\em cf.} \cref{rmk:mon-vs-psdmon}), which we will term a \emph{pseudomonoidal} vertical transformation. The difference can be appreciated by noting that monoidality of a vertical transformation in the sense of~\emph{loc.~cit.} makes the underlying $2$-natural transformation on the vertical $2$-category into a $\cat{Cat}$-enriched monoidal transformation, while for our transformations, this underlying $2$-natural transformation is  only  a \emph{monoidal pseudonatural transformation} in the sense of \cite[Definition~3]{Monoidalbicatshopfalgebroids}.

\begin{defi}[Pseudomonoidal vertical transformation]
\label{def:montransf}
Let $F \, , F' \colon\dc{C}\to\dc{D}$ be lax monoidal double functors.
A \emph{pseudomonoidal vertical transformation} $\sigma\colon F\Rightarrow F'$
is a vertical transformation equipped with squares
\begin{equation}\label{eq:sigma2}
 \begin{tikzcd}[column sep=.5in,ampersand replacement=\&]
FX_1 \ot FX_2 \ar[r,tick,"\hid"]
	\ar[d," F^2_{X,Y}"']\ar[ddr,phantom,"\Two\sigma^2_{X_1,X_2}"]
	\&
FX_1 \ot FX_2
	\ar[d,"\sigma_{X_1}\ot\sigma_{X_2}"] \\
F(X_1 \ot X_2)
	\ar[d,"\sigma_{X_1 \ot X_2}"']
	\&
F'_{X_1} \ot F' X_2
	\ar[d,"F'^2_{X_1,X_2}"] \\
F'(X_1 \ot X_2)
	\ar[r,tick,"\hid"']
	\&
F'(X_1 \ot X_2)
 \end{tikzcd} \qquad \text{and} \qquad
\begin{tikzcd}[sep=.4in,ampersand replacement=\&]
I
	\ar[d, " F^0"']
	\ar[r, tick, "\hid"]
	\ar[ddr,phantom,"\Two\sigma^0"]
	\&
I
	\ar[dd,"F'^0"]  \\
FI
	\ar[d, "\sigma_I"']
	\&   \\
F'I
	\ar[r, tick, "\hid"']
	\&
F'I \mathrlap{,}
   \end{tikzcd}
\end{equation}
which are invertible in the vertical 2-category $\mathsf{V}(\dc{D})$ and satisfying the following five coherence axioms:
\begin{equation}\label{eq:thisaxiom}
\scalebox{.8}{
 \begin{tikzcd}[sep=.4in,ampersand replacement=\&]
FX_1 \otimes FX_2
	\ar[r,tick,"FM\otimes FN"]
	\ar[d," F^2_{X_1,X_2}"description]
	\ar[dr,phantom,"\Two F^2_{M,N}"]
	\&
FY_1\otimes FY_2
	\ar[ddr,phantom,"\Two\sigma^2"]
	\ar[d," F^2_{Y_1,Y_2}"description]
	\ar[r, tick, "\hid"]
	\&
FY_1 \otimes FY_2
	\ar[d, "\sigma_{Y_1} \otimes \sigma_{Y_2}"description] \\
F(X_1\otimes X_2)
	\ar[r,tick,"F(M\otimes N)"'] \ar[d,"\sigma_{X_1\ot X_2}"description]
	\ar[dr,phantom,"\Two\sigma_{M\otimes N}"]
	\&  F(Y_1 \otimes Y_2)\ar[d,"\sigma_{Y_1 \ot Y_2}"description]
	\&  F'Y_1 \otimes F'Y_2 \ar[d, "F'^2_{Y_1, Y_2}"description]  \\
F'(X_1 \otimes X_2)\ar[r,tick,"F'(M\otimes N)"']
	\&  F'(Y_1 \otimes Y_2) \ar[r, tick, "\hid"']
	\& F'(Y_1 \otimes Y_2)
 \end{tikzcd}} =
 \scalebox{.8}{\begin{tikzcd}[sep=.4in,ampersand replacement=\&]
 FX_1 \otimes FX_2
 	\ar[r, tick, "\hid"]\ar[ddr,phantom,"\Two\sigma^2"]
	\ar[d," F^2_{X_1,X_2}"description]
 	\&
FX_1 \otimes FX_2
	\ar[r,tick,"FM\otimes FN"]\ar[d,"\sigma_{X_1} \otimes\sigma_{X_2}"description]
	\ar[dr,phantom,"\Two\sigma_M\otimes\sigma_N"]
	\&
FY_1\otimes FY_2
	\ar[d,"\sigma_{Y_1} \otimes\sigma_{Y_2} "description] \\
 F(X_1 \otimes X_2)
 	\ar[d, "\sigma_{X_1 \otimes X_2}"description]
	\&
F' X_1 \otimes F'X_2
	\ar[r,tick,"F'M\otimes F'N"']\ar[d,"F'^2_{X_1,X_2}"description]
	\ar[dr,phantom,"\Two F'^2_{M,N}"]
	\&
F'Y_1\otimes F'Y_2
	\ar[d,"F'^2_{Y_1,Y_2}"description] \\
 F'(X_1 \otimes X_2)
 	\ar[r, tick, "\hid"']
	\&
 F'(X_1\otimes X_2) \mathrlap{,}
 	\ar[r,tick,"F'(M\otimes N)"'] \& F'(Y_1 \otimes Y_2) \mathrlap{,}
\end{tikzcd}}
\end{equation}
\begin{displaymath}
 \scalebox{.8}{
\begin{tikzcd}[ampersand replacement=\&]
FX_1 \ot FX_2\ar[d,"Ff\ot Fg"']\ar[r,tick,"\hid"]\ar[dr,phantom,"\Two\hid_{Ff\ot Fg}"] \& FX_1 \ot
FX_2 \ar[d,"Ff\ot Fg"] \\
FX'_1 \ot
FX'_2\ar[r,tick,"\hid"]\ar[d," F^2_{X'_1,X'_2}"']\ar[ddr,phantom,"\Two\sigma^2"] \&
FX'_1 \ot F'_2 \ar[d,"\sigma_{X'_1}\ot\sigma_{X'_2}"] \\
F(X'_1 \ot X'_2)\ar[d,"\sigma_{X'_1 \ot X'_2}"'] \& F'X'_1 \ot F'X'_2\ar[d,"F'^2_{X'_1,X'_2}"] \\
F'(X'_1\ot X'_2)\ar[r,tick,"\hid"'] \& F'(X'_1\ot X'_2)
\end{tikzcd}} =
 \scalebox{.8}{
\begin{tikzcd}[ampersand replacement=\&]
FX_1\ot
FX_2\ar[r,tick,"\hid"]\ar[d," F^2_{X_1,X_2}"']\ar[ddr,phantom,"\Two\sigma^2"]
\& FX_1 \ot FX_2\ar[d,"\sigma_{X_1}\ot\sigma_{X_2}"] \\
F(X_1\ot X_2)\ar[d,"\sigma_{X_1\ot X_2}"'] \& F'X_1\ot F'X_2\ar[d,"F'^2_{X_1,X_2}"] \\
F'(X_1\ot X_2)\ar[dr,phantom,"\Two\hid_{F'(f\ot g)}"]\ar[r,tick,"\hid"]\ar[d,"F'(f\ot g)"'] \&
F'(X_1\ot X_2)\ar[d,"F'(f\ot g)"] \\
F'(X'_1\ot X'_2)\ar[r,tick,"\hid"'] \& F'(X'_1 \ot X'_2) \mathrlap{,}
\end{tikzcd}}
\end{displaymath}
\begin{multline*}
\scalebox{.8}{\begin{tikzcd}[ampersand replacement=\&]
(FX_1\ot FX_2) \ot FX_3
	\ar[d,"\alpha"']
	\ar[r,tick,"\hid"]  \&
(FX_1\ot FX_2) \ot FX_3
	\ar[d,"\alpha"']
	\ar[r,tick,"\hid"] \&
(FX_1\ot FX_2) \ot FX_3
	\ar[d,"\alpha "] \\
FX_1\ot ( FX_2 \ot FX_3)
	\ar[d,"\vid\ot F^2_{X_2,X_3}"']
        \&
FX_1\ot (FX_2 \ot FX_3)
	\ar[d,"\vid\ot F^2_{X_2,X_3}"']
	\ar[r,tick,"\hid"]
	\ar[ddr,phantom,"\Two\hid\ot\sigma^2"] \&
FX_1\ot (FX_2 \ot FX_3)
	\ar[d,"\vid\ot\sigma_{X_2} \ot\sigma_{X_3} "] \\
FX_1\ot F(X_2\ot X_3)
	\ar[dd," F^2_{X_1,X_2 \ot X_3}"']
	\ar[r,tick,"\hid"]\ar[dddr,phantom,"\Two\sigma^2"] \&
FX_1 \ot F(X_2\ot X_3)
	\ar[d,"\vid\ot\sigma_{X_2\ot X_3}"'] \&
FX_1\ot (F'X_2\ot F'X_3)
	\ar[d,"\vid\ot F'^2_{X_2,X_3}"] \\
	\&
FX_1 \ot F'(X_2\ot X_3)
	\ar[d,"\sigma_{X_1} \ot\vid"']
	\ar[r,tick,"\hid"] \&
FX_1\ot F'(X_2\ot X_3)
	\ar[d,"\sigma_X\ot\vid"] \\
F(X_1\ot (X_2 \ot X_3))
	\ar[d,"\sigma_{X_1\ot X_2 \ot X_3}"'] \&
F'X_1\ot F'(X_2\ot X_3)
	\ar[d,"F'^2_{X_1,X_2\ot X_3}"'] \&
F'X_1\ot F'(X_2\ot X_3)
	\ar[d,"F'^2_{X_1,X_2\ot X_3}"]\\
F'(X_1\ot (X_2\ot X_3))
	\ar[r,tick,"\hid"'] \&
F'(X_1\ot (X_2\ot X_3))\ar[r,tick,"\hid"'] \&
F'(X_1\ot (X_2\ot X_3))
\end{tikzcd}}= \\
\scalebox{.8}{\begin{tikzcd}[ampersand replacement=\&]
( FX_1 \ot FX_2) \ot FX_3
 	\ar[d," F^2_{X_1,X_2}\ot\vid"']
	\ar[r,tick,"\hid"] \&
(FX_1\ot FX_2) \ot FX_3
	\ar[d," F^2_{X_1,X_2}\ot\vid"]\ar[r,tick,"\hid"] \&
(FX_1\ot FX_2) \ot FX_3
	\ar[d,"\vid\ot\sigma_{X_3}"']\ar[r,tick,"\hid"] \&
(FX_1\ot FX_2) \ot FX_3
	\ar[d,"\vid\ot\sigma_{X_3}"] \\
F(X_1\ot X_2)\ot FX_3
	\ar[dd," F^2_{X_1\ot Y_2,X_3}"']
	\ar[dddr,phantom,"\Two\sigma^2"]
	\ar[r,tick,"\hid"] \&
F(X_1\ot X_2)\ot FX_3
	\ar[d,"\vid\ot\sigma_{X_3}"] \&
(FX_1\ot FX_2) \ot F'X_3
	\ar[d," F^2_{X_1,X_2}\ot\vid"']
	\ar[r,tick,"\hid"]
	\ar[ddr,phantom,"\Two\sigma^2\ot\hid"] \&
(FX_1\ot FX_2) \ot F'X_3
\ar[d,"\sigma_X\ot\sigma_Y\ot\vid"] \\
	\&
F(X_1\ot X_2)\ot F'X_3
	\ar[d,"\sigma_{X_1\ot X_2}\ot\vid"]
	\ar[r,tick,"\hid"]  \&
F(X_1\ot X_2)\ot F'X_3
	\ar[d,"\sigma_{X_1\ot X_2}\ot\vid"'] \&
(F'X_1\ot F'X_2) \ot F'X_3
	\ar[d,"F'^2_{X_1,X_2}\ot \vid"] \\
F((X_1\ot X_2)\ot X_3)
	\ar[d,"\sigma_{(X_1\ot X_2) \ot X_3}"'] \&
F'(X_1\ot X_2)\ot
F'X_3\ar[d,"F'^2_{X_1\ot X_2,X_3}"]
	\ar[r,tick,"\hid"] \&
F'(X_1\ot X_2)\ot F'X_3
	\ar[r,tick,"\hid"] \&
F'(X_1\ot X_2)\ot F'X_3
	\ar[d,"F'^2_{X_1\ot X_2,X_3}"] \\
F'((X_1 \ot X_2) \ot X_3)
	\ar[r,tick,"\hid"']
	\ar[d, "F(\alpha)"'] \&
F'((X_1\ot X_2) \ot X_3)
	\ar[d, "F(\alpha)"] \&
	\&
F'((X_1\ot X_2)\ot X_3)
	 \ar[d, "F(\alpha)"] \\
F'( X_1 \ot (X_2 \ot X_3))
	\ar[r,tick,"\hid"']  \&
F'( X_1 \ot (X_2 \ot X_3))
	\ar[rr,tick,"\hid"']  \&
	\&
F'( X_1 \ot (X_2 \ot X_3)) \mathrlap{,}
\end{tikzcd}}
\end{multline*}

\begin{displaymath}
\scalebox{.8}{\begin{tikzcd}[ampersand replacement=\&]
FX \ot I \ar[d,"\vid\ot F^0"']\ar[r,tick,"\hid"] \&
FX \ot I \ar[d,"\vid\ot F^0"] \\
FX\ot FI\ar[r,tick,"\hid"]\ar[d," F^2_{X,I}"']\ar[ddr,phantom,"\Two\sigma^2"]
\&
FX\ot FI\ar[d,"\sigma_X\ot\sigma_I"] \\
F(X \ot I) \ar[d,"\sigma_{X \ot I}"'] \& F'X\ot F'I\ar[d,"F'^2_{X,I}"] \\
F'(X \ot I) \ar[r,tick,"\hid"'] \& F'(X \ot I)
\end{tikzcd}}=
\scalebox{.8}{\begin{tikzcd}[ampersand replacement=\&]
FX\ot I \ar[d,"\vid\ot F^0"']\ar[r,tick,"\hid"]\ar[ddr,phantom,"\Two\hid\ot\sigma^0"] \& FX \ot I \ar[dd,"\vid\ot F'^0"] \\
FX\ot FI\ar[d,"\vid\ot\sigma_I"']\&  \\
FX\ot F'I\ar[r,tick,"\hid"]\ar[d,"\sigma_X\ot\vid"'] \& FX\ot
F'I\ar[d,"\sigma_X\ot\vid"] \\
F'X\ot F'I\ar[d,"F'^2_{X,I}"'] \& F'X\ot F'I\ar[d,"F'^2_{X,I}"] \\
F'(X \ot I) \ar[r,tick,"\hid"'] \& F'(X\ot I) \mathrlap{,}
\end{tikzcd}}\;\;
\scalebox{.8}{\begin{tikzcd}[ampersand replacement=\&]
I \ot FX\ar[d," F^0\ot\vid"']\ar[r,tick,"\hid"] \&
I \ot FX\ar[d," F^0\ot\vid"] \\
FI\ot FX\ar[r,tick,"\hid"]\ar[d,"F^2_{I,X}"']\ar[ddr,phantom,"\Two\sigma^2"]
\&
FI\ot FX\ar[d,"\sigma_I\ot\sigma_X"] \\
F(I \ot X) \ar[d,"\sigma_{I\ot X}"'] \& F'I\ot F'X\ar[d,"F'^2_{I,X}"] \\
F'(I \ot X) \ar[r,tick,"\hid"'] \& F'X
 \end{tikzcd}}=
\scalebox{.8}{\begin{tikzcd}[ampersand replacement=\&]
I \ot FX\ar[d," F^0\ot\vid"']\ar[r,tick,"\hid"]\ar[ddr,phantom,"\Two\sigma^0\ot\hid"] \& I \ot FX\ar[dd,"F'^0\ot\vid"] \\
FI\ot FX\ar[d,"\sigma_I\ot\vid"'] \&  \\
F'I\ot FX\ar[r,tick,"\hid"]\ar[d,"\vid\ot\sigma_X"'] \& F'I\ot
FX\ar[d,"\vid\ot\sigma_X"] \\
F'I\ot F'X\ar[d,"F'^2_{I,X}"'] \& F'I\ot F'X\ar[d,"F'^2_{I,X}"] \\
F'(I \ot X) \ar[r,tick,"\hid"'] \& F'(I \ot X) \mathrlap{.}
\end{tikzcd}}
\end{displaymath}
Note that the final four of these axioms only involve structure in the vertical $2$-category $\mathsf{V}(\dc{D})$; and in fact, they correspond exactly to the axioms for a monoidal pseudonatural transformation from~\cite[Definition~3]{Monoidalbicatshopfalgebroids}. More explicitly, the second axiom expresses that the $2$-cells $\sigma^2_{X_1, X_2}$ are components of a modification, while the third through fifth axioms are precisely the three coherence axioms of \emph{loc.~cit.}

If $\sigma^0$ and the components of $\sigma^2$ are identity 2-cells, then $\sigma$ becomes a
\emph{monoidal} vertical transformation in the sense of
\cite[Definition~2.15]{ConstrSymMonBicatsFun}.
In that case, $\sigma_0\colon F_0\Rightarrow F'_0$ and $\sigma_1\colon
F_1\Rightarrow F'_1$ are monoidal transformations in the usual sense between lax
monoidal functors.
\end{defi}

\begin{rmk} \label{rmk:mon-vs-psdmon}
The notion of a monoidal (rather than pseudomonoidal) vertical transformation
is insufficiently general for the situation we are interested in:
the monoidality of the free symmetric monoidal category double monad on the double category of
small categories, functors and profunctors, as considered in \cref{sec:application}. The underlying double
functor of this double monad can be equipped with lax monoidal structure, with respect to which the monad unit is a
monoidal
vertical transformation; however, the monad multiplication is not a monoidal as a vertical transformation, but
only \emph{pseudomonoidal}. This can be seen as a consequence of the fact that the free symmetric monoidal category monad is not
commutative, but only \emph{pseudocommutative} in the sense of~\cite{HylandPower}.
\end{rmk}

We now describe the final piece of structure needed for a double category of monoidal double functors.

\begin{defi}[Monoidal modification] \label{def:monmodif}
Let $\beta,\beta'$ be monoidal horizontal transformations and let
$\sigma,\tau$ be pseudomonoidal vertical transformations, as displayed on the boundary of:
\begin{displaymath}
  \begin{tikzcd}
F\ar[r,Rightarrow,bigtick,"\beta"]\ar[d,Rightarrow,"{\sigma}"']\ar[dr,phantom,"\scriptstyle\Ddownarrow{\gamma}"] & G \ar[d,Rightarrow,"\tau"] \\
F'\ar[r,Rightarrow,bigtick,"{\beta'}"'] & G'\mathrlap{ .}
\end{tikzcd}
\end{displaymath}
A \emph{monoidal modification} $\gamma$ filling this boundary is a modification of the displayed
shape satisfying the axioms:
\begin{equation*}
\begin{tikzcd}
FX_1\ot
FX_2\ar[d," F^2_{X_1,X_2}"']\ar[r,tick,"\beta_{X_1} \ot\beta_{X_2}"]\ar[dr,phantom,
"\Two\beta^2"] & GX_1\ot
GX_2\ar[d," G^2_{X_1,X_2}"]\ar[r,tick,"\hid"]\ar[ddr,phantom,"\Two\tau^2"] & GX_1\ot
GX_2\ar[d,"\tau_{X_1} \ot\tau_{X_2}"] \\
F(X_1\ot X_2)\ar[d,"\sigma_{X_1\ot X_2}"']\ar[r,tick,"\beta_{X_1 \ot
X_2}"']\ar[dr,phantom,"\Two\gamma_{X_1 \ot X_2}"] & G(X_1\ot X_2)\ar[d,"\tau_{X_1\ot X_2}"] &
G'X_1\ot G'X_2\ar[d," G'^2_{X_1,X_2}"] \\
F'(X_1\ot X_2)\ar[r,tick,"\beta'_{X_1\ot X_2}"'] & G'(X_1\ot X_2)\ar[r,tick,"\hid"'] &
G'(X_1\ot X_2)
\end{tikzcd} = \\
\begin{tikzcd}
FX_1\ot FX_2\ar[r,tick,"\hid"]\ar[d," F^2_{X_1,X_2}"']\ar[ddr,phantom,"\Two\sigma^2"] &
FX_1\ot FX_2\ar[d,"\sigma_{X_1}\ot
\sigma_{X_2}"']\ar[r,tick,"\beta_{X_1} \ot\beta_{X_2}"]\ar[dr,phantom,
"\Two\gamma_{X_1} \ot\gamma_{X_2}"] & GX_1\ot GX_2\ar[d,"\tau_{X_1} \ot\tau_{X_2}"] \\
F(X_1\ot X_2)\ar[d,"\sigma_{X_1\ot X_2}"'] & F'X_1\ot
F'X_2\ar[d," F'^2_{X_1,X_2}"']\ar[r,tick,"\beta'_{X_1} \ot\beta'_{X_2}"]\ar[dr,phantom,
"\Two\beta'^2"] & G'X_1\ot G'X_2\ar[d," G'^2_{X_1,X_2}"] \\
F'(X_1 \ot X_2)\ar[r,tick,"\hid"'] & F'(X_1\ot X_2)\ar[r,"\beta'_{X_1\ot X_2}"'] & G'(X_1\ot X_2)
\end{tikzcd}
\end{equation*}
\begin{equation}\label{eq:monmodif}
\begin{tikzcd}
I\ar[d," F^0"']\ar[dr,phantom,"\Two\beta^0"]\ar[r,tick,"\hid"] &
I\ar[d," G^0"]\ar[r,tick,"\hid"]\ar[ddr,phantom,"\Two\tau^0"] & I\ar[dd," G'^0"]\\
FI\ar[r,tick,"\beta_I"']\ar[dr,phantom,"\Two\gamma_I"]\ar[d,"\sigma_I"'] & GI\ar[d,"\tau_I"] & \\
F'I\ar[r,tick,"\beta'_I"'] & G'I\ar[r,tick,"\hid"'] & G'I
\end{tikzcd}\;=\;
\begin{tikzcd}
I\ar[d, " F^0"'] \ar[r, tick, "\hid"]\ar[ddr,phantom,"\Two\sigma^0"] & I\ar[ddr,phantom,"\Two\beta'^0"]\ar[dd,"F'^0"]\ar[r,tick,"\hid"] &
I\ar[dd,"G'^0"]  \\
FI \ar[d, "\sigma_I"'] & &  \\
F'I \ar[r, tick, "\hid"'] & F'I\ar[r,tick,"\beta'_I"'] & G'I\mathrlap{ .}
\end{tikzcd}
\end{equation}
\end{defi}

We now provide an analogue of \cref{thm:functor-dblcat} in the monoidal setting, by constructing a double category of monoidal double functors between two monoidal
double categories $\dc{C}$ and $\dc{D}$. It would be routine to construct
a double category of lax monoidal double functors, monoidal vertical transformations, monoidal horizontal
transformations, and monoidal modifications; however, because we wish to involve \emph{pseudomonoidal} vertical transformations, a little more care is needed in checking the details.

\begin{prop}[Functor double categories, monoidal case] \label{thm:functor-dblcat-monoidal} Let $\dc{C}$,
$\dc{D}$  be monoidal double categories. There is a double category
$\cat{MonDblCat}[\dc{C}, \dc{D}]$
of lax monoidal (pseudo) double functors, pseudomonoidal vertical transformations,
monoidal  horizontal transformations, and monoidal modifications.
\end{prop}

Note that in \cref{thm:functor-dblcat}, we considered \emph{oplax} double functors; here we consider only (pseudo) double functors, but endowed with \emph{lax} monoidal structure. While it certainly would be possible to consider ``lax monoidal oplax double functors'', this is not needed for our applications.

\begin{proof}
We first show that lax monoidal double functors and pseudomonoidal
vertical transformations form a category. Given pseudomonoidal vertical transformations $\sigma \colon F \Rightarrow F'$ and $\tau \colon F' \Rightarrow F''$, we endow the composite vertical
transformation $\tau\cdot\sigma\colon F\Rightarrow F''$
with pseudomonoidal structure via the pasting~composites:
\begin{equation}\label{eq:cdot}
\begin{tikzcd}
F{X_1}\ot F{X_2}\ar[d," F^2_{{X_1},{X_2}}"']\ar[r,tick,"\hid"]\ar[ddr,phantom,"\Two\sigma^2"] &
F{X_1}\ot F{X_2}\ar[d,"\sigma_{X_1}\ot\sigma_{X_2}"]\ar[r,tick,"\hid"] & F{X_1}\ot
F{X_2}\ar[d,"\sigma_{X_1}\ot\sigma_{X_2}"] \\
F({X_1}\ot {X_2})\ar[d,"\sigma_{{X_1}\ot {X_2}}"'] & F'{X_1}\ot
F'{X_2}\ar[d," F'^2_{{X_1},{X_2}}"]\ar[r,tick,"\hid"]\ar[ddr,phantom,"\Two\tau^2"] & F'{X_1}\ot
F'{X_2}\ar[d,"\tau_{X_1}\ot\tau_{X_2}"] \\
F'({X_1}\ot {X_2})\ar[d,"\tau_{{X_1}\ot {X_2}}"']\ar[r,tick,"\hid"'] & F'({X_1}\ot
{X_2})\ar[d,"\tau_{{X_1}\ot
{X_2}}"] & F''{X_1}\ot F''{X_2}\ar[d," F''^2_{{X_1},{X_2}}"] \\
F''({X_1}\ot {X_2})\ar[r,tick,"\hid"'] & F''({X_1}\ot {X_2})\ar[r,tick,"\hid"'] & F''({X_1}\ot {X_2})
\end{tikzcd}\qquad
\begin{tikzcd}
I\ar[d," F^0"']\ar[r,tick,"\hid"]\ar[ddr,phantom,"\Two\sigma^0"] & I\ar[r,tick,"\hid"]\ar[dd," F'^0"]\ar[dddr,phantom,"\Two\tau^0"] &
I\ar[ddd," F''^0"] \\
FI\ar[d,"\sigma_I"'] &&  \\
F'I\ar[d,"\tau_I"']\ar[r,tick,"\hid"] & F'I\ar[d,"\tau_I"] &  \\
F''I\ar[r,tick,"\hid"'] & F''I\ar[r,tick,"\hid"'] & F''I\mathrlap{ .}
\end{tikzcd}
\end{equation}
It is now routine to verify the pseudomonoidal vertical transformation axioms for $\tau \cdot \sigma$, and to check that this composition law is associative and unital, so yielding the desired category.

We next show that monoidal horizontal transformations and monoidal modifications
form a category; for which it suffices to verify that, given a pair of composable monoidal modifications, their composite \emph{qua} modification, as in~\cref{def:monmodif}, is again monoidal. This is straightforward.

We now provide the horizontal composition law for $\cat{MonDblCat}[\dc{C}, \dc{D}]$. Given monoidal horizontal transformations $\beta \colon F \ticktwoar G$ and $\gamma \colon G \ticktwoar H$, we endow the composite horizontal transformation $\gamma \circ \beta$ with monoidal structure via the pastings:
\begin{equation*}
  \cd[@C+1.5em]{
    F{X_1} \otimes F{X_2}
    \ar@{=}[d]_-{}
    \ar|@{|}[rr]^-{(\gamma_{X_1} \circ \beta_{X_1}) \ot (\gamma_{X_2} \circ \beta_{X_2})}
    \dtwocell{drr}{\tau} & &
    G{X_1} \otimes G{X_2}
    \ar@{=}[d]_-{}\\
    F{X_1} \otimes F{X_2}
    \ar|@{|}[r]^-{\beta_{X_1} \otimes \beta_{X_2}}
    \ar[d]_-{ F^2_{X_1,{X_2}}}
    \dtwocell{dr}{\beta^2_{X_1,{X_2}}} &
    G{X_1} \otimes G{X_2}
    \ar|@{|}[r]^-{\gamma_{X_1} \otimes \gamma_{X_2}}
    \ar[d]^-{ G^2_{X,{X_2}}}
    \dtwocell{dr}{\gamma^2_{X_1,{X_2}}} &
    H{X_1} \otimes H{X_2}
    \ar[d]^-{ G^2_{X,{X_2}}} \\
    F({X_1} \otimes {X_2})
    \ar|@{|}[r]_-{\beta_{{X_1} \otimes {X_2}}} &
    G({X_1} \otimes {X_2})
    \ar|@{|}[r]_-{\gamma_{{X_1} \otimes {X_2}}} &
    H({X_1} \otimes {X_2})
  } \quad\text{and} \quad
  \cd[@C+2em]{
    I
    \ar@{=}[d]_-{}
    \ar|@{|}[rr]^-{\hid_{I}}
    \twocong{drr} & & I
    \ar@{=}[d]_-{}\\
      I
      \ar[d]_-{ F^0}
      \ar|@{|}[r]^-{\hid_{I}}
      \dtwocell{dr}{\beta^0} &
      I
      \ar[d]^-{ G^0}
      \ar|@{|}[r]^-{\hid_{I}}
      \dtwocell{dr}{\gamma^0} &
      I
      \ar[d]^-{ H^0}  \\
      FI
      \ar|@{|}[r]_-{\beta_{I}} &
      GI
      \ar|@{|}[r]_-{\gamma_{I}} &
      HI\mathrlap{ .}
    }
  \end{equation*}
Direct verification yields the horizontal transformation axioms.
To make the assignment $\beta, \gamma \mapsto \gamma \circ \beta$ into a functor, it now suffices to observe that that the horizontal composition of two monoidal
modifications \emph{qua} modification is again monoidal; this is again a matter of direct verification. Finally, the globular constraints $a,\ell,r$ of $\cat{MonDblCat}[\dc{C}, \dc{D}]$ are inherited from $\cat{DblCat}[\dc{C}, \dc{D}]$, and it is simply a matter of checking that these are indeed monoidal modifications.
\end{proof}

The next result builds on \cref{lem:2}.

\begin{prop}
  \label{lem:1}
  A pseudomonoidal vertical transformation $\sigma \colon F
  \Rightarrow F'$  has a companion as a vertical 1-cell of $\cat{MonDblCat}[\dc{C}, \dc{D}]$ if and only if the underlying vertical
  transformation of $\sigma$ has a companion as a vertical 1-cell of
  $\cat{DblCat}[\dc{C}, \dc{D}]$, {i.e.}~it is special.
\end{prop}

\begin{proof}
  The `only if' direction is trivial: if $\sigma$ has a companion in $\cat{MonDblCat}[\dc{C}, \dc{D}]$, then applying the forgetful double functor $\cat{MonDblCat}[\dc{C}, \dc{D}] \rightarrow \cat{DblCat}[\dc{C}, \dc{D}]$ shows it has a companion in $\cat{DblCat}[\dc{C}, \dc{D}]$. For the `if' direction, given a pseudomonoidal transformation $\sigma$ as in
\cref{def:montransf}, the additional necessary data for the induced horizontal transformation $\comp{\sigma}\colon F\ticktwoar F'$ as described in the
proof of \cref{lem:2} to be monoidal are
the cells
of~\eqref{eq:1}. We obtain these as companion transposes of the
structure data \cref{eq:sigma2} of the pseudomonoidal vertical transformation
$\sigma$, as in:
\begin{displaymath}
(\comp{\sigma})^2_{X_1,X_2}\defeq \begin{tikzcd}
FX_1\ot
FX_2\ar[d,equal]\ar[r,tick,"\comp{\sigma}_{X_1}\ot\comp{\sigma}_{X_2}"]\ar[dr,
phantom,"\cong"] & F'X_1\ot F'X_2\ar[d,equal] \\
FX_1\ot
FX_2\ar[d," F^2_{X_1,X_2}"']\ar[r,tick,"\comp{\sigma_{X_1} \ot\sigma_{X_2}}"]\ar[
dr , phantom,"\Two\comp{\sigma^2}"] & F'X_1\ot F'X_2\ar[d,"F'^2_{X_1,X_2}"] \\
F(X_1\ot X_2)\ar[r,tick,"\comp{\sigma_{X_1\ot X_2}}"'] & F'(X_1\ot X_2)
 \end{tikzcd}\qquad
(\comp{\sigma})^0\defeq\begin{tikzcd}
I\ar[r,tick,"\hid_I"]\ar[d," F^0"']
\ar[dr,phantom,"\Two\comp{\sigma^0}"] & I\ar[d,"F'^0"] \\
FI\ar[r,tick,"\comp{\sigma}_I"'] & F'I
\end{tikzcd}
\end{displaymath}
where the top-left isomorphism arises due to the double functor $\ot$
preserving companions. That this makes $\comp \sigma$ into a monoidal horizontal transformation can now be checked by lengthy, but straightforward, calculations. Similarly, it is straightforward to verify that with respect to this structure, the companion $2$-morphisms $p_1$ and $p_2$ in $\cat{DblCat}[\dc{C}, \dc{D}]$ are monoidal, and so lift to $\cat{MonDblCat}[\dc{C}, \dc{D}]$ as required.
\end{proof}

\section{Monoids in monoidal double categories}
\label{sec:monoids}

In this section, we consider horizontal and vertical monoids in a monoidal double category.
When instantiated in the monoidal double categories $\cat{DblCat}[\dc{C},\dc{C}]$ and $\cat{MonDblCat}[\dc{C},\dc{C}]$ of
\cref{thm:functor-dblcat,thm:functor-dblcat-monoidal}, these will give us the notions of horizontal and vertical double monad, and of monoidal
horizontal and vertical double monad respectively, to be considered in \cref{sec:dbl-monad,sec:mon-dbl-monad}.

We begin with the notion of a horizontal monoid in a monoidal double category. This
is analogous to a pseudomonoid in a monoidal bicategory, in that the
associativity and unit axioms do not hold on the nose, but rather up to invertible squares.

\begin{defi} \label{def:hordoublemonoid}
  Let $\dc{C}$ be a monoidal double category. A \emph{horizontal monoid} in $\dc{C}$
  consists of:
    \begin{itemize}
    \item an object $A$;
    \item horizontal 1-cells $m \colon A \otimes A \tor A$ and $e \colon
      I \tor A$;
    \item invertible cells
      \begin{equation}\label{eq:structurehorizontalpseudo}
      \begin{tikzcd}
   (A\ot A)\ot A\ar[d,"\alpha"']\ar[r,tick,"m\ot\hid"]\ar[drr,phantom,"\Two\mathfrak{a}"] & A\ot A\ar[r,tick,"m"] & A\ar[d,equal] \\
   A\ot(A\ot A)\ar[r,tick,"\hid\ot m"'] & A\ot A\ar[r,tick,"m"'] & A
      \end{tikzcd}\ \
      \begin{tikzcd}
       A\ot I\ar[d,"\rho"']\ar[drr,phantom,"\Two\mathfrak{r}"]\ar[r,tick,"\hid\ot e"] & A\ot A\ar[r,tick,"m"] & A\ar[d,equal] \\
       A\ar[rr,tick,"\hid"'] && A
      \end{tikzcd}\ \ \text{and} \ \
      \begin{tikzcd}
       I\ot A\ar[d,"\lambda"']\ar[r,tick,"e\ot\hid"]\ar[drr,phantom,"\Two\mathfrak{l}"] & A\ot A\ar[r,tick,"m"] & A\ar[d,equal] \\
       A\ar[rr,tick,"\hid"'] && A \mathrlap{.}
      \end{tikzcd}
      \end{equation}
    \end{itemize}
    These data are required to satisfy the coherence axioms that:
\begin{equation*}
\scalebox{.85}{
\begin{tikzcd}[ampersand replacement=\&]
((A \otimes A) \otimes A) \otimes A \ar[d,"\alpha \otimes \vid"'] \ar[r,tick,"(m \otimes \hid) \otimes \hid"] \ar[drr,phantom,"\Two{\mathfrak a}\otimes 1"] \& (A \otimes A) \otimes A
\ar[r,tick,"m \otimes \hid"] \& A \otimes A \ar[d,equal] \ar[r,tick,"m"] \ar[dr,phantom,"\Two\vid_m"] \& A \ar[d,equal] \\
(A \otimes (A \otimes A)) \otimes A \ar[r,tick,"(\hid \otimes m) \otimes \hid"'] \ar[d,"\alpha"'] \ar[dr,phantom,"\Two\alpha"] \& (A \otimes A) \otimes A \ar[d,"\alpha"'] \ar[r,tick,"m \otimes \hid"']\ar[drr,phantom,"\Two {\mathfrak a}"] \& A \otimes A \ar[r,tick,"m"'] \& A \ar[d,equal] \\
A \otimes ((A \otimes A) \otimes A) \ar[r,tick,"\hid \otimes (m \otimes \hid)"'] \ar[d,"\vid\otimes \alpha"'] \ar[drr,phantom,"\Two 1 \otimes {\mathfrak a}"] \& A \otimes (A \otimes A) \ar[r,tick,"\hid \otimes m"'] \& A \otimes A \ar[r,tick,"m"']\ar[d,equal] \ar[dr,phantom,"\Two\vid_m"] \& A \ar[d,equal] \\
A \otimes (A \otimes (A \otimes A)) \ar[r,tick,"\hid \otimes (\hid \otimes m)"'] \& A \otimes (A \otimes A)  \ar[r,tick,"\hid \otimes m"'] \& A \otimes A \ar[r,tick,"m"'] \& A
\end{tikzcd}=
\begin{tikzcd}[ampersand replacement=\&]
((A \otimes A) \otimes A) \otimes A \ar[d,"\alpha"']\ar[r,tick,"(m \otimes \hid) \otimes \hid"']\ar[dr,phantom,"\Two\alpha"] \& (A \otimes A) \otimes A \ar[d,"\alpha"'] \ar[r,tick,"m \otimes \hid"] \ar[drr,phantom,"\Two{\mathfrak a}"] \& A \otimes A \ar[r,tick,"m"] \& A\ar[d,equal] \\
(A \otimes A) \otimes (A \otimes A) \ar[r,tick,"m \otimes \hid"] \ar[d,equal] \ar[drr,phantom,"\Two\cong"] \& A \otimes (A \otimes A) \ar[r,tick,"\hid \otimes m"'] \& A \otimes A \ar[r,tick,"m"'] \ar[d,equal] \ar[dr,phantom,"\Two\vid_m"] \& A \ar[d,equal] \\
(A \otimes A) \otimes (A \otimes A) \ar[r,tick,"\hid \otimes m"'] \ar[d,"\alpha"'] \ar[dr,phantom,"\Two\alpha"] \& (A \otimes A) \otimes A \ar[r,tick,"m \otimes \hid"'] \ar[d,"\alpha"'] \ar[drr,phantom,"\Two{\mathfrak a}"] \& A \otimes A \ar[r,tick,"m"'] \& A\ar[d,equal] \\
A \otimes (A \otimes (A \otimes A)) \ar[r,tick,"\hid \otimes (\hid \otimes m)"'] \& A \otimes (A \otimes A) \ar[r,tick,"\hid \otimes m"'] \& A \otimes A \ar[r,tick,"m"'] \& A
\end{tikzcd}}
\end{equation*}
\begin{displaymath}
\scalebox{.85}{\begin{tikzcd}[ampersand replacement=\&]
(A \otimes I) \otimes A \ar[r,tick,"(\hid \otimes e) \otimes \hid"] \ar[d,"\alpha"'] \ar[dr, phantom,"\Two\alpha"] \& (A \otimes A) \otimes A \ar[r,tick,"m \otimes \hid"] \ar[d,"\alpha"']\ar[drr,phantom,"\Two{\mathfrak a}"] \& A \otimes A \ar[r,tick,"m"] \& A \ar[d,equal] \\
A \otimes (I \otimes A)  \ar[d,"\vid \otimes \lambda"'] \ar[drr,phantom,"\Two 1\otimes{\mathfrak l}"] \ar[r,tick,"\hid \otimes (e \otimes \hid)"'] \& A \otimes (A \otimes A)
\ar[r,tick,"\hid \otimes m"'] \& A \otimes A \ar[dr,phantom,"\Two\vid_m"] \ar[d,equal] \ar[r,tick,"m"] \& A \ar[d,equal] \\
A \otimes A \ar[rr,tick,"\hid"'] \& \& A \otimes A \ar[r,tick,"m"'] \& A
\end{tikzcd}=
\begin{tikzcd}[ampersand replacement=\&]
(A \otimes I) \otimes A \ar[d,"\rho \otimes \vid"']\ar[r,tick,"(\hid \otimes e) \otimes \hid"]
\ar[drr,phantom,"\Two{\mathfrak r} \otimes 1"] \& (A \otimes A) \otimes A \ar[r,tick,"m \otimes \hid"] \& A \otimes
A\ar[d,equal]\ar[r,tick,"m"]\ar[dr,phantom,"\Two\vid_m"] \& A\ar[d,equal] \\
A \otimes A \ar[rr,tick,"\hid"'] \&\& A \otimes A \ar[r,tick,"m"'] \& A
\end{tikzcd}}
\end{displaymath}
\end{defi}

\begin{rmk}\label{rem:pseudomonhorizontalbicat}
  As discussed in \cref{sec:mon-dbl-cat}, under fairly mild conditions the horizontal bicategory $\ca{H}(\dc{C})$ of a monoidal double category $\dc{C}$ will have the structure of a monoidal double category, whose monoidal associativity and unit constraint $1$-cells are the companions of the corresponding constraints for $\dc{C}$. In this situation, horizontal monoids in $\dc{C}$ correspond to pseudomonoids in $\ca{H}(\dc{C})$ by taking the companion transposes of the coherence data~\eqref{eq:structurehorizontalpseudo}.
\end{rmk}

  \begin{defi}[Vertical monoid]  \label{def:verdoublemonoid}
    Let $\dc{C}$ be a monoidal double category. A \emph{vertical monoid} in $\dc{C}$ is a monoid in the monoidal
    category $\dc{C}_0$. Explicitly, it is an object $A$
    endowed with vertical 1-cells $m \colon A \otimes A \rightarrow A$
    and $e \colon I \rightarrow A$ satisfying the usual associativity  and unit
     laws.
\end{defi}

The next result shows how we may induce horizontal monoids from vertical ones, and will be applied in \cref{cor:horizontalpseudomonad}, relating
horizontal and vertical double monads, and \cref{cor:2}, relating monoidal horizontal and vertical double monads.

\begin{thm}\label{prop:verticalgiveshorizontal}
 Let $\dc{C}$ be a monoidal double category and $(A, m, e)$ be a vertical
 monoid in $\dc{C}$, such that $m$ and $e$ have companions. The companion transposes
\begin{displaymath}
\begin{tikzcd}
(A \otimes A) \otimes A
\ar[drr,phantom,"\Two \mathfrak a"]\ar[r,tick,"
\comp{m}\otimes\hid"]\ar[d,"\alpha"'] & A \otimes A \ar[r,tick,"\comp{m}"] &
A \ar[d,equal] \\
A \otimes (A\otimes A)  \ar[r,tick,"\hid\otimes\comp{m}"'] & A
\otimes A \ar[r,tick,"\comp{m}"'] & A
\end{tikzcd}\;\;
\begin{tikzcd}
A\otimes I \ar[drr,phantom,"\Two {\mathfrak r}"]\ar[d,"\rho"']\ar[r,tick,"\hid \otimes\comp{e}"]& A\otimes A\ar[r,tick,"\comp{m}"] &
A \ar[d,equal] \\
A \ar[rr,tick,"\hid"'] && A
\end{tikzcd}\;\;
\begin{tikzcd}
I\otimes
A\ar[r,tick,"\comp{e}\otimes\hid"]\ar[drr,phantom,"\Two{\mathfrak l}"]
\ar [ d , "\lambda"' ]
& A \otimes A \ar[r,tick,"\comp{m}"] & A\ar[d,equal] \\
A\ar[rr,tick,"\hid"'] && A
\end{tikzcd}
\end{displaymath}
 of the monoid identities endow $(A, \comp m, \comp e)$ with the
structure of a horizontal pseudomonoid.
\end{thm}
\begin{proof}
The displayed $2$-morphisms are
constructed using transpose operations like \cref{eq:transpose} from the vertical associativity and unitality monoid
axioms for $A$. They are vertically invertible since $\alpha$, $\rho$ and $\lambda$ are, so that $\mathfrak{a}^{\mi1}$ may be constructed as
companion transposes of the identities $m \circ (1 \otimes m) \circ \alpha^{\mi1} = m \circ (m \otimes 1)$, and similarly for $\mathfrak{l}^{\mi1}$
and $\mathfrak{r}^{\mi1}$.

The coherence axioms of \cref{def:hordoublemonoid} for a horizontal pseudomonoid
can now be checked by computing appropriate transposes of the
required diagrams and making use of \cref{lem:companioncomponents}.
\end{proof}

\section{Double monads}
\label{sec:dbl-monad}

For an ordinary category $\nc{C}$, the category of endofunctors of $\nc{C}$ has a monoidal structure given by composition, and a monoid therein is precisely a monad on $\nc{C}$. 
In the case of double categories, we can do something similar by exploiting our work in \cref{sec:dbl-fun,sec:monoids}, so leading to a notion of double monad: or rather, \emph{two} notions of double monad, horizontal and vertical.

To begin with, observe that \cref{thm:functor-dblcat}  states in particular that for any double category $\dc{C}$, there is a double category
$\cat{DblCat}[\dc{C},\dc{C}]$ of double endofunctors, vertical transformations, horizontal transformations and modifications
(\cref{def:doublefunctor,def:vert-transf,def:hor-transf,def:modification}). In fact, as is well-known, this double category is monoidal:

\begin{prop}[Composition monoidal structure]
 \label{ex:endofun-monoidal-comp}
Let $\dc{C}$ be a double category. The double category $\cat{DblCat}[\dc{C},\dc{C}]$ admits a monoidal structure given by composition.
\end{prop}

\begin{proof}
  We only sketch the proof; for a full treatment see, for example, \cite[Proposition~39]{Garner2006Double}.

  Given double endofunctors $F_1,F_2 \colon \dc{C} \rightarrow \dc{C}$, we define $F_1 \ot F_2$ to be the double endofunctor $F_2F_1$ with underlying ordinary functors $(F_2)_0 \circ (F_1)_0 \colon \dc{C}_0 \rightarrow \dc{C}_0$ and $(F_2)_1 \circ (F_1)_1 \colon \dc{C}_1 \rightarrow \dc{C}_1$, and with coherence data obtained by vertically pasting those for $F$ and $G$. Given vertical transformations $\sigma_1 \colon F_1 \Rightarrow F'_1$ and
$\sigma_2 \colon F_2 \Rightarrow F'_2$, we define $\sigma_1 \ot \sigma_2$ to be the vertical transformation $\sigma_2 \sigma_1 \colon F_2F_1 \Rightarrow F'_2F'_1$ with underlying ordinary natural transformations given by the horizontal composites $(\sigma_2)_0 \ast (\sigma_1)_0$ and $(\sigma_2)_0 \ast (\sigma_1)_0$. With the identity double functor as unit, this yields a strict monoidal structure on the category of double endofunctors and vertical transformations.

Next, given horizontal transformations $\beta_1 \colon F_1 \ticktwoar G_1$ and
$\beta_2 \colon  F_2 \ticktwoar G_2$ we define $\beta_1 \ot \beta_2$ to be the
horizontal transformation $\beta_2 \beta_1 \colon F_2 F_1 \ticktwoar G_2 G_1$ with
horizontal $1$-cell components\footnote{In the provided reference, the alternate choice $(\beta_2\beta_1)_X=(\beta_2)_{G_1X}\circ F_2(\beta_1)_X$ is used; this
results in a
different but equivalent monoidal structure.}
\begin{equation}\label{eq:deltabeta}
 (\beta_2 \beta_1)_{X}=F_2 F_1 X \xtickar{\;(\beta_2)_{F_1 X }\;} G_2 F_1 X \xtickar{\;G_2 (\beta_1)_X\;} G_2 G_1 X
\end{equation}
and remaining data obtained in an analogous way;
whereas for modifications $\gamma_1, \gamma_2$ as in
\begin{equation}\label{eq:gammaepsilon}
 \begin{tikzcd}
F_1 \ar[r,Rightarrow,bigtick,"\beta_1"]\ar[d,Rightarrow,"{\sigma_1}"']\ar[dr,phantom,
"\scriptstyle\Ddownarrow{\gamma_1}" ]
& G_1 \ar[d,Rightarrow,"\tau_1"] \\
F'_1 \ar[r,Rightarrow,bigtick,"{\beta'_1}"'] & G'_1
 \end{tikzcd}\qquad
 \begin{tikzcd}
F_2 \ar[r,Rightarrow,bigtick,"\beta_2"]\ar[d,Rightarrow,"\sigma_2"']\ar[dr,phantom,
"\scriptstyle\Ddownarrow\gamma_2" ]
& G_2 \ar[d,Rightarrow,"\tau_2"] \\
F_2'\ar[r,Rightarrow,bigtick,"{\beta'}_2"'] & G'_2
 \end{tikzcd}
\end{equation}
we define $\gamma_1 \ot \gamma_2$ to be the modification $\gamma_2 \gamma_1 \colon \beta_2 \beta_1 \Rrightarrow \beta'_2 \beta'_1$ with vertical source and target
$\sigma_2 \sigma_1$ and $\tau_2 \tau_1$, and $2$-morphism components:
\begin{equation}\label{eq:epsilongamma}
 \begin{tikzcd}[sep=.5in]
F_2 F_1 X
	\ar[d,"{(\sigma_2)}_{F_1 X}"']
	\ar[r,tick,"{(\beta_2)}_{F_1 X}"]
	\ar[dr,phantom,"\Two {(\gamma_2)}_{F_1 X}"] &
G_2 F_1 X
	\ar[r,tick,"{G_2} {(\beta_1)_X}"]
	\ar[d,"{(\sigma_2)}_{F_1 X}" description]
	\ar[dr,phantom,"\Two{(\sigma_2)}_{(\beta_1)_X}"] &
G_2 G_1 X
	\ar[d,"{(\sigma_2)}_{G_1 X}"] \\
F'_2 F_1 X
	\ar[d,"F'_2 {(\sigma_1)}_X"']
	\ar[r,tick,"(\beta'_2)_{F_1 X}"']
	\ar[dr,phantom,"\Two {(\beta'_2)}_{(\sigma_1)_X}" ] &
{G'_2} F_1 X
	\ar[d,"{G'_2} {(\sigma_1)_X}" description]
	\ar[r,tick,"{G'_2} {(\beta_1)_X} "]
	\ar[dr,phantom,"\Two {G'_2} {(\gamma_1)_X} "] &
{G'_2} G_1 X
	\ar[d,"{G'_2} {(\tau_1)_X} "] \\
{F'_2} {F'_1} X
	\ar[r,tick,"{(\beta'_2)}_{F'_1 X}"' ] &
{G'_2} {F'_1} X
	\ar[r,tick,"{G'_2} {(\beta'_1)_X}"'] &
{G'_2} {G'_1} X\mathrlap{ .}
 \end{tikzcd}
\end{equation}
These data endow the category of horizontal $1$-cells and $2$-morphisms with a \emph{non}-strict monoidal structure; for the monoidal constraints, given horizontal transformations $\beta_1 \colon
F_1 \ticktwoar G_1$, $\beta_2 \colon F_2 \ticktwoar G_2$ and $\beta_3 \colon F_3 \ticktwoar G_3$, the composites $\beta_3 (\beta_2 \beta_1)$ and $(\beta_3 \beta_2 )\beta_1$  have respective horizontal $1$-cell components
\begin{equation}\label{eq:associativity}
  G_3 ( G_2 (\beta_1)_X \circ (\beta_2)_{F_1 X}) \circ (\beta_3)_{F_2 F_1 X} \qquad \text{and} \qquad G_3G_2 (\beta_1)_X \circ (G_3(\beta_2)_{F_1 
X} \circ (\beta_3)_{F_2 F_1 X})
\end{equation}
and the desired globular associativity modification $(\beta_1 \otimes \beta_2) \otimes \beta_3 \Rrightarrow \beta_1 \otimes (\beta_2 \otimes \beta_3)$ has components given by the evident globular $2$-isomorphisms between these composites, built from associativity and functoriality of $G_3$. The unit constraints are handled similarly.

Finally, we must provide the globular $2$-isomorphisms $\tau$ and $\eta$ of~\eqref{eq:structure2cells1}. We describe only the case of $\tau$; for
which, consider horizontal transformations $\beta_1 \colon F_1 \ticktwoar G_1$, $\delta_1 \colon G_1 \ticktwoar H_1$,
$\beta_2 \colon  F_2 \ticktwoar G_2$ and $\delta_2 \colon G_2 \ticktwoar H_2$.
The two composite horizontal transformations $(\delta_1 \circ \beta_1) \otimes (\delta_2 \circ \beta_2)$ and $(\delta_1 \otimes \delta_2) \circ
(\beta_1 \otimes \beta_2)$ have respective horizontal $1$-cell components
\begin{equation}\label{eq:tauDbl}
  H_2((\delta_1)_X \circ (\beta_1)_X) \circ ((\delta_2)_{F_1 X} \circ (\beta_2)_{F_1 X}) \qquad \text{and} \qquad
  (H_2(\delta_1)_X \circ (\delta_2)_{G_1 X}) \circ (G_2(\beta_1)_X \circ (\beta_2)_{F_1 X})\mathrlap{ ,}
\end{equation}
which are related by globular $2$-isomorphisms built from functoriality constraints of $H_2$, associativity constraints of $\dc{C}$ and the coherence
$2$-isomorphism $(\delta_2)_{(\beta_1)_X}$ of the horizontal transformation $\delta_2$.
\end{proof}

By looking at horizontal and vertical  monoids (as introduced in \cref{def:hordoublemonoid} and \cref{def:verdoublemonoid}) in the endofunctor double category, we obtaine notions of \emph{horizontal} and \emph{vertical double monad}. These notions differ by the direction of the transformations
for the the multiplication and unit and by their strictness: a horizontal monad induces a pseudomonad on the horizontal bicategory, while a vertical monad induces a $2$-monad on the vertical $2$-category. We shall relate these notions in~\cref{cor:horizontalpseudomonad}.

\begin{defi}[Horizontal double monad] \label{def:horizontalpseudomonad}
Let $\dc{C}$ be a double category.  A \emph{horizontal double monad} on~$\dc{C}$  is
a horizontal monoid in the monoidal double category~$\nc{DblCat}[\dc{C}, \dc{C}]$.  Explicitly, it consists of:
\begin{itemize}
\item a double functor $T\colon\dc{C}\to\dc{C}$;
\item a horizontal transformation $m \co TT \ticktwoar T$, with components
  $m_X\colon TTX\horightarrow TX$,
\begin{equation}\label{eq:mcomponents}
\twocell{TTX}{m_X}{TX}{Tf}{T{X'}}{m_{X'}}{TT{X'}}{TTf}{m_f} \qquad\text{and} \qquad
 \begin{tikzcd}
TTX\ar[r,tick,"TTM"]\ar[d,equal]\ar[drr,phantom,"\Two m_M"] &
TTY\ar[r,tick,"m_Y"] & TY\ar[d,equal] \\
TTX\ar[r,tick,"m_X"'] & TX\ar[r,tick,"TM"'] & TY
   \end{tikzcd}
   \end{equation}
for each object $X$, vertical 1-cell $f\colon X\to {X'}$ and horizontal 1-cell $M \co X \horightarrow Y$;
 \item a horizontal transformation $e \co 1 \ticktwoar T$, with components
 $e_X\colon X\horightarrow TX$,
 \begin{equation}\label{eq:ecomponents}
\twocell{X}{e_X}{TX}{Tf}{T{X'}}{e_{X'}}{{X'}}{f}{e_f} \qquad \text{and} \qquad
   \begin{tikzcd}
   X\ar[r,tick,"M"]\ar[d,equal]\ar[drr,phantom,"\Two e_M"] & Y\ar[r,tick,"e_Y"]
& TY\ar[d,equal] \\
   X\ar[r,tick,"e_X"'] & TX\ar[r,tick,"TM"'] & TY
   \end{tikzcd}
  \end{equation}
for each object $X$, vertical 1-cell $f\colon X\to {X'}$ and horizontal 1-cell $M \co X \horightarrow Y$;
\item invertible modifications $\mathfrak{a}$, $\mathfrak{l}$ and $\mathfrak{r}$ with respective components at $X \in \dc{C}$ given by:
\begin{equation}\label{eq:alr}
\begin{tikzcd}
TTTX\ar[r,tick,"Tm_{X}"]\ar[d,equal]\ar[drr,phantom,"\Two{\mathfrak a}_X"] &
TTX\ar[r,tick,"m_X"] & TX\ar[d,equal] \\
TTTX\ar[r,tick,"m_{TX}"'] & TTX\ar[r,tick,"m_X"'] & TX
\end{tikzcd}\;
\begin{tikzcd}
TX\ar[drr,phantom,"\Two {\mathfrak r}_X"]\ar[d,equal]\ar[r,tick,"e_{TX}"] & TTX\ar[r,tick,"m_X"] & TX\ar[d,equal] \\
TX\ar[rr,tick,"\hid_{TX}"'] && TX
\end{tikzcd}\;
\begin{tikzcd}
TX\ar[r,tick,"Te_X"]\ar[d,equal]\ar[drr,phantom,"\Two {\mathfrak l}_X"] & TTX\ar[r,tick,"m_X"]
& TX\ar[d,equal] \\
TX\ar[rr,tick,"\hid_{TX}"'] && TX
\end{tikzcd}
\end{equation}
\end{itemize}
These data are subject to the axioms of \cref{def:hordoublemonoid}, noting carefully the order-reversal stemming from the fact that $F_1 \otimes F_2 = F_2 F_1$.
\end{defi}

A horizontal double monad is exactly the structure we need to define a horizontal Kleisli double category, as we shall do in~\cref{thm:KlS} below. However,
horizontal double monads involves non-trivial coherence axioms for associativity and unit; it is
therefore useful in practice to have some ways of constructing them from simpler kinds of data. For this purpose,
we recall from~\cite[\S 7]{Adjointfordoublecats} the following definition:

\begin{defi}[Vertical double monad] Let $\dc{C}$ be a double category.
\label{def:doublemonad}
  A \emph{vertical double monad} on
$\dc{C}$ is a vertical monoid in $\nc{DblCat}[\dc{C}, \dc{C}]$.
Explicitly, it consists of the following data:
\begin{itemize}
\item a double endofunctor $T\colon\dc{C}\to\dc{C}$;
\item a vertical transformation $m \co TT \Rightarrow T$, with components $m_X \co TTX \to TX$ and
\begin{equation}\label{eq:mM}
\begin{tikzcd}
TTX\ar[r,tick,"TTM"]\ar[d,"m_X"']\ar[dr,phantom,"\Two m_M"] & TTY\ar[d,"m_Y"] \\
TX\ar[r,tick,"TM"'] & TY
\end{tikzcd}
\end{equation}
for each object $X$ and horizontal 1-cell $M \co X \horightarrow Y$;
\item a vertical transformation $e \co 1 \Rightarrow T$, with components $e_X \co X \to TX$ and
\begin{equation}\label{eq:eM}
\begin{tikzcd}
X\ar[r,tick,"M"]\ar[d,"e_X"']\ar[dr,phantom,"\Two e_M"] & Y\ar[d,"e_Y"] \\
TX\ar[r,tick,"TM"'] & TY
\end{tikzcd}
\end{equation}
for each object $X$ and horizontal 1-cell $M \co X \horightarrow Y$.
\end{itemize}
These data are required to satisfy associativity and unitality conditions, as in \cref{def:verdoublemonoid}.
\end{defi}

The notion of a vertical monad is stricter than that of a horizontal monad and thus easier to exhibit in examples.
Once we have a vertical monad, the following result allows us to enhance it to a horizontal one.

\begin{thm}\label{cor:horizontalpseudomonad}
Let $\dc{C}$ be a double category and $T \co \dc{C} \to \dc{C}$ be a vertical double
monad. Assume that its multiplication $m \colon TT \Rightarrow T$ and unit $e \colon 1_\dc{C}\Rightarrow T$ are
special vertical transformations. Then $T$ induces a
horizontal double monad $(T,\comp{m},\comp{e})$ on $\dc{C}$.
\end{thm}

\begin{proof}
If we consider $T$ as a vertical monoid in $\cat{DblCat}[\dc{C}, \dc{C}]$, \cref{prop:verticalgiveshorizontal} ensures that it induces a horizontal
monoid therein (namely a horizontal double monad) whenever the unit $e$ and multiplication $m$ have companions as vertical transformations. By \cref{lem:2}, this
will happen if and only if they are special. 
\end{proof}

While a direct proof of \cref{cor:horizontalpseudomonad} would
certainly be possible, the more abstract approach we take has the
advantage of being equally applicable to the case of \emph{monoidal}
double monads (\cref{cor:2}), for which a direct approach seems less
practicable. It is to this that we now turn.

\section{Monoidal double monads}
\label{sec:mon-dbl-monad}

In this section, we retread the material of the previous section in the context of \emph{monoidal} double categories, leading to the
notions of a monoidal horizontal and monoidal vertical double monad, and results relating the two. In \cref{sec:kleisli}, we will exploit 
these
notions in order to impose monoidal structure on the Kleisli double category of a horizontal double monad.

As a first step, we show that when $\dc{C}$ is a monoidal double category, we can extend the composition monoidal structure on the endofunctor
double category $\cat{DblCat}[\dc{C},\dc{C}]$ as recalled in \cref{ex:endofun-monoidal-comp}, to a monoidal structure on the \emph{monoidal}
endofunctor double category $\cat{MonDblCat}[\dc{C},\dc{C}]$ of~\cref{thm:functor-dblcat-monoidal}.

\begin{prop}[Composition monoidal structure on {$\cat{MonDblCat}[\dc{C}, \dc{C}]$}]  \label{thm:ps-hom-c-c-monoidal}
Let $\dc{C}$ be a monoidal double category. The composition monoidal structure of the double category $\cat{DblCat}[\dc{C}, \dc{C}]$ lifts to a monoidal structure on the double category
$\cat{MonDblCat}[\dc{C}, \dc{C}]$ of monoidal endofunctors, \emph{pseudomonoidal}
vertical transformations, monoidal horizontal transformations and monoidal modifications.
\end{prop}

\begin{proof}
  We must lift each of the pieces of data exhibited in the proof
  of~\cref{ex:endofun-monoidal-comp} to the monoidal context.
  We first lift the strict monoidal structure on the category of $0$-cells and vertical $1$-cells.
  If $F_1,F_2\colon\dc{C}\to\dc{C}$ are lax
monoidal double endofunctors of $\dc{C}$, then their composite
$F_2F_1$ bears lax monoidal structure with vertical $1$-cell components
\begin{align*}
(F_2F_1)^2_{{X_1},{X_2}}&=F_2F_1{X_1}\ot F_2F_1{X_2}\xrightarrow{F_2^2}F_2(F_1{X_1}\ot
F_1{X_2})\xrightarrow{F_2F_1^2}F_2F_1({X_1}\ot {X_2}),\\
(F_2F_1)^0&=I\xrightarrow{F_2^0}F_2I\xrightarrow{F_2F_1^0}F_2F_1I
\end{align*}
and $2$-morphism components given similarly by
$(F_2F_1)^2_{M,N}=F_2(F_1^2)_{M,N}\cdot (F_2^2)_{F_1M,F_1N}$.
Next, if ${\sigma_1\colon F_1\Rightarrow F_1'}$ and
$\sigma_2\colon F_2\Rightarrow F_2'$ are pseudomonoidal vertical
transformations, then the composite vertical transformation $\sigma_2\sigma_1\colon
F_2F_1\Rightarrow F_2'F_1'$
bears pseudomonoidal structure witnessed by the $2$-morphisms
\begin{displaymath}
(\sigma_2\sigma_1)^2{=}
\begin{tikzcd}[ampersand replacement=\&]
F_2F_1{X_1}\ot F_2F_1{X_2}\ar[d," F_2^2"']\ar[r,tick,"\hid"] \&
F_2F_1{X_1}\ot F_2F_1{X_2}\ar[d," F_2^2"]\ar[r,tick,"\hid"]\ar[ddr,phantom,"\Two{\sigma_2}^2"]
\&
F_2F_1{X_1}\ot F_2F_1{X_2}\ar[d,"\sigma_2\ot\sigma_2"]\ar[r,tick,"\hid"] \&
F_2F_1{X_1}\ot F_2F_1{X_2}\ar[d,"\sigma_2\ot\sigma_2"]\\
F_2(F_1{X_1}\ot F_1{X_2})\ar[r,tick,"\hid"]\ar[d,"F_2 F_1^2"'] \&
F_2(F_1{X_1}\ot F_1{X_2})\ar[d,"\sigma_2"] \&
F_2'F_1{X_1}\ot F_2'F_1{X_2}\ar[d,"F_2'^2"]\ar[r,tick,"\hid"] \&
F_2'F_1{X_1}\ot F_2'F_1{X_2}\ar[d,"F_2'\sigma_1\ot F_2'\sigma_1"] \\
F_2F_1({X_1}\ot {X_2})\ar[d,"\sigma_2"'] \&
F_2'(F_1{X_1}\ot F_1{X_2})\ar[ddr,phantom,"\Two
F_2'\sigma_1^2"]\ar[r,tick,"\hid"]\ar[d,"F_2' F_1^2"] \&
F_2'(F_1{X_1}\ot F_1{X_2})\ar[d,"F_2'(\sigma_1\ot\sigma_1)"] \&
F_2'F_1'{X_1}\ot F_2'F_1'{X_2}\ar[d,"F_2'^2"] \\
F_2'F_1({X_1}\ot {X_2})\ar[r,tick,"\hid"]\ar[d,"F_2'\sigma_1"'] \&
F_2'F_1({X_1}\ot {X_2})\ar[d,"F_2'\sigma_1"] \&
F_2'(F_1'{X_1}\ot F_1'{X_2})\ar[r,tick,"\hid"]\ar[d,"F_2'F_1'^2"] \&
F_2'(F_1'{X_1}\ot F_1'{X_2})\ar[d,"F_2'F_1'^2"] \\
F_2'F_1'({X_1}\ot {X_2})\ar[r,tick,"\hid"'] \&
F_2'F_1'({X_1}\ot {X_2})\ar[r,tick,"\hid"'] \& F_2'F_1'({X_1}\ot {X_2})\ar[r,tick,"\hid"'] \& F_2'F_1'({X_1}\ot {X_2})
\end{tikzcd}
\end{displaymath}
\begin{displaymath}
(\sigma_2\sigma_1)^0{=}\begin{tikzcd}[ampersand replacement=\&]
I\ar[r,tick,"\hid"]\ar[d," F_2^0"']  \& I\ar[r,tick,"\hid"]\ar[d," F_2^0"']\ar[ddr,phantom,"\Two\sigma_2^0"] \& I\ar[dd,"F_2'^0"] \\
F_2I\ar[d,"F_2 F_1^0"']\ar[r,tick,"\hid"]\& F_2I\ar[d,"\sigma_2"'] \&  \\
F_2F_1I\ar[d,"\sigma_2"'] \& F_2'I\ar[r,tick,"\hid"]\ar[d,"F_2' F_1^0"']\ar[ddr,phantom,"\Two F_2'\sigma_1^0"] \&
F_2F_1I\ar[dd,"F_2'F_1'^0"] \\
F_2'F_1I\ar[d,"F_2'\sigma_1"']\ar[r,tick,"\hid"] \& F_2'F_1I\ar[d,"F_2'\sigma_1"'] \&  \\
F_2'F_1'I\ar[r,tick,"\hid"'] \& F_2'F_1'I\ar[r,tick,"\hid"'] \& F_2'F_1'I
 \end{tikzcd}
\end{displaymath}
where the empty squares are horizontal identities
existing due to naturality of $\sigma_2$ and $F_2'^2$. It is direct to check that these pseudomonoidal constraint cells are stable under vertical 
composition, so that we have a functorial tensor product on the category of lax monoidal functors and pseudomonoidal vertical transformations.
Taking this tensor product together with the (strict) monoidal
$1_\dc{C}$ as unit, we obtain a strict monoidal category: indeed, for any lax
monoidal double endofunctors $F_3,F_2,F_1$ of $\dc{C}$, $(F_3F_2)F_1=F_3(F_2F_1)$ as lax monoidal double
functors, since both have coherence vertical $1$-cells given by
\begin{displaymath}
F_3F_2F_1{X_1}\ot F_3F_2F_1{X_2}\xrightarrow{F_3^2}F_3(F_2F_1{X_1}\ot F_2F_1{X_2})\xrightarrow{F_3F_2^2}F_3F_2(F_1{X_1}\ot
F_1{X_2})\xrightarrow{F_3F_2F_1^2}F_3F_2F_1({X_1}\ot {X_2})\mathrlap{ ,}
\end{displaymath}
using the fact that double functors strictly preserve vertical composition; and
correspondingly for the coherence $2$-morphisms.

We now show that the category of horizontal monoidal transformations and monoidal modifications is monoidal. If $\beta_1\colon F_1\ticktwoar G_1$ and 
$\beta_2\colon
F_2\ticktwoar G_2$ are two monoidal horizontal transformations,
then the
horizontal transformation $\beta_2\beta_1\colon F_2F_1\ticktwoar G_2G_1$ given as in \cref{eq:deltabeta}
is monoidal, via the structure $2$-morphisms
\begin{displaymath}
(\beta_2\beta_1)^2=
\begin{tikzcd}[column sep=1in,row sep=.4in]
F_2F_1{X_1}\ot F_2F_1{X_2}\ar[d,equal]\ar[rr,tick,"(G_2\beta_1\circ\beta_2)\otimes (G_2\beta_1\circ\beta_2)"]
\ar[drr,phantom,"\Two\tau"] & &
G_2G_1{X_1}\ot G_2G_1{X_2}\ar[d,equal] \\
F_2F_1{X_1}\ot F_2F_1{X_2}\ar[d,"F_2^2"']\ar[r,tick,"\beta_2\ot\beta_2"]
\ar[dr,phantom,"\Two\beta_2^2"] &
G_2F_1{X_1}\ot G_2F_1{X_2}\ar[d,"G_2^2" description]\ar[r,tick,"G_2\beta_1\ot G_2\beta_1"]
\ar[dr,phantom,"\Two (G_2^2)_{\beta_1,\beta_1}"] &
G_2G_1{X_1}\ot G_2G_1{X_2}\ar[d,"G_2^2"] \\
F_2(F_1{X_1}\ot F_1{X_2})\ar[d,"F_2(F_1^2)"']\ar[r,tick,"\beta_2"']
\ar[dr,phantom,"\Two(\beta_2)_{F_1^2}"] &
G_2(F_1{X_1}\ot F_1{X_2})\ar[r,tick,"G_2(\beta_1\ot\beta_1)"']\ar[d,"G_2(F_1^2)" description]
\ar[dr,phantom,"\Two G_2(\beta_1^2)"] &
G_2(G_1{X_1}\ot G_1{X_2})\ar[d,"G_2(G_1^2)"] \\
F_2F_1({X_1}\ot {X_2})\ar[r,tick,"\beta_2"'] &
G_2F_1({X_1}\ot {X_2})\ar[r,tick,"G_2\beta_1"'] & G_2G_1({X_1}\ot {X_2})
 \end{tikzcd}
 \end{displaymath}
 \begin{displaymath}
 (\beta_2\beta_1)^0=
 \begin{tikzcd}[column sep=.6in, row sep=.4in]
I\ar[d,"F_2^0"']\ar[r,tick,"\hid"]\ar[dr,phantom,"\Two\beta_2^0"] &
I\ar[d,"G_2^0"description]\ar[r,tick,"\hid"] & I\ar[d,"G_2^0"] \\
F_2I\ar[d,"F_2F_1^0"']\ar[r,tick,"\beta_2"]\ar[dr,phantom,"\Two(\beta_2)_{F_1^0}"] &
G_2I\ar[r,tick,"\hid"]\ar[d,"G_2F_1^0"description]\ar[dr,phantom,"\Two G_2(\beta_1^0)"] &
G_2I\ar[d,"G_2G_1^0"] \\
F_2F_1I\ar[r,tick,"\beta_2"'] & G_2F_1I\ar[r,tick,"G_2\beta_1"'] & G_2G_1I\mathrlap{ .}
 \end{tikzcd}
\end{displaymath}

Moreover, given monoidal modifications $\gamma_1,\gamma_2$ as in \cref{eq:gammaepsilon}, their composite
$\gamma_2\gamma_1\colon\beta_2\beta_1\Rrightarrow\beta_2'\beta_1'$ given by~\eqref{eq:epsilongamma} can be verified to satisfy the axioms
\cref{eq:monmodif} that render it monoidal, using, among other things, the monoidality of $\gamma_1$ and $\gamma_2$. The functoriality of this tensor 
product is now inherited from $\cat{DblCat}[\dc{C}, \dc{C}]_1$, given that monoidality is a mere condition on a modification. It is moreover easy to 
check that the  associativity and unitality modifications in $\cat{DblCat}[\dc{C}, \dc{C}]_1$ become monoidal on lifting them to 
$\cat{MonDblCat}[\dc{C}, \dc{C}]_1$, so providing the last pieces of data for the desired monoidal structure.

It remains to lift $\tau$ and $\eta$ from $\cat{DblCat}[\dc{C}, \dc{C}]$ to $\cat{MonDblCat}[\dc{C}, \dc{C}]$: and this is again simply a matter of checking that the modifications obtained from $\cat{DblCat}[\dc{C}, \dc{C}]$ do indeed become monoidal modifications.
\end{proof}

Using this result, and paralleling the developments of~\cref{sec:dbl-monad}, we can now give succinct definitions of the notions
of monoidal horizontal and vertical double monad.

\begin{defi}[Monoidal horizontal double monad]
\label{def:2-h}
Let $\dc{C}$ be a monoidal double category.
 A \emph{monoidal horizontal double monad} on $\dc{C}$ is a horizontal monoid in the monoidal double category $\cat{MonDblCat}[\dc{C}, \dc{C}]$.
Explicitly, it is a horizontal double monad $(T, m, e)$ on $\dc{C}$ in the sense of \cref{def:horizontalpseudomonad} such that:
\begin{itemize}
\item the double functor $T\colon\dc{C}\to\dc{C}$ is lax monoidal, {\em i.e.}, it comes equipped with structure vertical 1-cells $T^2_{X_1,X_2}\colon TX_1\ot
TX_2\to T(X_1\ot X_2)$, $T^0\colon I\to TI$ and 2-morphisms
\begin{displaymath}
\begin{tikzcd}[column sep=.5in]
TX_1\ot TX_2\ar[d,"T^2_{X_1,X_2}"']\ar[dr,phantom,"\Two T^2_{M,N}"]\ar[r,tick,"TM\ot TN"] & TY_1\ot TY_2\ar[d,"T^2_{Y_1,Y_2}"] \\
T(X_1\ot X_2)\ar[r,tick,"T(M\ot N)"'] & T(Y_1\ot Y_2)
\end{tikzcd}
\end{displaymath}
satisfying the axioms of \cref{def:monoidaldoublefunctor};
\item the horizontal transformation $m \co TT \ticktwoar T$ is monoidal, {\em i.e.}, it comes equipped with structure 2-morphisms:
\begin{equation}\label{eq:m2}
\begin{tikzcd}[column sep=.5in]
TTX_1\ot TTX_2\ar[r,tick,"m_X\ot m_Y"]\ar[d,"T^2_{TX_1,TX_2}"']\ar[ddr,phantom,"\Two m^2"] & TX_1\ot TX_2\ar[dd,"T^2_{X_1,X_2}"] \\
T(TX_1\ot TX_2)\ar[d,"T(T^2_{X_1,X_2})"'] & \\
TT(X_1\ot X_2)\ar[r,tick,"m_{X_1\ot X_2}"'] & T(X_1\ot X_2)
\end{tikzcd}\qquad\text{and} \qquad
\begin{tikzcd}
I\ar[r,tick,"\hid_I"]\ar[d,"T^0"']\ar[ddr,phantom,"\Two m^0"] & I\ar[dd,"T^0"] \\
TI\ar[d,"T(T^0)"'] & \\
TTI\ar[r,tick,"m_I"'] & TI
\end{tikzcd}
\end{equation}
satisfying the axioms of \cref{def:monhortransf};
\item the horizontal transformation $e \co 1 \ticktwoar T$ is monoidal, {\em i.e.}~comes with structure 2-morphisms
\begin{equation}\label{eq:e2}
\begin{tikzcd}
X_1\ot X_2\ar[r,tick,"e_{X_1}\ot e_{X_2}"]\ar[d,equal]\ar[dr,phantom,"\Two e^2"] & TX_1\ot TX_2\ar[d,"T^2_{X_1,X_2}"] \\
X_1 \ot X_2 \ar[r,tick,"e_{X_1\ot X_2}"'] & T(X_1\ot X_2)
\end{tikzcd}\qquad\text{and} \qquad
\begin{tikzcd}
I\ar[r,tick,"\hid_I"]\ar[d,equal]\ar[dr,phantom,"\Two e^0"] & I\ar[d,"T^0"] \\
I\ar[r,tick,"e_I"'] & TI
\end{tikzcd}
\end{equation}
satisfying the axioms of \cref{def:monhortransf};
\item the modifications $\mathfrak{a}$, $\mathfrak{l}$, $\mathfrak{r}$ of \cref{eq:alr} are monoidal as in \cref{def:monmodif}.
\end{itemize}
\end{defi}

\begin{defi}[Pseudomonoidal vertical double monad]
\label{def:2}
Let $\dc{C}$ be a monoidal double category.
A \emph{pseudomonoidal vertical double monad} on $\dc{C}$  a vertical monoid in $\cat{MonDblCat}[\dc{C}, \dc{C}]$.
Explicitly, it is a vertical double monad $(T,m,e)$ on $\dc{C}$ in the sense of \cref{def:doublemonad}, such that:
  \begin{itemize}
  \item the double functor $T\co \dc{C}\to\dc{C}$ is lax monoidal, as in \cref{def:monoidaldoublefunctor};
  \item the vertical transformation $m \co TT \Rightarrow T$ is pseudomonoidal, \emph{i.e.}, it comes equipped with 2-morphisms
\begin{equation}\label{eq:here-2}
  \begin{tikzcd}[sep=.37in]
TTX_1\ot TTX_2\ar[r,tick,"\hid"]\ar[d,"T^2_{TX,TY}"']\ar[dddr,phantom,"\Two m^2"] &
TTX_1\ot TTX_2\ar[d,"m_{X_1} \ot m_{X_2}"] \\
T(TX_1\ot TX_2)\ar[d,"T(T^2_{X_1,X_2})"'] & TX_1\ot TX_2\ar[dd,"T^2_{X_1,X_2}"] \\
TT(X_1\ot X_2)\ar[d,"m_{X_1\ot X_2}"'] & & \\
T(X_1\ot X_2)\ar[r,tick,"\hid"'] & T(X_1\ot X_2)
 \end{tikzcd}\qquad
\begin{tikzcd}[sep=.4in]
I \ar[d, "T^0"']\ar[r, tick, "\hid"]\ar[dddr,phantom,"\Two m^0"]  & I\ar[ddd, "T^0"]  \\
 TI\ar[d,"T(T^0)"'] & \\
TTI \ar[d, "m_I"'] & \\
TI \ar[r, tick, "\hid"'] & TI
   \end{tikzcd}
   \end{equation}
satisfying the axioms of \cref{def:montransf};
  \item the vertical transformation $e \co 1 \Rightarrow T$ is pseudomonoidal , \emph{i.e.}, it comes equipped with 2-morphisms
  \begin{equation}\label{eq:here-1}
  \begin{tikzcd}[sep=.37in]
X_1\ot X_2\ar[r,tick,"\hid"]\ar[dd,"e_{X_1\ot X_2}"']\ar[ddr,phantom,"\Two e^2"] &
X_1\ot X_2\ar[d,"e_{X_1} \ot e_{X_2}"] \\
& TX_1\ot TX_2\ar[d,"T^2_{X_1,X_2}"] \\
T(X_1\ot X_2)\ar[r,tick,"\hid"'] & T(X_1\ot X_2)
 \end{tikzcd}\qquad
\begin{tikzcd}[sep=.4in]
I \ar[r, tick, "\hid"]\ar[dr,phantom,"\Two e^0"] \ar[d,"e_I"'] & I \ar[d, "T^0"] \\
TI \ar[r, tick, "\hid"'] & TI
   \end{tikzcd}
   \end{equation}
satisfying the axioms of \cref{def:montransf};
\item the pseudomonoidal structures of the composites $m \circ Tm$ and $m \circ mT \colon TTT \Rightarrow T$ are equal, while the pseudomonoidal structures of $m \circ Te$ and $m \circ eT \colon T \Rightarrow T$ are both trivial.
  \end{itemize}
\end{defi}

For a horizontal monad that arises from a vertical one via \cref{cor:horizontalpseudomonad},
we are naturally interested
in conditions on the vertical monad such that the induced horizontal monad is monoidal.
Thankfully, the conditions under which a pseudomonoidal
vertical double monad induces a monoidal horizontal double
monad are the same as for the non-monoidal
case of \cref{cor:horizontalpseudomonad}, as the next theorem shows.

\begin{thm}
  \label{cor:2}
Let $\dc{C}$  be a monoidal double category and $(T, m, e)$ be a
pseudomonoidal vertical double monad  on $\dc{C}$.
Assume that the underlying vertical
transformations $m \colon TT \Rightarrow T$ and $e \colon
1_\dc{C} \Rightarrow T$ are special.
Then $(T, m, e)$ induces a monoidal
horizontal double monad $(T, \comp{m}, \comp{e})$ on $\dc{C}$.
\end{thm}

\begin{proof}
If we consider $T$ as a vertical monoid in $\cat{MonDblCat}[\dc{C}, \dc{C}]$,
\cref{prop:verticalgiveshorizontal} ensures that it induces a horizontal pseudomonoid therein (namely a monoidal horizontal monad) when $m$ and $e$
have companions as pseudomonoidal vertical transformations. By \cref{lem:1}, this is true if and only if $m$ and $e$ are special
(\cref{def:special-vertical-transformation}). For example, the unit of the induced monoidal horizontal double monad structure on $T$ is the
horizontal transformation $\comp{e}\colon1\ticktwoar T$ which becomes monoidal with structure cells
\begin{equation}\label{eq:here2}
\begin{tikzcd}[column sep=.6in]
X_1 \ot X_2 \ar[r,tick,"\comp{e}_{X_1} \ot\comp{e}_{Y_2}"]\ar[dr,phantom,"\Two"]\ar[d,equal] & TX_1\ot TX_2\ar[d,"T^2_{X_1,X_2}"] \\
X_1 \ot X_2\ar[r,tick,"\comp{e}_{X_1 \ot X_2}"'] & T(X_1\ot X_2)
\end{tikzcd}\qquad
\begin{tikzcd}
I\ar[r,tick,"\hid_I"]\ar[d,equal] \ar[dr,phantom,"\Two"]& I\ar[d,"T^0"] \\
I\ar[r,tick,"\comp{e}_I"'] & TI
\end{tikzcd}
\end{equation}
that bijectively correspond, under transpose operations, to those of \cref{eq:here-1}.
\end{proof}

While a direct proof of \cref{cor:2} should be possible, the construction of all the data for a monoidal horizontal double monad from that of a
pseudomonoidal vertical double monad using companions, let alone the verification of the coherence axioms, would be a daunting task. It is at this point that the advantage of our abstract view becomes clear;
as an added bonus, the proofs of \cref{cor:horizontalpseudomonad,cor:2}
become essentially the same.

\section{Monoidal Kleisli double categories}
\label{sec:kleisli}

In this section, we first introduce the (horizontal) Kleisli double category $\Kl(T)$ for a
horizontal double monad $T$ (\cref{def:horizontalpseudomonad}) on a double category $\dc{C}$, with a particularly important case being where $T$ is induced from a vertical double monad (\cref{def:doublemonad}) as in \cref{cor:horizontalpseudomonad}.

We next consider what happens when the double category $\dc{C}$ is monoidal and the double monad $T$ is also monoidal. We would naturally expect the monoidal structure of $\dc{C}$ to extend to $\Kl(T)$, just as happens with an ordinary monoidal monad on an ordinary monoidal category. However, because the monoidal constraint data for a horizontal double monad does not point exclusively in the horizontal direction, things are slightly more subtle. To even obtain monoidal structure we must assume certain companions exist, and even then, this structure is only \emph{oplax} monoidal in general (\cref{thm:oplaxmonoidalKlT}). Again, the situation where $T$ is induced from a vertical double monad will be important, and in this special case, we describe sufficient conditions for this oplax monoidal structure on $\Kl(T)$ to be normal oplax (\cref{prop:normaloplaxdouble}) or (pseudo) monoidal (\cref{cor:pseudomonoidalKl}).

We begin with the construction of the horizontal Kleisli double category of a horizontal double monad. This construction is essentially contained in \cite{CruttwellShulman}; there, the authors start from a \emph{vertical} double monad $(T,m,e)$, and define from it a \emph{horizontal} Kleisli double category (Definition~4.1 of \emph{op.~cit.}) which in general is only a so-called \emph{virtual} double category. These are weaker structures than double categories, in which horizontal $1$-cells do not compose, but instead are formed into a structure of ``multi-$2$-morphisms''; however, \cite[Theorem~A.8]{CruttwellShulman} shows that, when the vertical transformations $e \colon 1 \Rightarrow T$ and $m \colon TT \Rightarrow T$ are \emph{special}, this virtual double category is in fact a double category. In this case, the horizontal Kleisli double category of \emph{loc.~cit.} can be obtained as follows: first apply \cref{cor:horizontalpseudomonad} to form the horizontal double monad $(T, \comp m, \comp e)$ associated to $(T,m,e)$; and then apply the following result.

\begin{thm}\label{thm:KlS}
Let $\dc{C}$ be a double category and $(T,m,e)$ be a horizontal double monad on it.
There is a double category $\Kl(T)$, called the \emph{horizontal Kleisli double category} of $\dc{C}$, wherein:
\begin{itemize}
\item \textbf{objects} are objects of $\dc{C}$;
\item \textbf{vertical 1-cells} are vertical 1-cells of $\dc{C}$;
\item \textbf{horizontal 1-cells} $M \co X \rightsquigarrow Y$ are
  horizontal 1-cells $X \tor TY$ of $\dc{C}$;
\item \textbf{2-morphisms} as to the left below, are the 2-morphisms of $\dc{C}$ as to the right:
  \begin{equation*}
    \cd{
      {X} \ar@{~>}[r]^-{M} \ar[d]_{f}
      \dtwocell{dr}{\phi} &
      {Y} \ar[d]^{g} \\
      {{X'}} \ar@{~>}[r]_-{M'} &
      {{Y'}}
    } \qquad \qquad
    \cd{
      {X} \ar|@{|}[r]^-{M} \ar[d]_{f} \dtwocell{dr}{\phi} &
      {TY} \ar[d]^{Tg} \\
      {{X'}} \ar|@{|}[r]_-{M'} &
      {T{Y'} \mathrlap{.}}
    }
  \end{equation*}
\end{itemize}
\end{thm}

\begin{proof} Vertical composition in $\Kl(T)$ is the same as  in $\dc{C}$;
horizontal composition of Kleisli 1-cells $M \colon X
\rightsquigarrow Y$
  and $N \colon Y \rightsquigarrow Z$ is given by
  \begin{equation}\label{eq:kcirc}
N \circ_\Kl M \defeq X\xtickar{M}TY\xtickar{TN}TTZ\xtickar{m_Z}TZ\mathrlap{;}
  \end{equation}
while horizontal pasting of Kleisli 2-morphisms $\phi$ and $\psi$ is given by
  \begin{equation*}
    \cd[@C+2em]{
      X \ar[d]_-{f} \ar|@{|}[r]^{M}
      \dtwocell{dr}{\phi} &
      TY \ar|@{|}[r]^{TN} \ar[d]|-{Tg}
      \dtwocell{dr}{T\psi} &
      TTZ \ar|@{|}[r]^{m_Z} \ar[d]|-{TTh}
      \dtwocell{dr}{m_h} &
      TZ \ar[d]^-{Th} \\
      {X'} \ar|@{|}[r]_{M'} &
      T{Y'} \ar|@{|}[r]_{TN'} &
      TTZ' \ar|@{|}[r]_{m_{Z'}} &
      TZ'\rlap{ .}
    }
  \end{equation*}
The horizontal identity 1-cell on $X$ is
\begin{equation}\label{eq:idK}
\begin{tikzcd}
\hid_X^\Kl\defeq X\ar[r,tick,"e_X"] &  TX
\end{tikzcd}
\end{equation}
and the horizontal identity 2-morphism on $f \colon X \rightarrow {X'}$ is
  \begin{equation*}
    \cd{
      X \ar[d]_-{f} \ar|@{|}[r]^-{e_X} \dtwocell{dr}{e_f}  & TX
\ar[d]^-{Tf} \\
      X' \ar|@{|}[r]_-{e_{X'}} & TX'\mathrlap{ .}
    }
  \end{equation*}
It is easy to see that horizontal composition of $2$-morphisms is
vertically functorial, and so it remains to give the coherence
constraints. Given Kleisli 1-cells $M \colon X \rightsquigarrow Y$
  and $N \colon Y \rightsquigarrow Z$ and $P \colon Z \rightsquigarrow
  W$, the associativity constraint is given by the following pasting,
  in which the horizontal 1-cell at the top is $(P \circ_\Kl N) \circ_\Kl M$ and the one
  at the bottom is $P \circ_\Kl (N \circ_\Kl M)$:
\begin{displaymath}
\begin{tikzcd}
X
\ar[ddd,equal]
\ar[r,tick,"M"] &
TY\ar[drrr,phantom,"\Two \xi"]\ar[rrr,tick,
"T({m_W}\circ TP\circ N)"]\ar[d,equal] &&& TTW\ar[d,equal]
\ar[r,tick,"{m_{W}}"] & TW\ar[d,equal] \\
& TY\ar[dd,equal]\ar[r,tick,"TN"'] &TTZ\ar[d,equal]\ar[r,tick,"TTP"']
& TTTW\ar[drr,phantom,"\Two{\mathfrak
a}"]\ar[r,tick,"T{m_W}"']\ar[d,equal] &
TTW\ar[r,tick,"{m_{W}}"'] & TW\ar[d,equal]\\
&&
TTZ\ar[drr,phantom,"\Two{m_P}"]\ar[d,equal]\ar
[r,tick,"TTP"'] & TTTW\ar[r,tick,"{m_{TW}}"'] &
TTW\ar[d,equal]\ar[r,tick,"{m_{W}}"'] & TW\ar[d,equal]\\
X
\ar[r,tick,"M"'] & TY\ar[r,tick,"TN"'] &
TTZ\ar[r,tick,"{m_Z}"'] & TZ\ar[r,tick,"TP"'] &
TTW\ar[r,tick,"{m_W}"'] & TW\mathrlap{.}
\end{tikzcd}
\end{displaymath}
The unit constraints are as follows, where the horizontal top 1-cells are $\hid^\Kl_Y \circ_{\Kl} M$ and
$M \circ_\Kl \hid^\Kl_X$:
\begin{displaymath}
 \begin{tikzcd}
X \ar[r,tick,"M"]
\ar[d,equal] &
TY\ar[r,tick,"Te_Y"]\ar[d,equal]\ar[drr,phantom,"\Two\mathfrak l"] &
TTY \ar[r,tick,"{m_Y}"] & TY \ar[d,equal]\\
X\ar[r,tick,"M"'] & TY \ar[rr,tick,"\hid_{TY}"']&& TY
\end{tikzcd}\qquad
\begin{tikzcd}
X\ar[d,equal]\ar[drr,phantom,"\Two (e_M)^{\mi 1}"]
\ar[r,tick,"e_X"] & TX \ar[r,tick,"TM"] &
TTY\ar[d,equal]
\ar[r,tick,"m_Y"] & TY \ar[d,equal]\\
X \ar[r,tick,"M"]\ar[d,equal] & TY\ar[d,equal]\ar[drr,phantom,"\Two\mathfrak r"] \ar[r,tick,"e_{TY}"] & TTY \ar[r,tick,"m_Y"] & TY
\ar[d,equal]\\
X  \ar[r,tick,"M"']  & TY \ar[rr,tick,"\hid_{TY}"'] & & TY  \mathrlap{.}
\end{tikzcd}
\end{displaymath}
Above, the 2-cells labelled $\mathfrak a,\mathfrak l,\mathfrak r$ are as in \cref{eq:alr} and the components of $m,e$ are as in
(\ref{eq:mcomponents}, \ref{eq:ecomponents}). The coherence axioms follow by the usual argument for a Kleisli
bicategory, {\em cf.}~\cite{CruttwellShulman,FioreM:relpkb}.
\end{proof}

We have not ascribed any kind of universal property to the construction of the Kleisli double category, and for our purposes we do not need to; however, if we were to do so, then, following~\cite{FormalTheoryMonadsI}, we would express it in terms of \emph{universal opalgebra} structure on the canonical embedding of $\dc{C}$ into $\Kl(T)$:
\begin{defi}
  \label{def:embedding}
  Let $\dc{C}$ be a double category and $(T,m,e)$ be a horizontal double monad on it. The \emph{canonical embedding} $F_T \colon \dc{C} \rightarrow \Kl(T)$ is the double functor which is the identity on objects and vertical $1$-cells; sends a horizontal $1$-cell $M \co X \horightarrow Y$ to $e_Y \circ M \co X \rightsquigarrow Y$, and correspondingly for 2-morphisms between horizontal $1$-cells.
\end{defi}

Applying \cref{lem:Ff} to this canonical embedding, we immediately obtain the following result concerning companions in Kleisli double categories ({\em cf.}~\cite[Proposition~7.5]{CruttwellShulman}):

\begin{prop}\label{prop:Klcompanions}
Let $\dc{C}$ be a double category, $T$ a horizontal double monad on $\dc{C}$, and
$f \co X \to X'$ a vertical 1-cell of $\dc{C}$. If $f$ has a companion as a vertical 1-cell of $\dc{C}$, then it
has a companion also as a vertical 1-cell of $\Kl(T)$. \hfill \qedsymbol
\end{prop}

We now consider the Kleisli double category when $\dc{C}$ is a monoidal double category and $T$ is a monoidal horizontal double monad (\cref{def:2-h}). As discussed above, it does not seem to be true in general that the monoidal structure of $\dc{C}$ will extend to $\Kl(T)$; however, under mild assumptions which
are satisfied in our applications, we do obtain at least an oplax monoidal structure (\cref{def:oplaxdoublecat}) on $\Kl(T)$:

\begin{thm}\label{thm:oplaxmonoidalKlT}
Let $\dc{C}$ be a monoidal double category and $T$ a
 monoidal horizontal double monad on $\dc{C}$. If the vertical 1-cell $T^0\colon I \rightarrow TI$ and each vertical $1$-cell
$T^2_{X_1,X_2} \colon TX_1 \otimes TX_2 \rightarrow T(X_1 \otimes X_2)$
  has a
companion, then the monoidal structure of $\dc{C}$ induces 
an oplax monoidal structure on $\Kl(T)$.
\end{thm}

\begin{proof}
The monoidal structure on the category of objects and vertical $1$-cells $\Kl(T)_0 = \dc{C}_0$ is inherited from $\dc{C}$. For the monoidal structure on the category $\Kl(T)_1$ of horizontal $1$-cells and $2$-morphisms, we define the tensor product of $M_1 \colon X_1 \rightsquigarrow Y_1$ and $M_2 \colon X_2 \rightsquigarrow Y_2$ and the monoidal unit $J$
to be:
\begin{displaymath}
M_1 \kot M_2 \defeq X_1 \otimes X_2 \xtickar{M_1 \otimes M_2 }TY_1 \otimes TY_2 \xtickar{\comp{T^2_{\mathrlap{Y_1,Y_2}}}\quad}T(Y_1 \otimes Y_2),
\qquad
J \defeq I\xtickar{\comp{T^0}}TI\mathrlap{ ;}
\end{displaymath}
while the binary tensor product of $2$-morphisms is given by:
\begin{displaymath}
 \begin{tikzcd}[column sep=.3in,]
X_1\ar[dr,phantom,"\Two\phi"]\ar[r,tick,"M_1"]\ar[d,"f_1"'] &
TY_1,\ar[d,"Tg_1"] &
X_2 \ar[dr,phantom,"\Two\psi"]\ar[r,tick,"M_2"]\ar[d,"f_2"'] &
TY_2\ar[drr,phantom,"\mapsto"]\ar[d,"Tg_2"] &&
X_1 \ot X_2 \ar[dr,phantom,"\Two\phi\ot\psi"]\ar[r,tick,"M_1\ot M_2"]\ar[d,"f_1\ot
f_2"'] & TY_1\ot TY_2 \ar[r,tick,"\comp{T^2_{\mathrlap{Y_1,Y_2}}}\quad"]\ar[d,"Tg_1\ot
Tg_2"description]\ar[dr,phantom,"\Two"] & T(Y_1 \ot Y_2)\ar[d,"T(g_1 \ot g_2)"] \\
X'_1\ar[r,tick,"M_1'"'] & TY'_1, & X'_2 \ar[r,tick,"M_2'"'] & TY'_2 && X'_1 \ot X'_2 \ar[r,tick,"M'_1 \ot M'_2"'] & TY'_1 \ot
TY'_2 \ar[r,tick,"\comp{T^2_{\mathrlap{Y'_1,Y'_2}}}\quad"']
&
T(Y'_1 \ot Y'_2)
 \end{tikzcd}
\end{displaymath}
where the right-hand 2-morphism is a companion transpose of the equality $T(g_1\ot g_2) \circ T^2_{Y_1, Y_2} = T^2_{Y_1', Y_2'} \circ Tg_1 \ot Tg_2$ of vertical $1$-cells expressing naturality of $T^2$.
The associativity constraint is given by the following pasting, where the horizontal composite at the top is $(M_1 \kot M_2) \kot M_3$
and the one at the bottom is $M_1 \kot (M_2 \kot M_3)$:
\begin{displaymath}
\begin{tikzcd}[column sep=.5in]
(X_1 \otimes X_2) \otimes X_3
	\ar[d,equal]
	\ar[rr,tick,"(\comp{T^2} \circ (M_1 \otimes M_2)) \otimes M_3"]
	\ar[drr,phantom,"\Two\tau\cref{eq:structure2cells1}"]
	&&
T(Y_1 \otimes Y_2) \otimes TY_3
	\ar[r,tick,"\comp{T^2}"]\ar[d,equal] &
T((Y_1 \otimes Y_2) \otimes Y_3)
	\ar[d,equal] \\
(X_1 \otimes X_2) \otimes X_3
	\ar[d,"\alpha_{X_1,X_2,X_3}"']
	\ar[r,tick,"(M_1 \otimes M_2) \otimes X_3"]
	\ar[dr,phantom,"\Two\alpha_{M_1,M_2,M_3}\cref{eq:associativitycomponents}"] &
(TY_1 \otimes TY_2) \otimes TY_3
	\ar[drr,phantom,"\Two(*)"]
	\ar[d,"\alpha_{TY_1,TY_2,TY_3}"]
	\ar[r,tick,"\comp{T^2}\otimes\hid"] &
T(Y_1 \otimes Y_2) \otimes TY_3
	\ar[r,tick,"\comp{T^2}"] & T((Y_1 \otimes Y_2) \otimes Y_3)
	\ar[d,"T\alpha_{Y_1,Y_2,Y_3}"] \\
X_1 \otimes (X_2 \otimes X_3)
	\ar[drr,phantom,"\Two\tau^{\textrm{-}1}\cref{eq:structure2cells1}"]
	\ar[d,equal] \ar[r,tick,"M_1 \otimes (M_2 \otimes M_3)"'] &
TY_1 \otimes (TY_2 \otimes TY_3)
	\ar[r,tick,"\hid\otimes\comp{T^2}"'] &
TY_1  \otimes T(Y_2 \otimes Y_3)
	\ar[r,tick,"\comp{T^2}"'] \ar[d,equal] &
T(Y_1 \otimes (Y_2 \otimes Y_3))
	\ar[d,equal] \\
X_1 \otimes (X_2 \otimes X_3)
	\ar[rr,tick,"M_1 \otimes(\comp{T^2}\circ (M_2 \otimes M_3))"'] &&
TY_1 \otimes T(Y_2 \otimes Y_3)
	\ar[r,tick,"\comp{T^2}"'] &
T(Y_1 \otimes (Y_2 \otimes Y_3))\mathrlap{.}
\end{tikzcd}
\end{displaymath}
Here, the $2$-cell $(*)$ is the transpose of the equality $(T\alpha)T^2(T^2\ot\vid)=T^2(\vid\ot T^2)\alpha$ of vertical $1$-cells expressing the associativity axiom for the vertical part of the monoidal double functor $T$; note that $(*)$ is
invertible (as all other 2-cells in the above composite) by \cref{prop:companionomnibus}\ref{omni-iii} and \ref{omni-vi}. The unit constraints are formed similarly as
follows, where
the top horizontal 1-cells are $J \kot M$ and $M \kot J$, respectively:
\begin{equation}\label{eq:Klunitors}
\begin{tikzcd}
I\ot X
\ar[rr,tick,"\comp{T^0}\ot M"]\ar[d,equal]\ar[drr,phantom,"\Two\tau"] &&
TI\ot TY\ar[r,tick,"\comp{T^2_{\mathrlap{I,Y}}}\quad"]\ar[d,equal] & T(I\ot Y)\ar[d,equal] \\
I\ot X\ar[d,"\lambda"']\ar[r,tick,"\hid\ot M"]\ar[dr,phantom,"\Two\lambda_M"] & I\ot
TY\ar[d,"\lambda"]\ar[drr,phantom,"\Two(*)"]\ar[r,tick,"\comp{T^0}\ot\hid"] &
TI\ot TY\ar[r,tick,"\comp{T^2_{\mathrlap{I,Y}}}\quad"] & T(I\ot Y)\ar[d,"T\lambda"] \\
X\ar[r,tick,"M"'] & TY\ar[rr,tick,"\hid"'] && TY \mathrlap{,}
\end{tikzcd}\quad
\begin{tikzcd}
X\ot I\ar[rr,tick,"M\ot \comp{T^0}"]\ar[drr,phantom,"\Two\tau"]\ar[d,equal]
&& TY\ot
TI\ar[r,tick,"\comp{T^2_{\mathrlap{Y,I}}}\quad"]\ar[d,equal] & T(Y\ot I)\ar[d,equal] \\
X\ot I\ar[d,"\rho"']\ar[dr,phantom,"\Two\rho_M"]\ar[r,tick,"M\ot \hid"] & TY\ot
I\ar[d,"\rho"]\ar[r,tick,"\hid\ot\comp{T^0}"]\ar[drr,phantom,"\Two(*)"] &
TY\ot TI\ar[r,tick,"\comp{T^2_{\mathrlap{Y,I}}}\quad"] & T(Y\ot I)\ar[d,"T\rho"] \\
X\ar[r,tick,"M"'] & TY\ar[rr,tick,"\hid"'] && TY \mathrlap{.}
\end{tikzcd}
\end{equation}
Here, the 2-cells $(*)$ are the transposes of the unit axioms for the vertical part of the monoidal double functor $T$, and are invertible by the same reasoning as before.

The oplax monoidal structure map to the left of \cref{eq:structure2cells1} has the form:
  \begin{equation*}
    \cd{
      X_1 \otimes X_2 \ar@{=}[d]_-{} \dtwocell{drr}{\tau}
\ar@{~>}[rr]^-{(N_1 \circ_\Kl M_1) \kot (N_2 \circ_\Kl M_2)}& &
      Z_1 \otimes Z_2 \ar@{=}[d]_-{} \\
      X_1 \otimes X_2 \ar@{~>}[r]_-{M_1 \kot M_2} &
      Y_1 \otimes Y_2 \ar@{~>}[r]_-{N_1 \kot N_2} &
      Z_1 \otimes Z_2
    }
  \end{equation*}
where $\circ_\Kl$ is defined as in \cref{eq:kcirc}; we obtain it as the pasting composite
\begin{equation}\label{eq:interchange}
\scalebox{.75}{
 \begin{tikzcd}[column sep=.45in,ampersand replacement=\&]
X_1\ot X_2\ar[drrr,phantom,"\Two\cong"]\ar[rrr,tick,"{(m \circ TN_1 \circ
M_1) \otimes (m \circ TN_2 \circ M_2)}"]\ar[d,equal] \&\&\& TZ_1\ot
TZ_2\ar[d,equal]\ar[rr,tick,"\comp{ T^2}"] \&\& T(Z_1\ot Z_2)\ar[d,equal]
\\
 X_1\ot X_2\ar[r,tick,"M_1\ot M_2"]\ar[ddd,equal] \& TY_1\ot
TY_2\ar[d,equal]\ar[r,tick,"TN_1\ot TN_2"] \& TTZ_1\ot
TTZ_2\ar[drrr,phantom,"\Two\comp{m^2}"]\ar[d,equal]\ar[r,tick,"m\ot m"]
\& TZ_1\ot TZ_2\ar[rr,tick,"\comp{ T^2}"] \&\& T(Z_1\ot Z_2)\ar[d,equal] \\
\& TY_1\ot TY_2\ar[d,equal]\ar[r,tick,"TN_1\ot
TN_2"']\ar[drr,phantom,"\Two\comp{T^2_{\mathrlap{N_1,N_2}}}\quad"] \& TTZ_1\ot
TTZ_2\ar[r,tick,"\comp{ T^2}"'] \& T(TZ_1\ot
TZ_2)\ar[d,equal]\ar[r,tick,"T\comp{ T^2}"'] \& TT(Z_1\ot
Z_2)\ar[r,tick,"m"'] \& T(Z_1\ot Z_2)\ar[d,equal] \\
\& TY_1\ot TY_2\ar[r,tick,"\comp{ T^2}"']\ar[d,equal] \& T(Y_1\ot
Y_2)\ar[d,equal]\ar[drr,phantom,"\Two\cong"]\ar[r,tick,
"T(N_1\ot N_2)"'] \& T(TZ_1\ot TZ_2)\ar[r,tick,"T\comp{ T^2}"'] \& TT(Z_1\ot
Z_2)\ar[d,equal]\ar[r,tick,"m"'] \& T(Z_1\ot Z_2)\ar[d,equal]\\
X_1\ot X_2\ar[r,tick,"M_1\ot M_2"'] \& TY_1\ot TY_2\ar[r,tick,"\comp{ T^2}"'] \&
T(Y_1\ot Y_2)\ar[rr,tick,"T(\comp{ T^2}\circ(N_1\ot N_2))"'] \&\& TT(Z_1\ot
Z_2)\ar[r,tick,"m"'] \& T(Z_1\ot Z_2)
 \end{tikzcd}}
\end{equation}
where the top left isomorphism is the monoidal interchange \cref{eq:structure2cells1} applied twice, the 2-morphism labelled $\comp{m^2}$ is a companion transpose of
the structure 2-morphism $m^2$ of the monoidal horizontal transformation $m$
as in \cref{eq:m2}, and the 2-morphism $\comp{T^2_{\mathrlap{N_1,N_2}}}\quad\;\;$ is a companion transpose of the component $T^2_{N_1,N_2}$ of the lax monoidal structure on the double
functor $T$ as in
\cref{eq:F2}.

The globular 2-morphism $\eta$ to the right of \cref{eq:structure2cells1} is obtained as follows, where the horizontal 1-cell
at the top is $\hid^\Kl_{X_1} \kot\hid^\Kl_{X_2}$ and the one at the bottom is $\hid^\Kl_{X_1\ot X_2}$:
\begin{equation}\label{eq:structure2}
 \begin{tikzcd}
 X_1\ot X_2\ar[d,equal]\ar[drr,phantom,"\Two\comp{e^2}"]
 \ar[r,tick,"e_{X_1} \ot e_{X_2}"] & TX_1\ot
TX_2\ar[r,tick,"\comp{T^2_{\mathrlap{X_1,X_2}}}\quad"] & T(X_1\ot X_2)\ar[d,equal] \\
X_1\ot X_2\ar[rr,tick,"e_{X_1\ot X_2}"']
&& T(X_1\ot X_2) \mathrlap{.}
 \end{tikzcd}
\end{equation}
where $\hid^\Kl$ is defined as in \cref{eq:idK}.
Here, the $2$-morphism filling the square is a companion transpose of
the structure 2-morphism $e^2$ of the monoidal horizontal transformation $e$ as in \cref{eq:e2}.
Finally, the globular structure 2-morphisms $\delta$ and $\iota$ of \cref{eq:structure2cells2} are defined to be
\begin{equation}\label{eq:structure34}
\delta=\begin{tikzcd}
I\ar[rrr,tick,"\comp{ T^0}"]\ar[d,equal]\ar[drrr,phantom,"\Two\comp{m^0}
"] &&& TI\ar[d,equal] \\
 I\ar[r,tick,"\comp{ T^0}"'] & TI\ar[r,tick,"T\comp{ T^0}"'] &
TTI\ar[r,tick,"m_{I}"'] & TI
 \end{tikzcd}\qquad
 \iota=\begin{tikzcd}
 I\ar[r,tick,"\comp{ T^0}"]\ar[d,equal]\ar[dr,phantom,"\Two\comp{e^0}"]
& TI\ar[d,equal] \\
 I\ar[r,tick,"e_I"'] & TI
 \end{tikzcd}
\end{equation}
obtained as the companion transpose of the structure 2-morphism $m^0$ from \cref{eq:m2} (using that $T\comp{T^0}$ is a companion of $T T^0$ by \cref{lem:Ff}); and the companion transpose of the structure 2-morphism
$e^0$ from \cref{eq:e2} respectively.

With some effort, one may show that with these structure cells, the horizontal double
Kleisli category~$\Kl(T)$ is an oplax monoidal double category in the sense of
\cref{def:oplaxdoublecat}. We do not provide the details here, but in \cref{sec:sample} we give some sample verifications, along with a number of technical lemmas used repeatedly in the calculations.
\end{proof}

It is very natural to ask when the oplax monoidal structure of the preceding definition is in fact a genuine (pseudo) monoidal structure, or at least a \emph{normal} oplax monoidal structure. For our purposes, we will only answer this question in the case of primary interest, where our monoidal horizontal monad is induced from a pseudomonoidal vertical monad (\cref{def:2}). To start with, putting together \cref{cor:2} and \cref{thm:oplaxmonoidalKlT} gives us:

\begin{cor}\label{cor:important}
Let $\dc{C}$  be a monoidal double category and $T$ be a pseudomonoidal vertical double monad.
If it is true that:
\begin{enumerate}[(i)]
\item the multiplication and unit of $T$ are special; and
\item all vertical 1-cells $ T^2_{X_1,X_2} \colon TX_1 \otimes TX_2 \rightarrow T(X_1 \otimes X_2)$ and
$ T^0\colon I \rightarrow TI$ have
companions,
\end{enumerate}
then the Kleisli double category $\Kl(T)$ of the induced
monoidal horizontal monad $(T, \comp m, \comp e)$ admits an oplax monoidal structure found as in~\cref{thm:oplaxmonoidalKlT}.  \qed
\end{cor}

\begin{rmk}\label{rem:etainvertible}
Notice that in the situation of the above corollary, it is only the interchange $2$-morphisms $\tau$ of the oplax monoidal structure which may not be invertible. Indeed, each of the other structure 2-morphisms $\eta, \delta, \iota$, as displayed
in \eqref{eq:structure2} and \eqref{eq:structure34}, must be invertible; for example, in this case $\eta$ is the transpose of the $2$-morphism
\cref{eq:here2} which is, in turn, the transpose of the $\mathsf{V}(\dc{C})$-invertible $2$-morphism $e^2$ \cref{eq:here-1} of the pseudomonoidal 
vertical transformation $e$, and as such, is invertible by~\cref{prop:companionomnibus}\ref{omni-vi}.
\end{rmk}

As explained above, we will now investigate when, in the situation of \cref{cor:important}, the oplax monoidal structure on $\Kl(T)$ is in fact normal in the sense of \cref{def:normality}. To this end, motivated by the theory of pseudo-commutative monads~\cite{HylandPower}, we define for a pseudomonoidal vertical monad
$(T,m,e)$ as in \cref{def:2} a vertical double transformation $\kappa\colon (\thg)\otimes T(\mathord{?}) \Rightarrow
T\circ(\thg\ot \mathord{?})$ called the \emph{strength}; this is given as the vertical composite $T^2 \circ (e\ot1)$ with components
\begin{equation}\label{eq:strengthdef}
\begin{tikzcd}[column sep=.7in]
X_1 \ot TX_2\ar[r,tick,"M_1\ot TM_2"]\ar[d,"\kappa_{X_1,X_2}"']\ar[dr,phantom,"\Two\kappa_{M_1,M_2}"] & Y_1 \ot TY_2\ar[d,"\kappa_{Y_1,Y_2}"] \\
T(X_1 \ot X_2)\ar[r,tick,"T(M_1\ot M_2)"'] & T(Y_1 \ot Y_2)
\end{tikzcd}\;\defeq\;
\begin{tikzcd}[column sep=.7in]
X_1 \ot TX_2 \ar[r,tick,"M_1 \ot TM_2"] \ar[d,"e_{X_1} \ot\vid " ']\ar[dr,phantom,"\Two e_{M_1} \ot\vid_{TM_2}"] & Y_1\ot TY_2\ar[d,"e_{Y_1} \ot\vid"]
\\
TX_1\ot TX_2\ar[r,tick,"TM_1\ot TM_2"]\ar[d,"T^2_{X_1,X_2}"']\ar[dr,phantom,"\Two T^2_{M_1,M_2}"] & TY_1\ot TY_2\ar[d,"T^2_{Y_1,Y_2}"] \\
T(X_1\ot X_2)\ar[r,tick,"T(M_1\ot M_2)"'] & T(Y_1\ot Y_2) \rlap{ .}
\end{tikzcd}
\end{equation}
In an analogous way, we can also define the \emph{costrength} as a vertical transformation $T(\thg)\otimes (\mathord{?})\Rightarrow T\circ(\thg\ot\mathord{?})$. It turns out that requiring these two
vertical transformations to be special, as in \cref{def:special-vertical-transformation}, is sufficient to make the oplax monoidal structure on $\Kl(T)$ normal:

\begin{prop}\label{prop:normaloplaxdouble}
Let $\dc{C}$  be a monoidal double category and $T$ be a pseudomonoidal vertical double monad.
If it is true that:
\begin{enumerate}[(i)]
\item the multiplication and unit of $T$ are special;
\item all vertical 1-cells $ T^2_{X_1,X_2} \colon TX_1 \otimes TX_2 \rightarrow T(X_1 \otimes X_2)$ and $ T^0\colon I \rightarrow TI$ have
companions; and
\item the strength and costrength of $T$ are special vertical transformations, \label{item3}
\end{enumerate}
then the oplax monoidal double structure on $\Kl(T)$ found as in \cref{cor:important} is normal.
\end{prop}

\begin{proof}
We need to verify that each oplax double functor $X_1 \kot (\thg) \co \Kl(T) \to \Kl(T)$ and $(\thg) \kot X_2 \co \Kl(T) \to \Kl(T)$ is in fact a (pseudo)
double functor. By \cref{rem:etainvertible}, we already know
that the square $\eta$ in \cref{eq:structure2} is invertible, which expresses the invertibility of the identity constraints for these
double functors. As for the binary functoriality constraints, it suffices by symmetry to consider the case of $X_1\kot(\thg)$. To say that its binary constraints are invertible is to say that the $2$-morphism in~\cref{eq:interchange} is
invertible
when
$X   = X_1 = Y_1 = Z_1$ and $M_1=N_1=\comp{e}_{X_1}$. The $2$-morphism in question is given by:
\begin{equation}\label{eq:tobeinvertible}
\scalebox{.8}{
 \begin{tikzcd}[column sep=.4in,ampersand replacement=\&]
X_1\ot X_2\ar[drrr,phantom,"\Two\cong"]\ar[rrr,tick,"{(\comp{m} \circ T\comp{e} \circ
\comp{e}) \ot (\comp{m} \circ TN_2 \circ M_2)}"]\ar[d,equal] \&\&\& TX_1\ot
TZ_2\ar[d,equal]\ar[rr,tick,"\comp{ T^2}"] \&\& T(X_1\ot Z_2)\ar[d,equal]
\\
 X_1\ot X_2\ar[r,tick,"\comp{e}\ot M_2"]\ar[d,equal] \& TX_1\ot
TY_2\ar[d,equal]\ar[r,tick,"T\comp{e}\ot TN_2"] \& TTX_1\ot
TTZ_2\ar[drrr,phantom,"\Two\comp{m}^2"]\ar[d,equal]\ar[r,tick,"\comp{m}\ot \comp{m}"]
\& TX_1\ot TZ_2\ar[rr,tick,"\comp{ T^2}"] \&\& T(X_1\ot Z_2)\ar[d,equal] \\
 X_1\ot X_2\ar[r,tick,"\comp{e}\ot M_2"]\ar[d,equal] \& TX_1\ot TY_2\ar[d,equal]\ar[r,tick,"T\comp{e}\ot
TN_2"']\ar[drr,phantom,"\Two\comp{T^2_{\mathstrut}}\!\!\!\!{\mathstrut}_{\comp{e},N_2}\quad"] \& TTX_1\ot
TTZ_2\ar[r,tick,"\comp{T^2}"'] \& T(TX_1\ot
TZ_2)\ar[d,equal]\ar[r,tick,"T\comp{T^2}"'] \& TT(X_1\ot
Z_2)\ar[r,tick,"\comp{m}"'] \& T(X_1\ot Z_2)\ar[d,equal] \\
 X_1\ot X_2\ar[r,tick,"\comp{e}\ot M_2"]\ar[d,equal] \& TX_1\ot TY_2\ar[r,tick,"\comp{ T^2}"']\ar[d,equal] \& T(X_1\ot
Y_2)\ar[d,equal]\ar[drr,phantom,"\Two\cong"]\ar[r,tick,
"T(\comp{e}\ot N_2)"'] \& T(TX_1\ot TZ_2)\ar[r,tick,"T\comp{ T^2}"'] \& TT(X_1\ot
Z_2)\ar[d,equal]\ar[r,tick,"\comp{m}"'] \& T(X_1\ot Z_2)\ar[d,equal]\\
X_1\ot X_2\ar[r,tick,"\comp{e}\ot M_2"'] \& TX_1\ot TY_2\ar[r,tick,"\comp{ T^2}"'] \& T(X_1\ot Y_2)\ar[rr,tick,"T(\comp{ T^2}\circ(\comp{e}\ot
N_2))"'] \&\& TT(X_1\ot Z_2)\ar[r,tick,"\comp{m}"'] \& T(X_1\ot Z_2) \mathrlap{.}
 \end{tikzcd}}
 \end{equation}
Clearly the first and final rows of this diagram are invertible. On the second row, the $2$-morphism~$\comp{m^2}$ corresponds under transpose to the~$\mathsf{V}(\dc{C})$-invertible cell~$m^2$ of the pseudomonoidal vertical
transformation~$m$ in~\cref{eq:here-2}, and as such is invertible by~\cref{prop:companionomnibus}\ref{omni-vi}. Thus, we will be done if we can also prove the invertibility of the third row of the diagram.

Now, since the strength \cref{eq:strengthdef} is by assumption a special vertical transformation, it is in particular true that the companion transpose of the component
$\kappa_{\comp{e},N_2}$ is
invertible. This transpose is equally well the composite:
\begin{displaymath}
 \begin{tikzcd}[column sep=.4in]
X_1\ot TY_2\ar[r,tick,"\comp{e}\ot TN_2"]\ar[d,equal]\ar[drr,phantom,"\Two(*)"] & TX_1\ot TTZ_2\ar[r,tick,"\comp{e} T \ot1"] & TTX_1\ot
TTZ_2\ar[r,tick,"\comp{T^2}"]\ar[d,equal] & T(TX_1\ot TZ_2)\ar[d,equal] \\
X_1\ot TY_2\ar[r,tick,"\comp{e}\ot1"]\ar[d,equal] & TX_1\ot TY_2\ar[r,tick,"T\comp{e}\ot TN_2"]\ar[d,equal]\ar[drr,phantom,"\Two\comp{T^2}"] &
TTX_1\ot TTZ_2\ar[r,tick,"\comp{T^2}"] & T(TX_1\ot TZ_2)\ar[d,equal] \\
X\ot TY_2\ar[r,tick,"\comp{e}\ot1"'] & TX\ot TY_2\ar[r,tick,"\comp{T^2}"'] & T(X_1\ot Y_2)\ar[r,tick,"T(\comp{e}\ot N_2)"'] & T(TX_2\ot TZ_2)
 \end{tikzcd}
\end{displaymath}
where the $2$-morphism $(*)$ on the top row is the transpose of the 2-cell $e_{\comp{e}_{X_1}}\ot\vid_{N_2}$. But by \cref{lem:companioncomponents}, this $(*)$ is itself the transpose of a vertical identity, and as such, is invertible by \cref{prop:companionomnibus}\ref{omni-vi}. It follows that the composite $2$-morphism comprising the bottom row of this diagram is invertible, which now implies the invertibility of the third row of~\eqref{eq:tobeinvertible} as desired.
\end{proof}

Although this will not be the case in our applications, we note in particular the following sufficient conditions for the induced oplax monoidal structure on $\Kl(T)$ to be not just normal oplax, but in fact a genuine (pseudo) monoidal structure.

\begin{cor}\label{cor:pseudomonoidalKl}
Let $\dc{C}$  be a monoidal double category and $T$ be a pseudomonoidal vertical double monad.
If it is true that:
\begin{enumerate}[(i)]
\item the multiplication and unit of $T$ are special;
\item all vertical 1-cells $ T^2_{X_1,X_2} \colon TX_1 \otimes TX_2 \rightarrow T(X_1 \otimes X_2)$ and $ T^0\colon I \rightarrow TI$ have
companions; and
\item the monoidality constraint $T^2$ of $T$ is a  special vertical transformation, 
\end{enumerate}
then the oplax monoidal double structure on $\Kl(T)$ found as in \cref{cor:important} is genuinely monoidal.
\end{cor}
\begin{proof}
  $\eta$, $\delta$ and $\iota$ are already known to be invertible, and, arguing as before, the assumption that $T^2$ is special ensures that \emph{every} component \eqref{eq:interchange} of the oplax monoidal interchange $\tau$ is invertible.
\end{proof}

In the situation of \cref{prop:normaloplaxdouble}, the fact that
$\Kl(T)$ is a normal oplax monoidal double category implies that its
horizontal bicategory inherits the monoidal structure, in the sense
specified in \cref{def:oplaxbicat}. We thus obtain the following result as the culmination of the abstract development of the paper thus far. 
This will be the result we use to obtain the monoidal structure on the bicategory of coloured symmetric sequences
in~\cref{sec:application}.

\begin{cor} \label{thm:bicat-oplax-monoidal}
Let $\dc{C}$  be a monoidal double category and $T$ be a pseudomonoidal vertical double monad. Under the
assumptions of \cref{prop:normaloplaxdouble}, the horizontal bicategory of $\Kl(T)$ admits a normal oplax monoidal structure.
\end{cor}

\begin{proof}
 This follows from \cref{thm:oplaxmonoidalbicat,prop:normaloplaxdouble,prop:Klcompanions}.
\end{proof}

\section{The arithmetic product of coloured symmetric sequences}
\label{sec:application}

In this section, we apply the theory developed in the previous
sections to our intended application, namely coloured symmetric
sequences. Throughout this section, we fix a cocomplete {\em cartesian
closed} category $\ca{V}$, considered as a symmetric monoidal closed category
with respect to its cartesian closed structure.
The restriction to a cartesian
monoidal structure was already made in \cite{DwyerW:BoardmanVtpo,GarnerLopezFranco}
and indeed it is essential for some of our results, as we explain further below.

In order to help readers follow our development, let us display the main double categories to be considered
in this section in a commutative diagram of inclusions:
\begin{displaymath}
\begin{gathered}
\xymatrix{
\VMat \ar[r] \ar[d] & \VSym \ar[d] \\
\VProf \ar[r] & \VCatSym \mathrlap{.} }
\end{gathered}
\end{displaymath}
On the left-hand side of the diagram, $\VMat$ is the double category
of matrices of \cref{ex:mat} and~$\VProf$ is the double category of profunctors of \cref{ex:prof}. On the right-hand side of the
diagram,~$\VCatSym$ is the double category of categorical symmetric sequences which arises from~$\VProf$ as a Kleisli double category,
and~$\VSym$ is its full double subcategory spanned by discrete $\ca{V}$-categories---much like the double category~$\VMat$ is a full
double subcategory of~$\VProf$. We will define the double categories~$\VCatSym$ and~$\VSym$ explicitly in \cref{thm:catsym-and-sym}, but in order to do so, we must first introduce
the relevant double monad for the Kleisli construction.

Let $X$ be a small $\ca{V}$-category. For
$n \in \mathbb{N}$, let us define the $\ca{V}$-category~$S_n(X)$ as follows. The objects of~$S_n(X)$
are $n$-tuples $\vec x = (x_1, \dots, x_{n})$ of objects of $X$. Given
two such $n$-tuples $\vec x = (x_1, \ldots, x_n)$ and $\vec{x}' = (x'_1, \ldots, x'_n)$,
the hom-object of maps between them is defined by
\[
S_n(X) [ \vec x, \vec x' ] \defeq
\bigsqcup_{\sigma \in \sym_n} \bigsqcap_{1 \leq i \leq n} X[x_{\sigma(i)}, x'_i]
\]
where $\sym_n$ is the $n$-th symmetric group, and where $\bigsqcup$ and $\bigsqcap$ denote coproduct and product respectively. We then let $\freesmc 
X$ be the following coproduct in $\VCat$:
\[
\freesmc X = \bigsqcup_{n \in \mathbb{N}} \freesmc_n(X) \mathrlap{.}
\]

The $\ca{V}$-category $\freesmc X$ admits a symmetric strict monoidal structure in which the tensor product, written as $\vec{x}, \vec{y} \mapsto 
\vec{x} \otimes \vec{y}$,
is given by concatenation of sequences; the tensor unit is given by the empty
sequence, written~$(\;)$; and the symmetry is given by the evident permutations. The operation
mapping $X$ to $\freesmc X$ extends to a 2-functor $\freesmc \co \VCat \to \VCat$, which is part of a 2-monad
whose strict algebras are symmetric strict monoidal $\ca{V}$-categories. The multiplication of this 2-monad
has components $m_X \co \freesmc \freesmc X \rightarrow \freesmc X$, for~$X \in \VCat$, defined by taking a list of
lists to its flattening:
\begin{equation*}
  m_X(\vec x^1, \dots, \vec x^k) \defeq  \vec x^1 \otimes \ldots \otimes \vec x^k  \mathrlap{ .}
\end{equation*}
The unit of the 2-monad has components $e_X \colon X \rightarrow \freesmc X$, for $X \in \VCat$,
defined by taking an object $x \in X$ to the singleton list $(x) \in \freesmc X$.

We now show that, firstly, $\freesmc $ can be made into a vertical double monad on $\VProf$, and secondly, that this vertical double monad can be turned 
into a horizontal double monad. To say that $\freesmc $ can be made into  a vertical double monad is equivalently to say that the underlying $2$-functor of 
$\freesmc $ extends along the inclusion of bicategories $\VCat \rightarrow \mathsf{Prof}_{\ca{V}}$---see Remark~\ref{rk:extensions} below---while to say that this vertical monad can be 
turned into a horizontal one amounts to saying that the whole $2$-monad $\freesmc $ extends from $\VCat$ to~$\mathsf{Prof}_{\ca{V}}$. This is a known 
result, and there are two approaches in the literature to proving it. The first uses the theory of pseudo-distributive laws; see, for 
example~\cite{FioreM:relpkb}. The second, which we follow here, is essentially a categorification of the approach of~\cite{Barr1970Relational}.

\begin{prop} \label{thm:double-monad-on-prof}
The free symmetric strict monoidal category $2$-monad $\freesmc$ on $\VCat$ extends in an essentially unique way to a vertical double monad on
$\VProf$. The multiplication and unit vertical transformations of this double monad are special.
\end{prop}

Here, we say that a vertical double monad $T$ on a double category $\dc{C}$ \emph{extends} a $2$-monad $R$ on the vertical $2$-category $\nc{V}(\dc{C})$, if $R$ is isomorphic (as a $2$-monad) to the 2-monad $\nc{V}(T)$ on $\nc{V}(\dc{C})$ induced by $T$. By saying that an extension of $R$ is \emph{essentially unique}, we mean to assert the contractibility of the category in which objects are vertical double monads $T$ on $\dc{C}$ endowed with an isomorphism of $2$-monads $\nc{V}(T) \cong R$, and morphisms are vertical double monad morphisms compatible with the isomorphisms to $R$.

\begin{proof}
  Because any double functor preserves companions (\cref{lem:Ff}), any extension of $\freesmc$ to $\VProf$ must satisfy $\freesmc(\comp{F}) \cong
\comp{\freesmc F}$ for a $\ca{V}$-functor $F$. Because $\comp{F} \dashv \coj{F}$ in $\ca{H}(\VProf)$, and any double functor preserves adjunctions in
the horizontal bicategory, we must also have $\freesmc(\coj{F}) \cong \coj{\freesmc F}$. Since by~\cite[\S 6]{FibrationsinBicats}, every
$\ca{V}$-profunctor $M \colon X \horightarrow Y$ admits a globular isomorphism to one of the form $\coj{G} \circ \comp{F}$ for a suitable cospan of
$\ca{V}$-functors $F \colon X \rightarrow Z \leftarrow Y \colon G$, the preceding conditions determine the action of $\freesmc$ on horizontal
$1$-cells of $\VProf$ to within isomorphism. Using this idea, one obtains the following explicit definition: given $M \colon X \horightarrow Y$,  $\freesmc M \colon \freesmc X \horightarrow \freesmc Y$ is defined as the
$\mathbb{N}$-indexed coproduct in $(\VProf)_1$ of $S_n(M) \colon S_n(X) \horightarrow S_n(Y)$, where \begin{equation}\label{eq:SM}
  S_n(M)(\vec y, \vec x) \defeq
\bigsqcup_{\sigma \in \sym_n} \bigsqcap_{1 \leq i \leq n} M(y_{\sigma(i)}, x_i) \mathrlap{.}
\end{equation}
The action of $\freesmc$ on $2$-morphisms of $\VProf$ is now forced by~\cref{prop:companionomnibus}\ref{omni-iii} and the fact that any double functor preserves
globularity; the reader will easily guess an explicit formula, and this guess is the correct one. This completes the extension of $\freesmc$ to a double functor on $\VProf$; that these data are indeed double functorial is verified, for example, in~\cite[Proposition~55]{Poly}, to which we refer for further details. Note also that the manner in which we defined the action of $\freesmc$ on horizontal $1$-cells and $2$-morphisms means that this extension is essentially unique, in the sense set out above.

We now extend the unit $e$ and multiplication $m$ of the $2$-monad $\freesmc$; the missing data are the $2$-morphism components $e_M$ and $m_M$
(\ref{eq:mM}, \ref{eq:eM}) associated to a horizontal $1$-cell $M \colon X \horightarrow Y$. Writing $M \cong \coj{G} \circ \comp{F}$ as before, we
see
that
$e_M$ and $m_M$ are determined by $e_{\comp{F}}$, $e_{\coj{G}}$, $m_{\comp{F}}$ and $m_{\coj{G}}$:
\begin{equation}\label{eq:mande}
  e_M =
  \begin{tikzcd}
    X \ar[d,equal] \ar[rr,tick,"M"]
    \ar[drr,phantom,"\cong"] & & Z \ar[d,equal] \\
    X \ar[r,tick,"\comp{F}"] \ar[d,"e_X"'] \ar[dr,phantom,"\Two e_{\comp{F}}"]&
    Z \ar[r,tick,"\coj{G}"] \ar[d,"e_Z" description] \ar[dr,phantom,"\Two e_{\coj{G}}"]&
    Y \ar[d,"e_Y"] \\
    \freesmc X \ar[r,tick,"\freesmc\comp{F}"'] \ar[d,equal] \ar[drr,phantom,"\cong"] &
    \freesmc Z \ar[r,tick,"\freesmc\coj{G}"'] &
    \freesmc Y \ar[d,equal] \\
    \freesmc X \ar[rr,tick,"\freesmc M"'] & &
    \freesmc Y
  \end{tikzcd} \qquad
  \text{and} \qquad
 m_M =
  \begin{tikzcd}
    \freesmc\freesmc X \ar[d,equal] \ar[rr,tick,"\freesmc\freesmc M"]
    \ar[drr,phantom,"\cong"] & & \freesmc\freesmc Z \ar[d,equal] \\
    \freesmc\freesmc X \ar[r,tick,"\freesmc\freesmc\comp{F}"] \ar[d,"m_X"'] \ar[dr,phantom,"\Two m_{\comp{F}}"]&
    \freesmc\freesmc Z \ar[r,tick,"\freesmc\freesmc\coj{G}"] \ar[d,"m_Z" description] \ar[dr,phantom,"\Two m_{\coj{G}}"]&
    \freesmc\freesmc Y \ar[d,"m_Y"] \\
    \freesmc X \ar[r,tick,"\freesmc\comp{F}"'] \ar[d,equal] \ar[drr,phantom,"\cong"] &
    \freesmc Z \ar[r,tick,"\freesmc\coj{G}"'] &
    \freesmc Y \ar[d,equal] \\
    \freesmc X \ar[rr,tick,"\freesmc M"'] & &
    \freesmc Y\mathrlap{ .}
  \end{tikzcd}
\end{equation}
To the left, \cref{lem:companioncomponents} implies that $e_{\comp{F}}$ is the companion transpose of the identity of vertical $1$-cells
$e_Z \circ F = \freesmc F \circ e_X$; while a suitable dual of \cref{lem:companioncomponents} implies that $e_{\coj{F}}$ is the
\emph{conjoint} transpose of the identity $e_Z \circ G = \freesmc G \circ e_Y$. In a similar way, the $2$-morphisms~$m_{\comp{F}}$ and~$m_{\coj{G}}$ 
are
forced. An explicit verification that these data satisfy the vertical double monad axioms is, again, given in~\cite[Proposition~55]{Poly}. So we have extended $\freesmc$ to a vertical double monad on $\VProf$; and like before, the manner in which we made this extension forces it to be essentially unique.

It remains only to verify that the unit and multiplication of our extended $\freesmc$ are special. Before doing so, we note a fact which will be used repeatedly in what follows. Suppose given profunctors $N \colon X \tickar Y$ and $M \colon 
\freesmc Y \tickar Z$. Then for any $z \in Z$ and $\vec x = (x_1, \dots, x_n) \in \freesmc X$, the value at $(z, \vec x)$ of the composite $M \circ 
\freesmc N \colon \freesmc X \tickar \freesmc Y \tickar Z$, as to the left below, is equally given as to the 
right:\begin{equation}\label{eq:reduction}
  \int^{\vec y \in \freesmc Y} M(z, \vec y) \times \freesmc N(\vec y, \vec x) \cong \int^{\vec y \in Y^n} M(z, \vec y) \times \bigsqcap_{1 \leqslant 
i \leqslant n} N(y_i, x_i)\rlap{ .}
\end{equation}
Indeed, we can immediately reduce the left-hand coend to one over $\vec y \in S_n Y$; and for such a $\vec y$, we have $M(z, \vec y) \times \freesmc 
N(\vec y, 
\vec x) \cong \bigsqcup_{\sigma \in \mathfrak{S}_n} M(z, \vec y) \times \bigsqcap_{1 \leqslant i \leqslant n } N(y_{\sigma(i)}, x_i)$. On the $\sigma$-summand of 
this coproduct, we define the component of the desired isomorphism~\cref{eq:reduction} to be
\begin{equation*}
  M(z, \vec y) \times \bigsqcap_{1 \leqslant i \leqslant n } N(y_{\sigma(i)}, x_i) \xrightarrow{M(1, \sigma^{\mi1}) \times 1} M(z, \sigma^\ast \vec y) \times 
\bigsqcap_{1 \leqslant i \leqslant n } N(\,(\sigma^\ast \vec y)_i, x_i) \hookrightarrow \int^{\vec y \in Y^n} \! M(z, \vec y) \times \bigsqcap_{1 
\leqslant i \leqslant n} N(y_i, x_i)
\end{equation*}
where $(\sigma^\ast \vec y)_i = y_{\sigma(i)}$ and where $\sigma^{\mi1} \colon \vec y \rightarrow \sigma^\ast \vec y$ is the evident symmetry 
isomorphism in $\freesmc Y$.

We now show that the extended vertical transformations $e$ and $m$ are special, {\em i.e.}~that the following companion transpose $2$-morphisms are invertible:
\begin{equation*}
\begin{tikzcd}[column sep=.6in]
Z
	\ar[r,tick,"M"]
	\ar[d, equal]
	 &
Y
	\ar[r,tick,"\comp{e}_{Y}"]
	\ar[d,phantom,"\Two \comp{e}_{M}"]  &
\freesmc Y \ar[d, equal]	\\
Z
	\ar[r,tick,"\comp{e}_{Z}"'] &
\freesmc Z
	\ar[r,tick,"\freesmc M"'] &
\freesmc Y
\end{tikzcd} \qquad\text{and} \qquad
\begin{tikzcd}[column sep=.6in]
\freesmc \freesmc Z
	\ar[r,tick,"\freesmc \freesmc M"]
	\ar[d, equal]
	 &
\freesmc \freesmc Y
	\ar[r,tick,"\comp{m}_{Y}"]
	\ar[d,phantom,"\Two \comp{m}_{M}"]  &
\freesmc Y \ar[d, equal]	\\
\freesmc \freesmc Z
	\ar[r,tick,"\comp{m}_{Z}"'] &
\freesmc Z
	\ar[r,tick,"\freesmc M"'] &
\freesmc Y\mathrlap{ .}
\end{tikzcd}
\end{equation*}
Starting to the left, let $z \in Z$ and $\vec y = (y_1, \dots, y_n) \in \freesmc Y$. To within isomorphism, using the formulas for composition and
companions for profunctors (\ref{equ:comp-of-prof}, \ref{equ:comp-and-coj-for-prof}) as well as \cref{eq:SM}, the profunctor at the bottom of the
square sends $(\vec y, z)$ to
\begin{equation}\label{eq:3}
  \freesmc  M(\vec y, (z)) =
  \begin{cases}
    M(y_1, z)  & \text{if $n = 1$;}\\
    0 & \text{otherwise.}
  \end{cases}
\end{equation}
On the other hand, the profunctor around the top sends $(\vec y, z)$ to:
\begin{equation}\label{eq:2}
  \int^{y' \in Y} \freesmc Y[\vec y, (y')] \times M(y',z) =
  \begin{cases}
    \int^{y' \in Y} Y[y_1, y'] \times M(y',z) & \text{if $n = 1$;}\\
    0 & \text{otherwise.}
  \end{cases}
\end{equation}
In the only non-trivial case where $n = 1$, the comparison $2$-cell
$\comp e_{M}$ from~\eqref{eq:2} to~\eqref{eq:3} is given by
composition: and this is invertible by the Yoneda lemma.

We proceed similarly for $\comp m_{M}$. Let $\vec z = (\vec z^1, \dots, \vec z^n)\in \freesmc \freesmc Z$ with $\vec z^i = (z_{m_{i\mi1}+1}, \dots, 
z_{m_i})$ for some
$0 = m_0 \leqslant m_1 \leqslant \dots \leqslant m_n$,
and let $\vec y = (y_1, \dots, y_m) \in \freesmc Y$. The only non-trivial case is when $m_n = m$ so we immediately restrict to that.
This time, the profunctor in the bottom row sends $(\vec y, \vec z)$ to
\begin{equation}\label{eq:4}
  \freesmc  M(\vec y, \bigotimes_{1 \leqslant i \leqslant n} {\vec z}^{\,i}) = \bigsqcup_{\sigma \in \mathfrak{S}_m} \bigsqcap_{1 \leqslant i 
\leqslant m} M(y_{\sigma(m)}, z_m)\mathrlap{ .}
\end{equation}
On the other hand, the profunctor around the top sends $(\vec y, \vec z)$ to
\begin{equation}\label{eq:5}
  \begin{aligned}
  \int^{\vec w \in \freesmc \freesmc Y} \freesmc Y[\vec y, \bigotimes_{1 \leqslant i \leqslant m}\vec w^i] \times \freesmc \freesmc M(\vec w, \vec z)
  &\cong
  \int^{\vec w^1, \dots, \vec w^n \in \freesmc Y} \freesmc Y[\vec y, \bigotimes_{1 \leqslant i \leqslant m}\vec w^i] \times \bigsqcap_{1 \leqslant j 
\leqslant n} \freesmc M(\vec w^j, \vec z^j) \\
    &\cong
    \int^{w_1, \dots, w_m \in Y} \freesmc Y[\vec y, \vec w] \times \bigsqcap_{\substack{1 \leqslant j \leqslant n\\m_{i\mi1}+1 \leqslant i \leqslant 
m_i}}
M(w_i, z_i)\\
        &\cong
  \bigsqcup_{\sigma \in \mathfrak{S}_m} \int^{w_1, \dots, w_m \in Y}  \bigsqcap_{1 \leqslant i \leqslant m} Y[y_{\sigma(i)}, w_i] \times  M(w_i,
z_i)\mathrlap{ ,}
\end{aligned}
\end{equation}
using~\cref{eq:reduction} once at the first step, and $n$ times at the second step.
The comparison $2$-morphism $\comp m_{M}$ from~\eqref{eq:5} to~\eqref{eq:4} is again given by composition, and this is again invertible by the Yoneda
lemma.
\end{proof}

\begin{rmk}
  \label{rk:extensions}
  In fact, the above argument shows that, \emph{any} $2$-monad $R$ on $\VCat$ which extends to $\VProf$ will have an essentially-unique such extension; for indeed, the action on horizontal $1$-cells must be given as $R(M) = \widecheck{RG} \circ \widehat{RF}$, where $M = \wc G \circ \wh F$, and the components of the extended unit and multiplication at a horizontal $1$-cell $M$ must be determined similarly. The only non-trivial point is verifying that composition of horizontal $1$-cells is preserved to within globular isomorphism---and this comes to the same thing as asking that the underlying $2$-functor of $R$ extends along the homomorphism of bicategories $\widehat{(\thg)} \colon \VCat \rightarrow \mathrm{Prof}_{\ca{V}}$. Thus, to give an extension of the $2$-monad $R$ to $\VProf$ is equally to give an extension of the underlying $2$-functor along $\widehat{(\thg)} \colon \VCat \rightarrow \mathrm{Prof}_{\ca{V}}$, as claimed above.
\end{rmk}
\begin{cor} \label{thm:hor-double-monad-on-prof}
The $2$-monad $\freesmc \co \VCat \to \VCat$ induces a horizontal double monad $\freesmc \colon \VProf \rightarrow \VProf$.
\end{cor}

\begin{proof} Apply~\cref{cor:horizontalpseudomonad} to the vertical double monad of \cref{thm:double-monad-on-prof} to obtain the horizontal double
monad $(\freesmc,\comp{m},\comp{e})$.
\end{proof}

We are now ready to recall the definition of categorical and coloured symmetric sequences.

\begin{defi}[Categorical and coloured symmetric sequences] \label{def:catsymseq}  \leavevmode
\begin{itemize}
\item Let $X$ and $Y$ be small $\ca{V}$-categories. A \emph{categorical symmetric sequence}
$M \co X \rightsquigarrow Y$ is a profunctor $M \co X  \horightarrow \freesmc Y$,  {\em i.e.} a $\ca{V}$-functor $M \co \freesmc Y^\op \times X \to
\ca{V}$.
\item Let $X$ and $Y$ be sets. A \emph{coloured symmetric sequence} $M \co X \rightsquigarrow Y$ is a categorical symmetric
sequence from $X$ to $Y$, considered as discrete $\ca{V}$-categories.
\end{itemize}
\end{defi}

Categorical symmetric sequences and coloured symmetric sequences are the horizontal $1$-cells of double categories that we denote $\VCatSym$ and
$\VSym$, which we may now obtain by forming the horizontal Kleisli double category (\cref{thm:KlS}) of the horizontal double monad $\freesmc  \colon 
\VProf \rightarrow \VProf$.

\begin{thm}
 \label{thm:catsym-and-sym} \leavevmode
\begin{enumerate}[(i)]
\item \label{catsym-i} There exists a double category $\VCatSym$ having small $\ca{V}$-categories as objects, $\ca{V}$-functors as vertical 1-cells
and
categorical symmetric sequences as horizontal 1-cells.
\item \label{catsym-ii}  There exists a double subcategory $\VSym$ having sets as objects, functions as vertical 1-cells and coloured symmetric
sequences
as horizontal 1-cells.
\end{enumerate}
\end{thm}

\begin{proof} For \cref{thm:catsym-and-sym}\ref{catsym-i}, it suffices that we apply \cref{thm:KlS} to the horizontal double monad $(\freesmc,\comp{m},\comp{e}) \colon \VProf
\to \VProf$ of \cref{thm:hor-double-monad-on-prof}, where $m$ and $e$ are as in \cref{eq:mande}.
Indeed, a categorical symmetric sequence $M \co X \rightsquigarrow Y$  is
precisely a horizontal Kleisli 1-cell, and so by \cref{eq:kcirc}, 
the composition of categorical symmetric sequences $M \co X \horightarrow \freesmc Y$ and $N \co Y \horightarrow \freesmc Z$ is the profunctor $N \circ_\Kl M \co X \horightarrow \freesmc Z$ found as the 
composite 
\[
\begin{tikzcd}
X \ar[r,tick,"M"] & \freesmc Y \ar[r,tick,"\freesmc N"] & \freesmc \freesmc Z \ar[r,tick,"\comp{m}_Z"] & \freesmc Z \mathrlap{.}
\end{tikzcd}
\]
Using \cref{equ:comp-of-prof,equ:comp-and-coj-for-prof,eq:SM}, this composite has value at $\vec z = (z_1, \dots, z_m) \in \freesmc Z$ and $x \in X$ 
given by
\begin{equation*}
  (N \circ_{\Kl} M)(\vec z, x) = \int^{\vec w \in \freesmc \freesmc Z, \vec y \in \freesmc Y} \freesmc Z[\vec z, \bigotimes_i \vec w^i] \times 
\freesmc N(\vec w, \vec y) \times M(\vec y, x)
\end{equation*}
which by applying~\cref{eq:reduction} simplifies
to the following well-known formula (c.f.~\cite[eq.~(10)]{FioreM:carcbg}):
\begin{equation*}
  (N \circ_{\Kl} M)(\vec{z}, x) = \int^{\vec w^1, \dots, \vec w^m \in \freesmc Z, \vec y \in \freesmc Y} \freesmc Z[\vec z, \bigotimes_i \vec w^i] 
\times \bigsqcap_{1 \leqslant i \leqslant m} N(\vec w^i, y_i) \times M(\vec y, x)
\end{equation*}
which is a generalisation of the substitution monoidal structure for symmetric sequences~\cite{KellyGM:opejpm}. \cref{thm:catsym-and-sym}\ref{catsym-ii} follows
immediately and the formula for composition does not actually simplify significantly, since $\freesmc Y$ and $\freesmc Z$ are genuine categories even 
when $Y$ and $Z$ are sets.
\end{proof}

\begin{rmk}
Even if the primary focus of our interest is the double category of coloured symmetric sequences $\VSym$, it is useful to consider the larger double
category of categorical symmetric sequences $\VCatSym$. The reason is that the latter arises naturally from the double category of profunctors as a
Kleisli double
category and enjoys better closure properties than the former, since the free symmetric strict monoidal category on a discrete $\ca{V}$-category is
not discrete.
\end{rmk}

We now wish to apply the theory developed in the previous sections in order to obtain the desired oplax monoidal structures on~$\VCatSym$ and~$\VSym$. First note that, by \cref{ex:prof-mon}, the double category $\VProf$ has a monoidal structure induced
from that on~$\ca{V}$. 
Thus, by~\cref{prop:normaloplaxdouble}, it suffices to show that the vertical double monad $\freesmc  \colon \VProf \rightarrow \VProf$ has 
well-behaved pseudomonoidal structure.

The key fact which allows us to do this is that, as a 2-monad on $\VCat$, $\freesmc $ is pseudomonoidal~\cite{HylandPower,KellyGM:coherence-lax}.
Indeed, \cite[Section~3.3]{HylandPower} shows that $\freesmc $ can be equipped with the structure of a pseudo-commutative 2-monad,
while~\cite[Theorem~7]{HylandPower} states that every pseudo-commutative 2-monad is pseudomonoidal,
\emph{cf.}~also~\cite[Theorem~2.3]{Kock1972Strong}. For our purposes, it will be convenient to describe the relevant structure explicitly.
First of all, $\freesmc$ admits a \emph{strength}~\cite{Kock1972Strong} given by:
\begin{equation}
  \label{eq:strength}
\begin{aligned}
  \kappa \colon X \times \freesmc Y &\rightarrow \freesmc(X \times Y) \\
  (x, \vec y) & \mapsto \bigl((x, y_1), \dots, (x, y_{n})\bigr) \mathrlap{,}
\end{aligned}
\end{equation}
as well as a \emph{costrength} $\kappa' \colon \freesmc X \times Y \rightarrow \freesmc (X \times Y)$ given dually. Note that, because the formula 
for~$\kappa(x, \vec y)$ repeats the variable $x$, we can only make the assignment of~\cref{eq:strength} $\ca{V}$-functorial when~$\ca{V}$ is 
\emph{cartesian} monoidal---and this explains why we made this restriction in the first place. In this situation, the 2-functor 
$\freesmc$ acquires \emph{two} canonical lax monoidal structures built from the strength, the costrength and the monad multiplication as 
in~\cite[eqs.~(2.1) and (2.2)]{Kock1972Strong}. In our case, one of
these lax monoidal structures has $\freesmc^0 \colon 1 \rightarrow
\freesmc 1$ given by the monad unit, and $\freesmc^2_{X,Y} \co \freesmc X
\times \freesmc Y \to
\freesmc(X \times Y)$ defined by lexicographic product:
\begin{equation}
\label{equ:monoidal-freesmc}
 \bigl( (x_1, \dots, x_m) \, , \, (y_1, \dots, y_n)\bigr) \quad \mapsto \quad \bigl((x_1, y_1), (x_1, y_2), \dots, (x_1, y_n), (x_2, y_1), \dots, (x_m,
y_n)\bigr)\mathrlap{ ,}
\end{equation}
which we sometimes also denote by $\vec x \kot\vec y$ as in
\cref{equ:monoidal-for-freesmc}. The other lax monoidal structure has
the same $\freesmc^0$ and binary constraints 
$\freesmc X \times \freesmc Y \rightarrow \freesmc (X \times Y)$ given by \emph{colexicographic} product. Evidently, these two lax monoidal structures 
are isomorphic, and this is the key aspect of $\freesmc$ being pseudo-commutative in the sense of~\cite{HylandPower}.

In this situation, by~\cite[Theorem~7]{HylandPower}, which is a higher-dimensional adaptation of~\cite[Theorem~2.3]{Kock1972Strong}, $\freesmc$ 
becomes a pseudomonoidal $2$-monad with respect to the lax monoidal structure $\freesmc^2$. It is not hard to see that the monad unit $e \colon 1 
\Rightarrow \freesmc $ is in fact a genuine monoidal transformation; 
however, the multiplication $m \co \freesmc \freesmc  \Rightarrow \freesmc$
is not monoidal, but only a \emph{pseudomonoidal} transformation; which is to say that the two sides of
the diagram
\begin{equation}\label{eq:pseudomonoidalm}
\begin{tikzcd}
\freesmc \freesmc  X \times \freesmc \freesmc Y
	\ar[d, "m_X \times m_Y "']
	\ar[r, "\freesmc^2"]
	\ar[drr,phantom,"\Two m^2_{X, Y}"]
 &
\freesmc(\freesmc X \times \freesmc Y)
	\ar[r, "\freesmc(\freesmc^2)"]
	&
\freesmc \freesmc (X \times Y)
	\ar[d, "m_{X \times Y}"] \\
    \freesmc X \times \freesmc Y
    \ar[rr, "\freesmc^2"'] &&
    \freesmc(X \times Y)\mathrlap{ .}
    \end{tikzcd}
\end{equation}
are not equal, but only coherently isomorphic via a $2$-cell as displayed. We now describe this $2$-cell $m^2_{X,Y}$ explicitly. To this end, let us 
take a typical element of $\freesmc \freesmc X \times \freesmc \freesmc Y$, say
$\bigl( ( \vec x^1, \ldots, \vec x^k ) , (\vec y^1, \ldots \vec y^\ell) \bigr)$
 where $ \vec x^i = (x^i_1, \ldots, x^i_{m_i})$ for $1 \leq i \leq k$
 and $ \vec y^j = (y^j_1, \ldots, y^j_{n_j})$ for $1 \leq j \leq \ell$.
On the one hand, around the lower side of~\cref{eq:pseudomonoidalm}, this element is
sent first to $\big( (x^1_1, \dots, x^k_{m_k}), (y^1_1, \dots, y^\ell_{n_\ell}) \big)$ and then to
\begin{equation*}
\big(   (x^1_1, y^1_1), (x^1_1, y^1_2), \dots, (x^1_1, y^{\ell}_{n_\ell}), (x^1_2, y^1_1), \dots, (x^1_2, y^\ell_{n_\ell}), \dots, (
x^{k}_{m_k}, y^{\ell}_{n_\ell}) \big)\mathrlap{ .}
 \end{equation*}
 This is the lexicographic order on four indices.
 On the other hand, around the upper side of~\cref{eq:pseudomonoidalm} we obtain first $\bigl(( \vec x^1, \vec y^1), ( \vec x^1, \vec y^2), \dots, ( \vec x^k, \vec y^\ell)\bigr)$
 and then, applying \eqref{equ:monoidal-freesmc} to each pair, we get
 \begin{equation*}
   \Bigl(\!\bigl((x^1_1, y^1_1), (x^1_1, y^1_2), \dots, (x^1_{m_1}, y^1_{n_1})\bigr),
   \bigl((x^1_1, y^2_1), (x^1_1, y^2_2), \dots, (x^1_{m_1}, y^2_{n_2})\bigr),
   \dots,
   \bigl((x^k_1, y^\ell_1), (x^k_1, y^\ell_2), \dots, (x^k_{m_k}, y^\ell_{n_\ell})\bigr)\!\Bigr)
 \end{equation*}
 and so finally
 \begin{equation*}
   \bigl((x^1_1, y^1_1), (x^1_1, y^1_2), \dots, (x^1_{m_1}, y^1_{n_1}),
   (x^1_1, y^2_1), \dots, (x^1_{m_1}, y^2_{n_2}),
   \dots, (x^k_{m_k}, y^\ell_{n_\ell})\bigr)\rlap{ .}
 \end{equation*}
 This is a twisted lexicographic order on the indices, with the
 significance order being $1$--$3$--$2$--$4$ rather than
 $1$--$2$--$3$--$4$.
 Clearly, there is a unique bijection $\theta$ which exchanges these
 orderings, giving the components of the desired natural isomorphism $m^2_{X, Y}$ filling~\cref{eq:pseudomonoidalm}. This establishes the binary pseudomonoidality of $m$; the corresponding nullary pseudomonoidality constraint $m^0$ is in fact the \emph{identity}. That these data satisfy the necessary coherences to form a pseudomonoidal monad is now asserted in~\cite[Theorem~7]{HylandPower}, but one could also establish this
directly, following a reasoning similar to that used in~\cite[Section~3.3]{HylandPower} to establish pseudocommutativity.

The next lemma extends the pseudomonoidal structure of the
2-monad $\freesmc  \co \VCat \to \VCat$ to a pseudomonoidal structure in the sense of~\cref{def:2} on the 
vertical double monad $\freesmc  \co \VProf \to \VProf$ of \cref{thm:double-monad-on-prof}.
Before doing this, let us note that under our assumption that $\ca{V}$ is cartesian monoidal, the induced monoidal structure on $\VProf$ is also cartesian, in the sense that the ordinary monoidal structures on $(\VProf)_0$ and $(\VProf)_1$ that underlie it are both cartesian monoidal; as such, we will continue to write $\times$ rather than $\otimes$ for this tensor product, in particular, for its action on horizontal $1$-cells of $\VProf$.

\begin{lem} \label{thm:pseudomonoidal-monad-on-prof}
The  vertical double monad $\freesmc \colon \VProf \to \VProf$ is
pseudomonoidal.
\end{lem}

\begin{proof} We need to check that $\freesmc$ is a lax monoidal double functor as in \cref{def:monoidaldoublefunctor}
and that $m$ and $e$ are pseudomonoidal vertical transformations as in \cref{def:montransf}.
Let us begin by showing that $\freesmc $ admits a lax monoidal structure. For $X_1, X_2 \in \VProf$,
the vertical 1-cells $\freesmc^2_{X_1,X_2} \co \freesmc X_1 \times \freesmc X_2 \to \freesmc (X_1 \times X_2)$ and $\freesmc^0 \co 1 \to \freesmc 1$
are given in~\eqref{equ:monoidal-freesmc}. For $M_1 \co X_1 \horightarrow Y_1$ and $M_2 \co X_2 \horightarrow 
Y_2$, the squares
\begin{equation}
\label{equ:square-for-psdmon}
\begin{tikzcd}[column sep=.6in]
\freesmc X_1\times \freesmc X_2 \ar[r,tick,"\freesmc M_1\times
\freesmc M_2"]\ar[d,"{ \freesmc^2_{X_1,X_2}}"']\ar[dr,phantom,"\Two \freesmc^2_{M_1,M_2}"] &
\freesmc Y_1 \times \freesmc Y_2\ar[d," \freesmc^2_{Y_1,Y_2}"] \\
\freesmc (X_1\times X_2)\ar[r,tick,"\freesmc (M_1\times M_2)"'] & \freesmc (Y_1\times Y_2)
\end{tikzcd}
\end{equation}
can be constructed following the same reasoning as in the proof of \cref{thm:double-monad-on-prof}, \emph{i.e.}~reducing to the cases where
$M = \comp{F}$,  $N = \comp{G}$ and $M = \coj{F}$, $N = \coj{G}$, and using that~$\freesmc $ is lax monoidal
on~$\VCat$.

We already observed that the unit $e$ is genuinely monoidal at the $2$-monad level, and the same is true for $e$ \emph{qua} vertical transformation. As for the vertical transformation $m$, the axioms for a pseudomonoidal vertical transformation that concern only the vertical fragment
are exactly those expressing that $m$ is a pseudomonoidal natural transformation in $\VCat$, which we have discussed above. The only axiom
not of this form is \cref{eq:thisaxiom}, and this can be verified using the construction of the squares in~\eqref{equ:square-for-psdmon}
via a
reduction to companions and conjoints and the modification axiom for the 2-cells filling~\eqref{eq:pseudomonoidalm}.
\end{proof}

So $\freesmc $ extends to a pseudomonoidal vertical double monad $\freesmc  \colon \VProf \rightarrow \VProf$; and the last step required to establish 
the normal oplax monoidal structure of coloured symmetric sequences is to
verify that this vertical double monad satisfies the additional conditions of \cref{prop:normaloplaxdouble}.

\begin{lem} \label{thm:S-satisifes-extra}
The pseudomonoidal vertical double monad $\freesmc  \colon \VProf \rightarrow \VProf$ has the properties that:
\begin{enumerate}[(i)]
\item \label{part-i} The multiplication $m \co \freesmc \freesmc  \Rightarrow \freesmc $ and unit $e \co 1 \Rightarrow \freesmc $  are special 
vertical 
transformations.
\item \label{part-ii} The vertical 1-cells $\freesmc^2_{X, Y} \co \freesmc X \times \freesmc Y \to \freesmc (X \times Y)$
and $\freesmc^0 \co 1 \to \freesmc 1$ have companions.
\item \label{part-iii} The strength and costrength of $\freesmc $ are special vertical transformations.
\end{enumerate}
\end{lem}

\begin{proof}
\cref{part-i} was already shown as part of \cref{thm:double-monad-on-prof}.
\cref{part-ii} is immediate since every vertical 1-cell in $\VProf$ has a companion, see \cref{ex:Proffibrant}.
For \cref{part-iii}, the two cases are dual, so we only provide details for one. According 
to~\cref{prop:normaloplaxdouble}, the 
strength in question is defined from the lax monoidal structure $\freesmc^2$ of $\freesmc $ via~\cref{eq:strengthdef}; but because we originally 
obtained $\freesmc^2$ using the strength $\kappa$ of~\eqref{eq:strength} and the dual costrength, it follows as in~\cite[Theorem~2.3]{Kock1972Strong} 
that the strength of~\cref{eq:strengthdef} is this same $\kappa$, and similarly for the costrength. Thus, the condition we must prove {\em e.g.} for 
the costrength is that, 
for any $M \co X \horightarrow Y$ and $N \colon W \horightarrow Z$, the companion transpose $2$-morphism
\begin{equation}
  \label{eq:strengthcomp}
\begin{tikzcd}[column sep=.6in]
\freesmc X \times W
	\ar[r,tick,"\freesmc M\times N"]
	\ar[d, equal]
	 &
\freesmc Y \times Z
	\ar[r,tick,"\comp{\kappa}_{Y, Z}"]
	\ar[d,phantom,"\Two {\comp \kappa}_{M, N}"]  &
\freesmc (Y \times Z) \ar[d, equal]	\\
\freesmc X \times W
	\ar[r,tick,"\comp{\kappa}_{X, Z}"'] &
\freesmc (X \times W)
	\ar[r,tick,"\freesmc (M \times N)"'] &
\freesmc (Y \times Z)
\end{tikzcd}
\end{equation}
is invertible. 
To this end,
let $\vec{u} = \big( (y_1, z_1), \ldots, (y_m, z_m) \big)$ in $\freesmc (Y \times Z)$ and let $(\vec x, w) \in \freesmc X \times 
W$ where $\vec x = (x_1, \dots, x_m)$.
Note we assume that $\vec{u}$ and $\vec{x}$ have the same length; that we may do so without loss of generality will be clear from the formulae which follow.
Now, the profunctor along the bottom of~\cref{eq:strengthcomp} has value at $( \vec{u}, (\vec{x}, w))$ given by
\begin{align*}
\freesmc (M \times N) \big( \vec{u}, \kappa(\vec{x}, w) \big)
	& = \bigsqcup_{\sigma \in \sym_m} \bigsqcap_{1 \leq i \leq m} (M \times N) \big( ( {y}_{\sigma(i)}, z_{\sigma(i)}), (x_i, w) \big) \\
	& = \bigsqcup_{\sigma \in \sym_m} \bigsqcap_{1 \leq i \leq m} M(y_{\sigma(i)}, x_i) \times N(z_{\sigma(i)}, w) \mathrlap{.}
\end{align*}
On the other hand, the profunctor across the top of~\cref{eq:strengthcomp} has value at $( \vec{u}, (\vec{x}, w))$ given by
\begin{align*}
	& {}\phantom{\cong} {} \int^{\vec{v} \in \freesmc Y, z' \in Z} \freesmc (Y \times Z)[\vec{u}, \kappa(\vec{v}, z')] \times 
(\freesmc M \times N)\bigl((\vec{v}, z'), (\vec{x}, w)\bigr)\\
	& = \int^{\vec{v} \in \freesmc Y, z' \in Z} \freesmc (Y \times Z)[\vec{u}, \kappa(\vec{v}, z')] \times \freesmc M(\vec{v}, 
\vec x) \times N(z', w)\\
	& \cong \int^{v_1, \dots, v_m \in Y, z' \in Z} \freesmc (Y \times Z)[\vec{u}, \kappa(\vec{v}, z')] \times \bigsqcap_{1 
\leqslant i \leqslant m} M(v_i, x_i) \times N(z', w) \\
	& \cong \bigsqcup_{\sigma \in \sym_m} \int^{v_1, \dots, v_m \in Y, z' \in Z} \bigsqcap_{1 \leqslant i \leqslant m} (Y \times 
        Z)[(y_{\sigma(i)}, z_{\sigma(i)}), (v_i, z')] \times \bigsqcap_{1 \leqslant i \leqslant m} M(v_i, x_i) \times N(z', w)\\
	& \cong \bigsqcup_{\sigma \in \sym_m} \int^{v_1, \dots, v_m \in Y, z' \in Z} \bigsqcap_{1 \leqslant i \leqslant m} Y[y_{\sigma(i)} , v_i] \times 
        Z[z_{\sigma(i)}, z'] \times \bigsqcap_{1 \leqslant i \leqslant m} M(v_i, x_i) \times N(z', w)\\
	& \cong \bigsqcup_{\sigma \in \sym_m} \bigsqcap_{1 \leqslant i \leqslant m} M(y_{\sigma(i)}, x_i) \times N(z_{\sigma(i)}, w)\rlap{ ,}
\end{align*}
where at the first isomorphism we use~\eqref{eq:reduction} and at the
final one we use the Yoneda lemma.
By tracing it through we may see that the isomorphism constructed in this way is exactly ${\comp \kappa}_{M, N}$, which is thus invertible.
As noted above, the specialness of the costrength follows by an
identical dual argument.
\end{proof}

\begin{thm}
\label{thm:double-catsym-oplax-monoidal}
The double category $\VCatSym$ of categorical symmetric sequences admits a normal oplax monoidal structure,
given by arithmetic product of categorical symmetric sequences. Moreover, this restricts to a normal oplax monoidal
structure on the double category $\VSym$  of coloured symmetric sequences.
\end{thm}

\begin{proof} $\VCatSym$ is the horizontal Kleisli double category of the
  horizontal double monad induced by the vertical double monad $\freesmc \co \VProf \to \VProf$, as seen in the proof of~\cref{thm:catsym-and-sym}. Moreover, by~\cref{thm:pseudomonoidal-monad-on-prof}, the vertical double monad $\freesmc$ is
pseudomonoidal, and by \cref{thm:S-satisifes-extra}, it satisfies the further hypotheses of   \cref{cor:important} and \cref{prop:normaloplaxdouble}. Applying these results, we see that the monoidal structure of $\VProf$ extends to a normal oplax monoidal structure on 
$\VCatSym$, which clearly restricts back to the full sub-double-category $\VSym$.

It remains to show that the tensor product of horizontal $1$-cells computes the arithmetic product of categorical symmetric sequences. For categorical symmetric sequences $M_1 \co X_1 \rightsquigarrow Y_1$ and $M_2 \co X_2 \rightsquigarrow Y_2$, the tensor product
$M_1 \kot M_2 \co X_1 \times X_2 \rightsquigarrow Y_1 \times Y_2$ is defined as the profunctor:
\begin{equation*}
\begin{tikzcd}[column sep=.5in]
X_1 \times X_2 \ar[r,tick,"M_1 \times M_2"] & \freesmc Y_1 \times \freesmc Y_2 \ar[r,tick,"\comp{ \freesmc^2_{\mathrlap{Y_1,Y_2}}}\quad"] &
\freesmc(Y_1 \times Y_2) \mathrlap{.}
\end{tikzcd}
\end{equation*}
We now unfold this expression explicitly. First, by the definition of a companion in~\eqref{equ:comp-and-coj-for-prof} applied to
$ \freesmc^2_{Y_1,Y_2}$, the second of these profunctors is given by 
\[
 \comp{\freesmc^2_{\mathrlap{Y_1,Y_2}}}\quad\;\;( \vec y, ( \vec y_1, \vec y_2)) = \freesmc(Y_1 \times Y_2)[  \vec y, \vec y_1 \kot \vec y_2] 
\mathrlap{,}
\]
where $\vec y_1 \kot \vec y_2$ is given by the lexicographic ordering \cref{equ:monoidal-freesmc}.
Thus, using the definition of tensor product of profunctors~\eqref{equ:tensor-of-prof} and of composition of profunctors
in~\eqref{equ:comp-of-prof}, we obtain
\begin{equation}\label{eq:oplaxcatsym}
(M_1 \kot M_2)( \vec{y}, (x_1,x_2)) =
\int^{\vec y_1, \vec y_2} \freesmc(Y_1 \times Y_2)[\vec{y}, \vec y_1 \kot \vec y_2] \times M_1(\vec y_1, x_1) \times M_2(\vec y_2, x_2)\mathrlap{,}
\end{equation}
which is the formula for the arithmetic product of categorical and coloured symmetric sequences and in particular gives
 the formula in~\eqref{equ:coloured-matrix-multiplication} for coloured symmetric sequences.
\end{proof}

The final step is to obtain the desired oplax monoidal structures at the level of bicategories rather than double categories. Indeed, 
the horizontal bicategories of $\VCatSym$ and $\VSym$ are the  bicategories of categorical and
coloured symmetric sequences introduced in \cite{FioreM:relpkb} for~$\ca{V} = \Set$  and in~\cite{GambinoJoyal} for a general~$\ca{V}$. Thus, we can apply \cref{thm:bicat-oplax-monoidal} to obtain:

\begin{thm}
 \label{thm:main-app-2}
The bicategory of categorical symmetric sequences $\mathsf{CatSym}_{\ca{V}}$ admits a normal oplax monoidal structure,
given by the arithmetic product of categorical symmetric sequences. Furthermore, this normal oplax monoidal
structure restricts to the bicategory of coloured symmetric sequences $\mathsf{Sym}_{\ca{V}}$.
\end{thm}

Apart from its intrinsic interest, and the construction of an example of a sophisticated kind of low-dimensional categorical
structure by purely algebraic means, without any appeal to homotopy theory, \cref{thm:double-catsym-oplax-monoidal} and
\cref{thm:main-app-2} will be essential for
subsequent work on the Boardman--Vogt tensor product of bimodules between symmetric coloured operads, extending
that in~\cite{DwyerW:BoardmanVtpo} for bimodules between symmetric operads.

\appendix\section{}

\subsection{Coherence axioms for an oplax monoidal double category}\label{sec:appendix}

In this appendix, we spell out in detail the axioms for an oplax monoidal category as in \cref{def:mondoublecat}. In \cite{ConstrSymMonBicatsFun}, the explicit axioms for a \emph{monoidal} double
category
can be found, which are analogous but differ in that we use the opposite orientation for our (non-invertible) structure maps, and also provide additional
non-invertible structure cells $\delta$ and $\iota$ as in \cref{eq:structure2cells2}, which in the pseudo case can be chosen to be identities and so omitted.

The 2-cells $\tau$ and $\eta$ as in \cref{eq:structure2cells1} satisfy the following axioms, for $M_i\colon X_i\horightarrow Y_i$, $N_i\colon
Y_i\horightarrow
Z_i$ and $P_i\colon Z_i\horightarrow U_i$:
\begin{displaymath}
\scalebox{.8}{\begin{tikzcd}[ampersand replacement=\&]
X_1\ot X_2\ar[d,equal]\ar[drrr,phantom,"\Two\tau"]\ar[rrr,tick,"\big((P_1\circ N_1)\circ M_1\big)\ot\big((P_2\circ N_2)\circ M_2\big)"] \&\&\&
U_1\ot U_2\ar[d,equal] \\
X_1\ot X_2\ar[d,equal]\ar[r,tick,"M_1\ot M_2"']\ar[dr,phantom,"\Two\vid"] \& Y_1\ot Y_2\ar[d,equal]\ar[rr,tick,"(P_1\circ N_1)\ot(P_2\circ N_2)"]
\ar[drr,phantom,"\Two\tau"] \&\& U_1\ot Y_2\ar[d,equal] \\
X_1\ot X_2\ar[r,tick,"M_1\ot M_2"'] \& Y_1\ot Y_2\ar[r,tick,"N_1\ot N_2"'] \& Z_1\ot Z_2\ar[r,tick,"P_1\ot P_2"'] \& U_1\ot U_2
 \end{tikzcd}=
 \begin{tikzcd}[ampersand replacement=\&]
X_1\ot X_2\ar[d,equal]\ar[drrr,phantom,"\Two a\ot a"]\ar[rrr,tick,"\big((P_1\circ N_1)\circ M_1\big)\ot\big((P_2\circ N_2)\circ M_2\big)"] \&\&\&
U_1\ot U_2\ar[d,equal] \\
X_1\ot X_2\ar[d,equal]\ar[drrr,phantom,"\Two \tau"]\ar[rrr,tick,"\big(P_1\circ (N_1\circ M_1)\big)\ot\big(P_2\circ (N_2\circ M_2)\big)"'] \&\&\&
U_1\ot U_2\ar[d,equal] \\
X_1\ot X_2\ar[d,equal]\ar[rr,tick,"(N_1\circ M_1)\ot (N_2\circ M_2)"']\ar[drr,phantom,"\Two\tau"] \&\& Z_1\ot Z_2\ar[r,tick,"P_1\ot
P_2"']\ar[d,equal]\ar[dr,phantom,"\Two\vid"] \& U_1\ot U_2\ar[d,equal] \\
X_1\ot X_2\ar[r,tick,"M_1\ot M_2"'] \& Y_1\ot Y_2\ar[r,tick,"N_1\ot N_2"'] \& Z_1\ot Z_2\ar[r,tick,"P_1\ot P_2"'] \& U_1\ot U_2
\end{tikzcd}}
\end{displaymath}
\begin{displaymath}
\scalebox{.7}{\begin{tikzcd}[ampersand replacement=\&]
X_1\ot X_2\ar[rr,tick,"(\hid\circ M_1)\ot(\hid\circ M_2)"]\ar[d,equal]\ar[drr,phantom,"\Two\tau"] \&\& Y_1\ot Y_2\ar[d,equal] \\
X_1\ot X_2\ar[r,tick,"M_1\ot M_2"]\ar[dr,phantom,"\Two\vid"]\ar[d,equal] \& Y_1\ot Y_2\ar[d,equal]\ar[r,tick,"\hid\ot\hid"]\ar[dr,phantom,"\Two\eta"]
\& Y_1\ot Y_2\ar[d,equal] \\
X_1\ot X_2\ar[r,tick,"M_1\ot M_2"]\ar[d,equal]\ar[drr,phantom,"\Two\ell"] \& Y_1\ot Y_2\ar[r,tick,"\hid"] \& Y_1\ot Y_2\ar[d,equal] \\
X_1\ot X_2\ar[rr,tick,"M_1\ot M_2"'] \&\& Y_1\ot Y_2
 \end{tikzcd}=
 \begin{tikzcd}[ampersand replacement=\&]
X_1\ot X_2\ar[rr,tick,"(\hid\circ M_1)\ot(\hid\circ M_2)"]\ar[d,equal]\ar[drr,phantom,"\Two\ell\ot\ell"] \&\& Y_1\ot Y_2\ar[d,equal] \\
X_1\ot X_2\ar[rr,tick,"M_1\ot M_2"'] \&\& Y_1\ot Y_2
 \end{tikzcd}}\;
\scalebox{.7}{\begin{tikzcd}[ampersand replacement=\&]
X_1\ot X_2\ar[rr,tick,"(M_1\circ\hid)\ot(M_2\circ\hid)"]\ar[d,equal]\ar[drr,phantom,"\Two\tau"] \&\& Y_1\ot Y_2\ar[d,equal] \\
X_1\ot X_2\ar[r,tick,"\hid\ot\hid"]\ar[dr,phantom,"\Two\eta"]\ar[d,equal] \& X_1\ot X_2\ar[d,equal]\ar[r,tick,"M_1\ot M_2"]\ar[dr,phantom,"\Two\vid"]
\& Y_1\ot Y_2\ar[d,equal] \\
X_1\ot X_2\ar[r,tick,"\hid"]\ar[d,equal]\ar[drr,phantom,"\Two r"] \& X_1\ot X_2\ar[r,tick,"M_1\ot M_2"] \& Y_1\ot Y_2\ar[d,equal] \\
X_1\ot X_2\ar[rr,tick,"M_1\ot M_2"'] \&\& Y_1\ot Y_2
\end{tikzcd}=
\begin{tikzcd}[ampersand replacement=\&]
X_1\ot X_2\ar[rr,tick,"(M_1\circ\hid)\ot(M_2\circ\hid)"]\ar[d,equal]\ar[drr,phantom,"\Two r\ot r"] \&\& Y_1\ot Y_2\ar[d,equal] \\
X_1\ot X_2\ar[rr,tick,"M_1\ot M_2"'] \&\& Y_1\ot Y_2
\end{tikzcd}}
\end{displaymath}
These axioms make $\otimes$ into an oplax double functor. Notice that at the bottom of the left diagram of the first axiom there is a
composition
associativity constraint implied.

The 2-cells $\delta$ and $\iota$ as in \cref{eq:structure2cells2} satisfy the following axioms
\begin{equation}\label{eq:deltaiotaaxioms}
\scalebox{.8}{\begin{tikzcd}[ampersand replacement=\&]
I_0\ar[d,equal]\ar[drrr,phantom,"\Two\delta"]\ar[rrr,tick,"I_1"]\ar[d,equal] \&\&\& I_0\ar[d,equal] \\
I_0\ar[r,tick,"I_1"]\ar[d,equal]\ar[dr,phantom,"\Two\vid"] \& I_0\ar[rr,tick,"I_1"]\ar[d,equal]\ar[drr,phantom,"\Two\delta"] \&\& I_0\ar[d,equal] \\
I_0\ar[r,tick,"I_1"'] \& I_0\ar[r,tick,"I_1"'] \& I_0\ar[r,tick,"I_1"'] \& I_0
\end{tikzcd}=
\begin{tikzcd}[ampersand replacement=\&]
I_0\ar[d,equal]\ar[drrr,phantom,"\Two\delta"]\ar[rrr,tick,"I_1"]\ar[d,equal] \&\&\& I_0\ar[d,equal] \\
I_0\ar[rr,tick,"I_1"]\ar[d,equal]\ar[drr,phantom,"\Two\delta"] \&\& I_0\ar[dr,phantom,"\Two\vid"]\ar[r,tick,"I_1"]\ar[d,equal]\& I_0\ar[d,equal] \\
I_0\ar[r,tick,"I_1"'] \& I_0\ar[r,tick,"I_1"'] \& I_0\ar[r,tick,"I_1"'] \& I_0
\end{tikzcd}}
\end{equation}
\begin{displaymath}
\scalebox{.8}{\begin{tikzcd}[ampersand replacement=\&]
I_0\ar[rr,tick,"I_1"]\ar[d,equal]\ar[drr,phantom,"\Two\delta"] \&\& I_0\ar[d,equal] \\
I_0\ar[r,tick,"I_1"]\ar[d,equal]\ar[dr,phantom,"\Two\vid"] \& I_0\ar[r,tick,"I_1"]\ar[d,equal]\ar[dr,phantom,"\Two\iota"] \& I_0\ar[d,equal] \\
I_0\ar[r,tick,"I_1"]\ar[d,equal]\ar[drr,phantom,"\Two\ell"] \& I_0\ar[r,tick,"\hid_{I_0}"] \& I_0\ar[d,equal] \\
I_0\ar[rr,tick,"I_1"'] \&\&I_0
\end{tikzcd}=
\begin{tikzcd}[ampersand replacement=\&]
I_0\ar[rr,tick,"I_1"]\ar[drr,phantom,"\Two\vid"]\ar[d,equal] \&\& I_0\ar[d,equal] \\
I_0\ar[rr,tick,"I_1"'] \&\& I_0
\end{tikzcd}}\qquad
\scalebox{.8}{\begin{tikzcd}[ampersand replacement=\&]
I_0\ar[rr,tick,"I_1"]\ar[d,equal]\ar[drr,phantom,"\Two\delta"] \&\& I_0\ar[d,equal] \\
I_0\ar[r,tick,"I_1"]\ar[d,equal]\ar[dr,phantom,"\Two\iota"] \& I_0\ar[r,tick,"I_1"]\ar[dr,phantom,"\Two\vid"]\ar[d,equal] \& I_0\ar[d,equal] \\
I_0\ar[r,tick,"I_1"]\ar[d,equal]\ar[drr,phantom,"\Two r"] \& I_0\ar[r,tick,"\hid_{I_0}"] \& I_0\ar[d,equal] \\
I_0\ar[rr,tick,"I_1"'] \&\&I_0
\end{tikzcd}=
\begin{tikzcd}[ampersand replacement=\&]
I_0\ar[rr,tick,"I_1"]\ar[drr,phantom,"\Two\vid"]\ar[d,equal] \&\& I_0\ar[d,equal] \\
I_0\ar[rr,tick,"I_1"'] \&\& I_0
\end{tikzcd}}
\end{displaymath}
which make $I\colon\bf{1}\to\dc{C}$ into an oplax double functor.

Next, for horizontal 1-cells $M_i\colon X_i\horightarrow Y_i$ and $N_i\colon Y_i\horightarrow Z_i$, the following axioms hold
\begin{displaymath}
\scalebox{.8}{\begin{tikzcd}[column sep=.6in,ampersand replacement=\&]
(X_1\ot X_2)\ot X_3\ar[d,equal]\ar[drr,phantom,"\Two\tau\ot\vid"]\ar[rr,tick,"\big((N_1\circ M_1)\ot(N_2\circ M_2)\big)\ot(N_3\circ M_3)"] \&\&
(Z_1\ot Z_2)\ot Z_3\ar[d,equal] \\
(X_1\ot X_2)\ot X_3\ar[d,equal]\ar[rr,tick,"\big((N_1\ot N_2)\circ(M_1\ot M_2)\big)\ot(N_3\circ M_3)"']\ar[drr,phantom,"\Two\tau"] \&\&
(Z_1\ot Z_2)\ot Z_3\ar[d,equal] \\
(X_1\ot X_2)\ot X_3\ar[d,"\alpha"']\ar[dr,phantom,"\Two\alpha"]\ar[r,tick,"(M_1\ot M_2)\ot M_3"'] \&
(Y_1\ot Y_2)\ot Y_3\ar[d,"\alpha"']\ar[dr,phantom,"\Two\alpha"]\ar[r,tick,"(N_1\ot N_2)\ot N_3"'] \&
(Z_1\ot Z_2)\ot Z_3\ar[d,"\alpha"] \\
X_1\ot(X_2\ot X_3)\ar[r,tick,"M_1\ot(M_2\ot M_3)"'] \& Y_1\ot(Y_2\ot Y_3)\ar[r,tick,"N_1\ot(N_2\ot N_3)"'] \& Z_1\ot(Z_2\ot Z_3)
 \end{tikzcd}}=
\scalebox{.8}{\begin{tikzcd}[column sep=.6in,ampersand replacement=\&]
(X_1\ot X_2)\ot X_3\ar[d,"\alpha"']\ar[drr,phantom,"\Two\alpha"]\ar[rr,tick,"\big((N_1\circ M_1)\ot(N_2\circ M_2)\big)\ot(N_3\circ M_3)"] \&\&
(Z_1\ot Z_2)\ot Z_3\ar[d,"\alpha"] \\
X_1\ot(X_2\ot X_3)\ar[d,equal]\ar[rr,tick,"(N_1\circ M_1)\ot\big((N_2\circ M_2)\ot(N_3\circ M_3)\big)"]\ar[drr,phantom,"\Two\vid\ot\tau"] \&\&
Z_1\ot(Z_2\ot Z_3)\ar[d,equal] \\
X_1\ot(X_2\ot X_3)\ar[d,equal]\ar[drr,phantom,"\Two\tau"]\ar[rr,tick,"(N_1\circ M_1)\ot\big((N_2\ot N_3)\circ(M_2\ot M_3)\big)"'] \&\&
(Z_1\ot Z_2)\ot Z_3\ar[d,equal] \\
X_1\ot(X_2\ot X_3)\ar[r,tick,"M_1\ot(M_2\ot M_3)"']\& Y_1\ot(Y_2\ot Y_3)\ar[r,tick,"N_1\ot(N_2\ot N_3)"'] \& Z_1\ot(Z_2\ot Z_3)
 \end{tikzcd}}
\end{displaymath}
\begin{displaymath}
\scalebox{.8}{\begin{tikzcd}[ampersand replacement=\&,column sep=1in]
(X_1\ot X_2)\ot X_3\ar[d,equal]\ar[r,tick,"(\hid_{X_1}\ot\hid_{X_2})\ot\hid_{X_3}"]\ar[dr,phantom,"\Two\eta\ot\vid"] \& (X_1\ot X_2)\ot
X_3\ar[d,equal] \\
(X_1\ot X_2)\ot X_3\ar[d,equal]\ar[r,tick,"\hid_{X_1\ot X_2}\ot\hid_{X_3}"']\ar[dr,phantom,"\Two\eta"] \& (X_1\ot X_2)\ot X_3\ar[d,equal] \\
(X_1\ot X_2)\ot X_3\ar[d,"\alpha"']\ar[r,tick,"\hid_{(X_1\ot X_2)\ot X_3}"']\ar[dr,phantom,"\Two\hid_{\alpha}"] \& (X_1\ot X_2)\ot X_3\ar[d,"\alpha"]
\\
X_1\ot (X_2\ot X_3)\ar[r,tick,"\hid_{X_1\ot (X_2\ot X_3)}"'] \& X_1\ot (X_2\ot X_3)
\end{tikzcd}=
\begin{tikzcd}[ampersand replacement=\&,column sep=1in]
(X_1\ot X_2)\ot X_3\ar[d,"\alpha"']\ar[r,tick,"(\hid_{X_1}\ot\hid_{X_2})\ot\hid_{X_3}"]\ar[dr,phantom,"\Two\alpha"] \& (X_1\ot X_2)\ot
X_3\ar[d,"\alpha"] \\
X_1\ot(X_2\ot X_3)\ar[d,equal]\ar[r,tick,"\hid_{X_1}\ot(\hid_{X_2}\ot\hid_{X_3})"']\ar[dr,phantom,"\Two\vid\ot\eta"] \& X_1\ot(X_2\ot
X_3)\ar[d,equal] \\
X_1\ot(X_2\ot X_3)\ar[d,equal]\ar[r,tick,"\hid_{X_1}\ot\hid_{X_2\ot X_3}"']\ar[dr,phantom,"\Two\eta"] \& X_1\ot(X_2\ot X_3)\ar[d,equal] \\
X_1\ot(X_2\ot X_3)\ar[r,tick,"\hid_{X_1\ot (X_2\ot X_3)}"'] \& X_1\ot (X_2\ot X_3)
\end{tikzcd}}
\end{displaymath}
which make associativity into a vertical transformation of oplax double functors (together with the naturality of components which comes from
$(\dc{D}_1,\otimes_1,I_1)$ being a monoidal category). Moreover, the following axioms hold
\begin{equation}\label{eq:sampleax}
\scalebox{0.8}{
\begin{tikzcd}[ampersand replacement=\&]
I_0\ot X\ar[rr,tick,"I_1\ot(N\circ M)"]\ar[d,equal]\ar[drr,phantom,"\Two\delta\ot\vid"] \&\& I_0\ot Z\ar[d,equal] \\
I_0\ot X\ar[d,equal]\ar[rr,tick,"(I_1\circ I_1)\ot(N\circ M)"]\ar[drr,phantom,"\Two\tau"] \&\& I_0\ot Z\ar[d,equal] \\
I_0\ot X\ar[r,tick,"I_1\ot M"]\ar[d,"\lambda"']\ar[dr,phantom,"\Two\lambda_M"] \& I_0\ot Y\ar[d,"\lambda"']\ar[r,tick,"I_1\ot
N"]\ar[dr,phantom,"\Two\lambda_N"] \& I_0\ot Z\ar[d,"\lambda"] \\
X\ar[r,tick,"M"'] \& Y\ar[r,tick,"N"'] \& Z
\end{tikzcd}=
\begin{tikzcd}[ampersand replacement=\&]
I_0\ot X\ar[rr,tick,"I_1\ot(N\circ M)"]\ar[d,"\lambda"']\ar[drr,phantom,"\Two\lambda_{N\circ M}"] \&\& I_0\ot Z\ar[d,"\lambda"] \\
X\ar[r,tick,"M"'] \& Y\ar[r,tick,"N"'] \& Z
\end{tikzcd}\;\;
\begin{tikzcd}[ampersand replacement=\&]
 I_0\ot X\ar[r,tick,"I_1\ot\hid_X"]\ar[d,equal]\ar[dr,phantom,"\Two\iota\ot\vid"] \& I_0\ot X\ar[d,equal] \\
 I_0\ot X\ar[r,tick,"\hid_{I_0}\ot\hid_X"]\ar[d,equal]\ar[dr,phantom,"\Two\eta"] \& I_0\ot X\ar[d,equal] \\
 I_0\ot X\ar[d,"\lambda"']\ar[r,tick,"\hid_{I_0\ot X}"]\ar[dr,phantom,"\Two\hid_\lambda"] \& I_0\ot X\ar[d,"\lambda"] \\
 X\ar[r,tick,"\hid_X"'] \& X
\end{tikzcd}=
\begin{tikzcd}[ampersand replacement=\&]
I_0\ot X\ar[r,tick,"I_1\ot\hid_X"]\ar[dr,phantom,"\Two\lambda_{\hid_X}"]\ar[d,"\lambda"'] \& I_0\ot X\ar[d,"\lambda"] \\
X\ar[r,tick,"\hid_X"'] \& X
\end{tikzcd}}
\end{equation}
\begin{displaymath}
\scalebox{0.8}{
\begin{tikzcd}[ampersand replacement=\&]
X\ot I_0\ar[rr,tick,"(N\circ M)\ot I_1"]\ar[d,equal]\ar[drr,phantom,"\Two\vid\ot\delta"] \&\& Z\ot I_0\ar[d,equal] \\
X\ot I_0\ar[rr,tick,"(N\circ M)\ot (I_1\circ I_1)"]\ar[d,equal]\ar[drr,phantom,"\Two\tau"] \&\& Z\ot I_0\ar[d,equal] \\
X\ot I_0\ar[r,tick,"M\ot I_1"]\ar[d,"\rho"']\ar[dr,phantom,"\Two \rho_M"] \& Y\ot I_0\ar[dr,phantom,"\Two \rho_N"]\ar[r,tick,"N\ot
I_1"]\ar[d,"\rho"'] \& Z\ot I_0 \ar[d,"\rho"] \\
X\ar[r,tick,"M"'] \& Y\ar[r,tick,"N"'] \& Z
 \end{tikzcd}=
 \begin{tikzcd}[ampersand replacement=\&]
X\ot I_0\ar[rr,tick,"(N\circ M)\ot I_1"]\ar[d,"\rho"']\ar[drr,phantom,"\Two \rho_{N\circ M}"] \&\& Z\ot I_0\ar[d,"\rho"] \\
 X\ar[r,tick,"M"'] \& Y\ar[r,tick,"N"'] \& Z
\end{tikzcd}\;\;
\begin{tikzcd}[ampersand replacement=\&]
X\ot I_0\ar[r,tick,"\hid_X\ot I_1"]\ar[d,equal]\ar[dr,phantom,"\Two\vid\ot\iota"] \& X\ot I_0\ar[d,equal] \\
X\ot I_0\ar[r,tick,"\hid_X\ot\hid_{I_0}"]\ar[d,equal]\ar[dr,phantom,"\Two\eta"] \& X\ot I_0\ar[d,equal] \\
X\ot I_0\ar[r,tick,"\hid_{X\ot I_0}"]\ar[d,"\rho"']\ar[dr,phantom,"\Two\hid_{\rho}"] \& X\ot I_0\ar[d,"\rho"] \\
X\ar[r,tick,"\hid_X"'] \& X
\end{tikzcd}=
\begin{tikzcd}[ampersand replacement=\&]
X\ot I_0\ar[d,"\rho"']\ar[r,tick,"\hid_X\ot I_1"]\ar[dr,phantom,"\Two \rho_{\hid_X}"] \& X\ot I_0\ar[d,"\rho"] \\
X\ar[r,tick,"\hid_X"'] \& X
\end{tikzcd}}
\end{displaymath}
which make the left unit constraint $\lambda\colon\ot\circ(I\times1)\Rightarrow1$ and the right unit constraint $\rho\colon\ot\circ(1\times I)\Rightarrow1$ into
vertical transformations of oplax double functors (together with the naturality of components coming from $(\dc{D}_1,\otimes_1,I_1)$ being a monoidal
category).

\subsection{Proof of \cref{thm:oplaxmonoidalKlT}}\label{sec:sample}
In this section, we illustrate the main ideas in the proof of \cref{thm:oplaxmonoidalKlT}. We begin by stating a couple of technical lemmas which, along with~\cref{lem:companioncomponents}, are used repeatedly in the proof.
\begin{lem}\label{lem:xitranspose}
Let $F\colon \dc{C}\to\dc{D}$ be a double functor. For any horizontal $1$-cell $M\colon X\horightarrow Y$ and vertical $1$-cell $f\colon Y\to Y'$ in
$\dc{C}$, a 2-morphism in $\dc{D}$ of the form on the left corresponds, under transpose
operations, to a 2-morphism of the form on the right
\begin{displaymath}
\begin{tikzcd}
FX\ar[rr,tick,"F(\comp{f}\circ M)"]\ar[d,equal]\ar[drr,phantom,"\Two\xi"] && FY'\ar[d,equal] \\
FX\ar[r,tick,"FM"]\ar[d]\ar[dr,phantom,"\Two\phi"] & FY\ar[r,tick,"F\comp{f}"]\ar[d]\ar[dr,phantom,"\Two\comp{\psi}"] & FY'\ar[d] \\
\bullet\ar[r,tick] & \bullet\ar[r,tick] & \bullet
\end{tikzcd}\qquad
\begin{tikzcd}
FX\ar[r,tick,"FM"]\ar[dd]\ar[ddr,phantom,"\Two\phi"] & FY\ar[r,tick,"\hid_{FY}"]\ar[ddr,phantom,"\Two\psi"]\ar[dd] & FY\ar[d,"Ff"] \\
&& FY'\ar[d] \\
\bullet\ar[r,tick] & \bullet\ar[r,tick] & \bullet
\end{tikzcd}
\end{displaymath}
for any 2-morphisms $\phi,\psi$ of the right shape.
\end{lem}
\begin{proof}
By pasting the 2-morphism $F(p_2\circ\vid_{M})$ on the top of the left diagram, we obtain
\begin{displaymath}
\begin{tikzcd}
FX\ar[rr,tick,"F(\hid_Y\circ M)"]\ar[d,equal]\ar[drr,phantom,"\Two F(p_2\circ\vid_M)"] && FY\ar[d,"Ff"] \\
FX\ar[rr,tick,"F(\comp{f}\circ M)"']\ar[d,equal]\ar[drr,phantom,"\Two\xi_{\comp{f},M}"] && FY'\ar[d,equal] \\
FX\ar[r,tick,"FM"']\ar[d]\ar[dr,phantom,"\Two\phi"] & FY\ar[d]\ar[dr,phantom,"\Two\comp{\psi}"]\ar[r,tick,"F\comp{f}"'] & FY'\ar[d] \\
\bullet\ar[r,tick] & \bullet\ar[r,tick] & \bullet
 \end{tikzcd}=
 \begin{tikzcd}
FX\ar[rr,tick,"F(\hid_Y\circ M)"]\ar[d,equal]\ar[drr,phantom,"\Two\xi_{\hid_Y,M}"] && FY\ar[d,equal] \\
FX\ar[r,tick,"FM"']\ar[d,equal]\ar[dr,phantom,"\Two\vid_{FM}"] & FY\ar[r,tick,"F(\hid_Y)"']\ar[d,equal]\ar[dr,phantom,"\Two Fp_2"] & FY\ar[d,"Ff"] \\
FX\ar[r,tick,"FM"']\ar[d]\ar[dr,phantom,"\Two\phi"] & FY\ar[d]\ar[r,tick,"F\comp{f}"']\ar[dr,phantom,"\Two\comp{\psi}"] & FY'\ar[d] \\
\bullet\ar[r,tick] & \bullet\ar[r,tick] & \bullet
 \end{tikzcd}=\begin{tikzcd}
FX\ar[r,tick,"FM"]\ar[dd]\ar[ddr,phantom,"\Two\phi"] & FY\ar[r,tick,"\hid_{FY}"]\ar[ddr,phantom,"\Two\psi"]\ar[dd] & FY\ar[d,"Ff"] \\
&& FY'\ar[d] \\
\bullet\ar[r,tick] & \bullet\ar[r,tick] & \bullet
\end{tikzcd}
\end{displaymath}
where the first equality is due to naturality of components of $\xi$, and the second one (up to pasting with appropriate coherence isomorphisms)
follows from the definition of the transpose $\comp{\psi}$ and the unitality axiom for $F$.
\end{proof}

\begin{lem}\label{lem:technical2}
 Let $\beta\colon F\ticktwoar G$ be a horizontal transformation and $f\colon X\to X'$ a vertical 1-cell in $\dc{C}$. If $f$ has a companion
$\comp{f}$, then the globular coherence $2$-isomorphism $\beta_{\comp{f}}$ is vertically inverse to the transpose of the $2$-morphism component
$\beta_f$, {\em i.e.}~:
\begin{displaymath}
\begin{tikzcd}
FX\ar[r,tick,"\hid_{FX}"]\ar[d,equal]\ar[dr,phantom,"\Two Fp_2"] & FX\ar[d,"Ff"']\ar[r,tick,"\beta_X"]\ar[dr,phantom,"\Two\beta_f"] &
GX\ar[r,tick,"G\comp{f}"]\ar[d,"Gf"]\ar[dr,phantom,"\Two Gp_1"] & GX'\ar[d,equal] \\
FX\ar[r,tick,"F\comp{f}"'] & FX'\ar[r,tick,"\beta_{X'}"'] & GX'\ar[r,tick,"\hid_{GX'}"'] & GX'
\end{tikzcd}
=
\begin{tikzcd}
FX\ar[r,tick,"\beta_X"]\ar[d,equal]\ar[drr,phantom,"\Two\beta_{\comp{f}}^{\mi1}"] & GX\ar[r,tick,"G\comp{f}"] &
GX'\ar[d,equal] \\
FX\ar[r,tick,"F\comp{f}"'] & FX'\ar[r,tick,"\beta_{X'}"'] & GX'
\end{tikzcd}
\end{displaymath}
\end{lem}

\begin{proof}
 We will show that if we vertically compose the transpose of $\beta_f$ with $\beta_{\comp{f}}$ on both sides, it produces a vertical identity (up
to coherence isomorphisms). Indeed,
 \begin{displaymath}
 \begin{tikzcd}
 FX\ar[r,tick,"\hid_{FX}"]\ar[d,equal] & FX\ar[d,equal]\ar[r,tick,"F\comp{f}"]\ar[drr,phantom,"\Two\beta_{\comp{f}}"] &
 FX'\ar[r,tick,"\beta_{X'}"] & GX'\ar[d,equal] \\
 FX\ar[r,tick,"\hid_{FX}"]\ar[d,equal]\ar[dr,phantom,"\Two Fp_2"] & FX\ar[d,"Ff"']\ar[r,tick,"\beta_{X}"]\ar[dr,phantom,"\Two\beta_f"] &
 GX\ar[d,"Gf"]\ar[r,tick,"G\comp{f}"]\ar[dr,phantom,"\Two Gp_1"] & GX'\ar[d,equal] \\
 FX\ar[r,tick,"F\comp{f}"'] & FX'\ar[r,tick,"\beta_{X'}"'] & GX'\ar[r,tick,"\hid_{GX'}"'] & GX'
 \end{tikzcd}\stackrel{\cref{eq:betanaturality}}{=}
 \begin{tikzcd}
 FX\ar[r,tick,"\hid_{FX}"]\ar[d,equal]\ar[dr,phantom,"\Two Fp_2"] & FX\ar[d,"Ff"']\ar[r,tick,"F\comp{f}"]\ar[dr,phantom,"\Two Fp_1"] &
 FX'\ar[r,tick,"\beta_{X'}"]\ar[d,equal]\ar[dr,phantom,"\Two\beta_{\vid_{X'}}"] & GX'\ar[d,equal] \\
 FX\ar[r,tick,"F\comp{f}"]\ar[d,equal] & FX'\ar[d,equal]\ar[r,tick,"\hid_{FX'}"]\ar[drr,phantom,"\Two\beta_{\hid_{X'}}"] &
 FX'\ar[r,tick,"\Two\beta_{X'}"] & GX'\ar[d,equal] \\
 FX\ar[r,tick,"F\comp{f}"'] & FX'\ar[r,tick,"\beta_{X'}"'] & GX'\ar[r,tick,"\hid_{GX'}"'] & GX'
 \end{tikzcd}\stackrel{\substack{(\ref{eq:betafunctor})\\(\ref{eq:piaxioms})}}{=}\vid_{\beta_{X'}\circ F\comp{f}}
 \end{displaymath}
\begin{displaymath}
 \begin{tikzcd}
FX\ar[r,tick,"\hid_{FX}"]\ar[d,equal]\ar[dr,phantom,"\Two Fp_2"] & FX\ar[d,"Ff"']\ar[dr,phantom,"\Two\beta_f"]\ar[r,tick,"\beta_X"] &
GX\ar[d,"Gf"]\ar[dr,phantom,"\Two Gp_1"]\ar[r,tick,"G\comp{f}"] & GX'\ar[d,equal] \\
FX\ar[r,tick,"F\comp{f}"']\ar[d,equal]\ar[drr,phantom,"\Two\beta_{\comp{f}}"] & FX'\ar[r,tick,"\beta_{X'}"'] &
GX'\ar[d,equal]\ar[r,tick,"\hid_{GX'}"'] & GX'\ar[d,equal] \\
FX\ar[r,tick,"\beta_{X}"'] & GX\ar[r,tick,"G\comp{f}"'] & GX'\ar[r,tick,"\hid_{GX'}"'] & GX
 \end{tikzcd}\stackrel{\cref{eq:betanaturality}}{=}
 \begin{tikzcd}
FX\ar[r,tick,"\hid_{FX}"]\ar[d,equal]\ar[drr,phantom,"\Two\beta_{\hid_{X}}"] & FX\ar[r,tick,"\beta_X"] &
GX\ar[r,tick,"G\comp{f}"]\ar[d,equal] & GX'\ar[d,equal] \\
FX'\ar[r,tick,"\beta_X"']\ar[d,equal]\ar[dr,phantom,"\Two\beta_{\vid_X}"] & GX\ar[r,tick,"\hid_{GX}"']\ar[d,equal]\ar[dr,phantom,"\Two Gp_2"] &
GX\ar[d,"Gf"]\ar[r,tick,"G\comp{f}"]\ar[dr,phantom,"\Two Gp_1"] & GX'\ar[d,equal] \\
FX\ar[r,tick,"\beta_{X}"'] & GX\ar[r,tick,"G\comp{f}"'] & GX'\ar[r,tick,"\hid_{GX'}"'] & GX\mathrlap{ .}
\end{tikzcd}\stackrel{\substack{(\ref{eq:betafunctor})\\(\ref{eq:piaxioms})}}{=}\vid_{G\comp{f}\circ\beta_X}\qedhere
\end{displaymath}
\end{proof}

We now provide a sample verification of one of the axioms needed in the proof of \cref{thm:oplaxmonoidalKlT}. We will show that the left unit constraint of the monoidal structure of a Kleisli double category $\Kl(T)$ for a monoidal
horizontal double monad is compatible with horizontal composition, namely the top left axiom of~\eqref{eq:sampleax}.

First of all, the left unit constraint
components $J\kot M\Rightarrow M$ as in \cref{eq:Klunitors} bijectively correspond to cells
\begin{displaymath}
\begin{tikzcd}
I\ot X\ar[ddd,"\lambda"']\ar[r,tick,"\hid\ot M"]\ar[dddr,phantom,"\Two\lambda_M"] & I\ot
TY\ar[dddr,phantom,"\Two\hid_\lambda"]\ar[ddd,"\lambda"']\ar[r,tick,"\hid"] & I\ot
TY\ar[d,"T^0\ot\vid"]
\\
&& TI\ot TY\ar[d,"T^2"] \\
&& T(I\ot Y)\ar[d,"T\lambda"] \\
X\ar[r,tick,"M"'] & TY\ar[r,tick,"\hid"'] & TY
\end{tikzcd}
\end{displaymath}
using \cref{lem:xitranspose}, since $\tau$ is the composition comparison structure map for the double functor $\ot$ of the monoidal double
category $\dc{C}$. Using appropriate transpose operations like \cref{eq:transpose}, we can therefore transform the axiom at question to one that does
not involve companions
of 1-cells and 2-morphisms as follows: the right-hand side of \cref{eq:sampleax} becomes
\begin{equation}\label{eq:thatone}
 \begin{tikzcd}[ampersand replacement=\&]
I\ot X\ar[ddd,"\lambda"']\ar[rrr,tick,"\hid\ot(m\circ TN\circ M)"]\ar[dddrrr,phantom,"\Two\lambda_{m\circ TN\circ M}"] \&\&\& I\ot
TZ\ar[ddd,"\lambda"]\ar[r,tick,"\hid"] \& I\ot TZ\ar[d,"T^0\ot\vid"] \\
\&\&\&\& TI\ot TZ\ar[d,"T^2"] \\
\&\&\&\& T(I\ot Z)\ar[d,"T\lambda"] \\
X\ar[r,tick,"M"'] \& TY\ar[r,tick,"TN"'] \& TTZ\ar[r,tick,"m_Z"'] \& TZ\ar[r,tick,"\hid"'] \& TZ
 \end{tikzcd}
 \end{equation}
 whereas the left-hand side, using naturality of $\tau$, becomes
 \begin{equation}\label{eq:thisone}
 \scalebox{0.8}{
 \begin{tikzcd}[ampersand replacement=\&]
I\ot X\ar[d,equal]\ar[rrrrrr,tick,"\hid\ot(m\circ TN\circ M)"]\ar[drrrrrr,phantom,"\Two\tau"] \&\&\&\&\&\& I\ot TZ\ar[d,equal] \\
I\ot X\ar[r,tick,"\hid\ot M"]\ar[ddddd,"\lambda"']\ar[dddddr,phantom,"\Two\lambda_M"] \& I\ot TY\ar[ddddd,"\lambda"']\ar[r,equal] \&
I\ot TY\ar[r,tick,"\hid\ot TN"]\ar[d,"T^0\ot\vid"'] \& I\ot TTZ\ar[d,"T^0\ot\vid"]\ar[r,equal] \& I\ot TTZ\ar[r,equal]\ar[d,"T^0\ot\vid"'] \& I\ot
TTZ\ar[r,tick,"\hid\ot m"]\ar[d,"T^0\ot\vid"']\ar[ddr,phantom,"\Two m^0\ot\vid"] \& I\ot TZ\ar[dd,"T^0\ot\vid"] \\
\&\& TI\ot TY\ar[d,"T^2"']\ar[r,tick,"\hid\ot TN"]\ar[dr,phantom,"\Two T^2_{1,N}"] \& TI\ot TTZ\ar[d,"T^2"] \& TI\ot TTZ\ar[d,"T^2"]\ar[r,equal]
\& TI\ot TTZ\ar[d,"TT^0\ot\vid"'] \& \\
\& \& T(I\ot Y)\ar[ddd,"T\lambda"']\ar[dddr,phantom,"\Two T\lambda_N"]\ar[r,tick,"T(\hid\ot N)"'] \& T(I\ot TZ)\ar[ddd,"T\lambda"]\ar[r,equal] \&
T(I\ot TZ)\ar[d,"T(T^0\ot\vid)"'] \& TTI\ot TTZ\ar[d,"T^2"']\ar[r,tick,"m\ot m"]\ar[ddr,phantom,"\Two m^2"] \& TI\ot TZ\ar[dd,"T^2"] \\
\&\&\&\& T(TI\ot TZ)\ar[r,equal]\ar[d,"TT^2"'] \& T(TI\ot TZ)\ar[d,"TT^2"'] \& \\
\&\&\&\& TT(I\ot Z)\ar[d,"TT\lambda"'] \& TT(I\ot Z)\ar[d,"TT\lambda"']\ar[dr,phantom,"\Two m_\lambda"]\ar[r,tick,"m"] \& T(I\ot Z)\ar[d,"T\lambda"]
\\
X\ar[r,tick,"M"'] \& TY\ar[r,equal] \& TY\ar[r,tick,"TN"'] \& TTZ\ar[r,equal] \& TTZ\ar[r,equal] \& TTZ\ar[r,tick,"m"'] \& TZ
 \end{tikzcd}}
\end{equation}
Now using the fact that $(T^2,T^0)$ and $(m^2,m^0)$ are the structure maps of lax monoidal functors, namely
\begin{displaymath}
\scalebox{0.8}
{\begin{tikzcd}[ampersand replacement=\&]
 I\ot TY\ar[dr,phantom,"\Two\hid_{T^0}\ot\vid"]\ar[r,tick,"\hid\ot TN"]\ar[d,"T^0\ot\vid"'] \& I\ot TTZ\ar[d,"T^0\ot\vid"] \\
 TI\ot TY\ar[d,"T^2"']\ar[r,tick,"\hid\ot TN"]\ar[dr,phantom,"\Two T^2_{1,N}"] \& TI\ot TTZ\ar[d,"T^2"] \\
 T(I\ot Y)\ar[ddd,"T\lambda"']\ar[dddr,phantom,"\Two T\lambda_N"]\ar[r,tick,"T(\hid\ot N)"'] \& T(I\ot TZ)\ar[ddd,"T\lambda"] \\
 \& \\
 \& \\
 TY\ar[r,tick,"TN"'] \& TTZ
\end{tikzcd}=
\begin{tikzcd}[ampersand replacement=\&]
I\ot TY\ar[r,tick,"\hid\ot TN"]\ar[d,"\lambda"']\ar[dr,phantom,"\Two\lambda_{TN}"] \& I\ot TTZ\ar[d,"\lambda"] \\
TY\ar[r,tick,"TN"'] \& TTZ
\end{tikzcd}\;\;
 \begin{tikzcd}[ampersand replacement=\&]
I\ot TTZ\ar[r,tick,"\hid\ot m"]\ar[d,"T^0\ot\vid"']\ar[ddr,phantom,"\Two m^0\ot\vid"] \& I\ot TZ\ar[dd,"T^0\ot\vid"] \\
TI\ot TTZ\ar[d,"TT^0\ot\vid"'] \& \\
TTI\ot TTZ\ar[d,"T^2"']\ar[r,tick,"m\ot m"]\ar[ddr,phantom,"\Two m^2"] \& TI\ot TZ\ar[dd,"T^2"] \\
T(TI\ot TZ)\ar[d,"TT^2"'] \& \\
TT(I\ot Z)\ar[d,"TT\lambda"']\ar[dr,phantom,"\Two m_\lambda"]\ar[r,tick,"m"] \& T(I\ot Z)\ar[d,"T\lambda"] \\
TTZ\ar[r,tick,"m"'] \& TZ
\end{tikzcd}=
\begin{tikzcd}[ampersand replacement=\&]
I\ot TTZ\ar[r,tick,"\hid\ot m"]\ar[dddd,"\lambda"']\ar[ddddr,phantom,"\Two\lambda_m"] \& I\ot TZ\ar[dddd,"\lambda"] \\
\hole \\
\hole \\
\hole \\
TTZ\ar[r,tick,"m"'] \& TZ
\end{tikzcd}}
\end{displaymath}
the diagram \cref{eq:thisone} reduces to
\begin{displaymath}
 \begin{tikzcd}
I\ot X\ar[d,equal]\ar[drrr,phantom,"\Two\tau"]\ar[rrr,tick,"\hid\ot(m\circ TN\circ M)"] &&& I\ot TZ\ar[d,equal] \\
I\ot X\ar[r,tick,"\hid\ot M"]\ar[d,"\lambda"']\ar[dr,phantom,"\Two\lambda_M"] & I\ot Y\ar[d,"\lambda"']\ar[r,tick,"\hid\ot
TN"]\ar[dr,phantom,"\Two\lambda_{TN}"] & I\ot TTZ\ar[d,"\lambda"']\ar[r,tick,"\hid\ot m"]\ar[dr,phantom,"\Two\lambda_m"] &
I\ot TZ\ar[d,"\lambda"] \\
X\ar[r,tick,"M"'] & TY\ar[r,tick,"TN"'] & TTZ\ar[r,tick,"m"'] & TZ
 \end{tikzcd}
\end{displaymath}
which is equal to \cref{eq:thatone}, since $\lambda$ is a vertical transformation of double functors \cref{eq:verticaltransfax}.

\bibliographystyle{alpha}
\bibliography{references}

\end{document}